\documentclass{article}

\usepackage{microtype}
\usepackage{graphicx}
\usepackage{subcaption}
\usepackage{booktabs}

\usepackage{hyperref}

\usepackage[accepted]{icml2024}

\usepackage{amsmath}
\usepackage{amssymb}
\usepackage{mathtools}
\usepackage{amsthm}
\usepackage{mathdots}

\usepackage[capitalize,noabbrev]{cleveref}

\theoremstyle{plain}
\newtheorem{theorem}{Theorem}[section]
\newtheorem{proposition}[theorem]{Proposition}
\newtheorem{lemma}[theorem]{Lemma}
\newtheorem{corollary}[theorem]{Corollary}
\theoremstyle{definition}

\theoremstyle{remark}

\usepackage[textsize=tiny]{todonotes}

\newcommand{\JA}{\opJ_{\opA}}

\newcommand{\cut}[1]{{}}

\newcommand{\va}{{\mathbf{a}}}
\newcommand{\vb}{{\mathbf{b}}}

\newcommand{\ve}{{\mathbf{e}}}

\newcommand{\vh}{{\mathbf{h}}}

\newcommand{\vv}{{\mathbf{v}}}

\newcommand{\vA}{{\mathbf{A}}}
\newcommand{\vB}{{\mathbf{B}}}

\newcommand{\vG}{{\mathbf{G}}}

\newcommand{\vL}{{\mathbf{L}}}

\newcommand{\vS}{{\mathbf{S}}}

\newcommand{\cA}{{\mathcal{A}}}

\newcommand{\cC}{{\mathcal{C}}}

\newcommand{\cF}{{\mathcal{F}}}

\newcommand{\cI}{{\mathcal{I}}}

\newcommand{\cN}{{\mathcal{N}}}
\newcommand{\cO}{{\mathcal{O}}}

\newcommand{\cQ}{{\mathcal{Q}}}

\newcommand{\cS}{{\mathcal{S}}}
\newcommand{\cT}{{\mathcal{T}}}

\usepackage[bb=boondox]{mathalfa}

\usepackage{xparse}
\DeclareFontFamily{U}{ntxmia}{}
\DeclareFontShape{U}{ntxmia}{m}{it}{<-> ntxmia }{}
\DeclareFontShape{U}{ntxmia}{b}{it}{<-> ntxbmia }{}
\DeclareSymbolFont{lettersA}{U}{ntxmia}{m}{it}
\SetSymbolFont{lettersA}{bold}{U}{ntxmia}{b}{it}

\ExplSyntaxOn
\NewDocumentCommand{\varmathbb}{m}
 {
  \tl_map_inline:nn { #1 }
   {
    \use:c { varbb##1 }
   }
 }
\tl_map_inline:nn { ABCDEFGHIJKLMNOPQRSTUVWXYZ }
 {
  \exp_args:Nc \DeclareMathSymbol{varbb#1}{\mathord}{lettersA}{\int_eval:n { `#1+67 }}
 }
\exp_args:Nc \DeclareMathSymbol{varbbk}{\mathord}{lettersA}{169}
\ExplSyntaxOff
\DeclareMathOperator*{\argmin}{argmin}

\newcommand*{\fix}{\mathrm{Fix}\,}
\newcommand*{\zer}{\mathrm{Zer}\,}
\newcommand*{\gra}{\mathrm{Gra}\,}

\newcommand{\dom}{\mathrm{dom}\,}

\newcommand{\reals}{\mathbb{R}}

\newcommand{\opA}{{\varmathbb{A}}}

\newcommand{\opI}{{\varmathbb{I}}}
\newcommand{\opJ}{{\varmathbb{J}}}

\newcommand{\opT}{{\varmathbb{T}}}

\newcommand{\inprod}[2]{\left\langle #1,#2 \right\rangle}
\newcommand{\sqnorm}[1]{\left\| #1 \right\|^2}
\newcommand{\tilA}{\Tilde{\opA}}

\newcommand{\topa}{\tilde{\opA}}
\newcommand{\norm}[1]{\left\|#1\right\|}

\newcommand{\inner}[2]{\left\langle #1 , #2 \right\rangle}

\newcommand{\Z}{X}
\newcommand{\pr}[1]{\left( #1 \right)}
\newcommand{\pmat}[1]{\begin{pmatrix} #1 \end{pmatrix}}
\newcommand{\set}[1]{\left\{ #1 \right\}}
\newcommand{\yap}{\delta}
\newcommand{\OurSpace}{\reals^d}

\newcommand{\sop}[1]{\nabla_{\pm} \mathbf{L} (#1)}
\newcommand{\at}{\text{A}}
\newcommand{\half}{1/2}
\newcommand{\hf}{\half}

\icmltitlerunning{Optimal Acceleration for Minimax and Fixed-Point Problems is Not Unique}

\begin{document}

\twocolumn[
\icmltitle{
Optimal Acceleration for Minimax and Fixed-Point Problems is Not Unique
}

\icmlsetsymbol{equal}{*}

\begin{icmlauthorlist}
\icmlauthor{TaeHo Yoon}{snumath}
\icmlauthor{Jaeyeon Kim}{snumath}
\icmlauthor{Jaewook J.\ Suh}{snumath}
\icmlauthor{Ernest K.\ Ryu}{snumath}
\end{icmlauthorlist}

\icmlaffiliation{snumath}{Department of Mathematical Sciences, Seoul National University}

\icmlcorrespondingauthor{Ernest Ryu}{ernestryu@snu.ac.kr}

\icmlkeywords{Optimization, Minimax optimization, Fixed-point problems, Acceleration}

\vskip 0.3in
]

\printAffiliationsAndNotice{} 

\begin{abstract}
Recently, accelerated algorithms using the anchoring mechanism for minimax optimization and fixed-point problems have been proposed, and matching complexity lower bounds establish their optimality. In this work, we present the surprising observation that the optimal acceleration mechanism in minimax optimization and fixed-point problems is not unique. Our new algorithms achieve exactly the same worst-case convergence rates as existing anchor-based methods while using materially different acceleration mechanisms. Specifically, these new algorithms are dual to the prior anchor-based accelerated methods in the sense of H-duality. This finding opens a new avenue of research on accelerated algorithms since we now have a family of methods that empirically exhibit varied characteristics while having the same optimal worst-case guarantee.
\end{abstract}

\section{Introduction}
Accelerated algorithms using the so-called anchoring mechanism have been recently proposed for solving minimax optimization and fixed-point problems. Furthermore, these algorithms are optimal: for minimax problems, the gap between the upper and lower bounds is a constant factor of $16$, and for fixed-point problems, there is no gap, not even a constant factor. Therefore, anchoring was thought to be ``the'' correct acceleration mechanism for these setups.

In this work, however, we present the surprising observation that the optimal acceleration mechanism in minimax optimization and fixed-point problems is not unique. For minimax optimization, we introduce a new algorithm with the same worst-case rate as the best-known algorithm. For fixed-point problems, we introduce a continuous family of exact optimal algorithms, all achieving the same worst-case rate that exactly matches the known lower bound. 
The representative cases of our new accelerated algorithms are dual algorithms of the prior anchor-based accelerated algorithms in the sense of H-duality. The resulting new acceleration mechanisms are materially different from the existing anchoring mechanism.

These findings show that anchor-based acceleration is not unique and sufficient as the mechanism of achieving the exact optimal complexity and enable us to correctly reframe the study of optimal acceleration as a study of a \emph{family} of acceleration mechanisms rather than a singular one. This shift in perspective will likely be critical in the future research toward a more complete and fundamental understanding of accelerations in fixed-point and minimax problems.

\vspace{-.2cm}

\subsection{Preliminaries}

We use standard notation for set-valued operators \cite{bauschke2011convex,RyuYin2022_largescale}.
An operator $\opA\colon \reals^d \rightrightarrows \reals^d$ is a set-valued function (so $\opA (x) \subseteq \reals^d$ for $x\in\reals^d$).
For simplicity, we write $\opA x = \opA (x)$.
The graph of $\opA$ is defined and denoted as $\gra \opA = \{(x,y) \,|\, x \in \reals^d, y \in \opA x\}$.
The inverse of $\opA$ is defined by $\opA^{-1} y = \{x\in\reals^d \,|\, y \in \opA x \}$.
Scalar multiples and sums of operators are defined in the Minkowski sense. 
If $\opT x$ is a singleton for all $x\in\reals^d$, we write $\opT\colon \reals^d \to \reals^d$ and treat it as a function.
An operator $\opT\colon \reals^d \to \reals^d$ is $L$-Lipschitz ($L > 0$) if $\norm{\opT x- \opT y} \le L \norm{ x -  y}$ for all $x,y \in \reals^d$.
We say $\opT$ is nonexpansive if it is $1$-Lipschitz.

A function $\vL\colon \reals^n \times \reals^m \to \reals$ is convex-concave if $\vL(u, v)$ is convex in $u$ for all fixed $v \in \reals^m$ and concave in $v$ for all fixed $u \in \reals^n$.
If $\vL(u_\star, v) \le \vL(u_\star, v_\star) \le \vL(u, v_\star)$ for all $(u,v) \in \reals^n \times \reals^m$, then $(u_\star, v_\star)$ is a saddle point of $\vL$.
For $L > 0$, if $\vL$ is differentiable and $\nabla \vL$ is $L$-Lipschitz on $\reals^n \times \reals^m$, we say $\vL$ is $L$-smooth.
In this case, we define the saddle operator of $\vL$ by $\sop{u,v} = \left( \nabla_u \vL(u, v), - \nabla_v \vL(u, v) \right)$.
In most of the cases, we use the joint variable notation $x = (u, v)$ and concisely write $\sop{x}$ in place of $\sop{u,v}$.

\vspace{-.2cm}

\subsection{Related work}
Here, we quickly review the most closely related prior work while deferring the more comprehensive literature survey to \cref{section:appendix-related-work}.

\vspace{-.2cm}

\paragraph{Fixed-point algorithms.}
A fixed-point problem solves
\begin{align}
\label{def:fixed-point-problem}
\begin{array}{cc}
    \underset{y\in \reals^d}{\text{find}} & y = \opT y 
\end{array}
\end{align}
for $\opT \colon \mathbb{R}^d \to \mathbb{R}^d$. 
The magnitude of the fixed-point residual $y_k - \opT y_k$ is one natural performance measure. \citet{SabachShtern2017_first} first achieved the rate $\sqnorm{y_k - \opT y_k} = \mathcal{O}(1/k^2)$ through the Sequential Averaging Method, and
\citet{Lieder2021_convergence} showed that Halpern iteration with specific parameters, which we call \ref{alg:halpern} in \cref{subsection:summary-fixed-point}, improves upon the rate of \citet{SabachShtern2017_first} by a factor of $16$.
Furthermore, \citet{ParkRyu2022_exact} provided a matching complexity lower bound showing that the rate of \citet{Lieder2021_convergence} is exactly optimal.

\vspace{-.1in}

\paragraph{Minimax algorithms.}
Minimax optimization solves
\begin{align}
\label{def:minimax-problem}
\begin{array}{cc}
    \underset{u \in \reals^n}{\text{minimize}} \,\, \underset{v \in \reals^m}{\text{maximize}} &  \vL (u, v)
\end{array}
\end{align}
for $\vL\colon\reals^n \times \reals^m \to \reals$.
Under the assumption of convex-concavity, $\sqnorm{\nabla \vL(u_k, v_k)}$ is one natural performance measure. 
\citet{YoonRyu2021_accelerated} first provided the (order-optimal) accelerated $\sqnorm{\nabla \vL(u_k, v_k)} = \cO(1/k^2)$ rate via the Extra Anchored Gradient (EAG) algorithm together with $\Omega(1/k^2)$ complexity lower bound. The Fast Extragradient (FEG) algorithm of \citep{LeeKim2021_fast} then improved this rate by a constant factor, achieving the currently best-known constant.

\vspace{-.1in}

\paragraph{Duality of algorithms.}
H-duality \cite{KimOzdaglarParkRyu2023_timereversed, KimParkOzdaglarDiakonikolasRyu2023_mirror} is a duality correspondence between algorithms. 
The H-duality theory of \citet{KimOzdaglarParkRyu2023_timereversed} shows that in smooth convex minimization, an algorithm's rate with respect to function value translates to the rate of its H-dual algorithm with respect to gradient norm and vice versa. 
Our paper establishes a different H-duality theory for fixed-point algorithms.

\vspace{-.1cm}

\subsection{Contribution and organization}
\paragraph{Contributions.}
This work is presenting a new class of accelerations in fixed-point and minimax problems. Our findings provide the perspective that the study of optimal acceleration must be viewed as a study of a \emph{family} of acceleration mechanisms rather than a singular one.

\vspace{-.1in}

\paragraph{Organization.}
\cref{section:summary} provides an overview of the novel algorithms \ref{alg:dual_halpern} and \ref{alg:dual-feg} and their continuous-time model.
\cref{section:Dual-Halpern} presents the analysis of \ref{alg:dual_halpern}.
\cref{section:family-of-algorithms} presents an infinite family of fixed-point algorithms achieving the same exact optimal rates. 
\cref{section:H-duality} presents the H-duality for fixed-point problems, which explains the symmetry and connection underlying \ref{alg:halpern} and \ref{alg:dual_halpern}.
\cref{section:minimax} provides the analysis of~\ref{alg:dual-feg} and its H-dual relationship with~\ref{alg:feg}. 
\cref{section:continuous-time} explores a continuous-time perspective of the new algorithms.
\cref{section:experiments} provides numerical simulations.

\vspace{-.1cm}

\section{Summary of new acceleration results}
\label{section:summary}

In this section, we provide an overview of novel accelerated algorithms for several setups.
For each setup, we first review the existing (primal) algorithm using the anchor acceleration mechanism and then show its dual counterpart with identical rates but using a materially different acceleration mechanism. (The meaning ``dual'' is clarified later.)
Throughout the paper, we write  $N\ge 1$ to denote the pre-specified iteration count of the algorithm.

\subsection{Fixed-point problems}
\label{subsection:summary-fixed-point}
Consider the fixed-point problem \eqref{def:fixed-point-problem}, where $\opT \colon \reals^d \to \reals^d$ is nonexpansive.
We denote $\fix\opT = \{y \in \reals^d \,|\, y = \opT y\}$ and assume $\fix\opT\ne \emptyset$.

The (primal) Optimal Halpern Method (\ref{alg:halpern})\footnote{Some prior work referred to this method as the ``Optimized'' Halpern Method, but we now know the method is (exactly) optimal as \citep{ParkRyu2022_exact} provided a matching lower bound.} \cite{Halpern1967_fixed, Lieder2021_convergence} is
\begin{align} \label{alg:halpern} \tag{OHM}
    y_{k+1} = \frac{k+1}{k+2} \opT y_k + \frac{1}{k+2} y_0 % \tag{Halpern}
\end{align}
for $k=0,1,\dots$. Equivalently, we can write
\begin{align*}
    y_{k+1}
    &= y_k - \frac{1}{k+2} (y_k - \opT y_k) + \frac{k}{k+2} \left( \opT y_k - \opT y_{k-1} \right)
\end{align*}
where we define $\opT y_{-1} = y_0$.
\ref{alg:halpern} exhibits the rate
\begin{align*}
    \|y_{k-1} - \opT y_{k-1}\|^2 \le \frac{4\sqnorm{y_0 - y_\star}}{k^2}
\end{align*}
for $k=1,2,\dots$ and $y_\star \in \fix\opT$ \cite{Lieder2021_convergence}.

We present the new method,
Dual Optimal Halpern Method (\ref{alg:dual_halpern}):
\begin{align} \label{alg:dual_halpern} \tag{Dual-OHM}
    y_{k+1} = y_k + \frac{N-k-1}{N-k} \left( \opT y_k - \opT y_{k-1} \right) % \tag{Dual-Halpern}
\end{align}
for $k=0,1,\dots,N-2$, where we define $\opT y_{-1} = y_0$.
Equivalently,
\begin{align}
\begin{aligned}
    z_{k+1} & = \frac{N-k-1}{N-k} z_k - \frac{1}{N-k} \left( y_k - \opT y_k \right) \\
    y_{k+1} & = \opT y_k - z_{k+1} 
\end{aligned}
\label{alg:dual_halpern_with_z}
\end{align}
for $k=0,1,\dots,N-2$, where $z_0 = 0$.
\ref{alg:dual_halpern} exhibits the rate 
\begin{align*}
    \|y_{N-1} - \opT y_{N-1}\|^2 \le \frac{4\sqnorm{y_0 - y_\star}}{N^2} 
\end{align*}
for $y_\star \in \fix\opT$.
This rate exactly coincides with the rate of
\ref{alg:halpern} for $k=N$.
We discuss the detailed analysis in~\cref{section:Dual-Halpern}.

As shown in \citep{ParkRyu2022_exact}, there exists a nonexpansive operator $\opT\colon \reals^d \to \reals^d$ with $d\ge 2N-2$ such that
\begin{align*}
    \sqnorm{y_{N-1} - \opT y_{N-1}} \ge \frac{4\sqnorm{y_0 - y_\star}}{N^2} 
\end{align*}
for any deterministic algorithm using $N-1$ evaluations of $ \opT$.
Therefore, \ref{alg:halpern} is exactly optimal; it cannot be improved, not even by a constant factor, in terms of worst-case performance.
The discovery of~\ref{alg:dual_halpern} is surprising as it shows that the exact optimal algorithm is not unique.

\subsection{Smooth convex-concave minimax optimization}
\label{subsection:summary-smooth-minimax}
Consider the minimax optimization problem \eqref{def:minimax-problem}, 
where $\vL\colon  \reals^n \times \reals^m \to \reals$ is convex-concave and $L$-smooth.
Convex-concave minimax problems are closely related to fixed-point problems, and the anchoring mechanism of \ref{alg:halpern} for accelerating fixed-point algorithms has been used to accelerate algorithms for minimax problems \cite{YoonRyu2021_accelerated, LeeKim2021_fast}.
We show that \ref{alg:dual_halpern} also has its minimax counterpart.
In the following, write $\opA = \nabla_\pm \vL$ for notational conciseness.

The (primal) Fast Extragradient \footnote{\ref{alg:feg} was designed primarily for weakly nonconvex-nonconcave problems, but we consider its application to the special case of convex-concave problems.} (\ref{alg:feg}) \citep{LeeKim2021_fast} is
\begin{align}
\begin{split}
    x_{k+\hf} & = x_k + \frac{1}{k+1} (x_0 - x_k) - \frac{k}{k+1} \alpha \opA x_k \\
    x_{k+1} & = x_k + \frac{1}{k+1} (x_0 - x_k) - \alpha \opA x_{k+\hf}
    % x_{k+\hf} & = x_k + \frac{1}{k+1} (x_0 - x_k) - \frac{k}{k+1} \alpha \sop{x_k}\\
    % x_{k+1} & = x_k + \frac{1}{k+1} (x_0 - x_k) - \alpha \sop{x_{k+\hf}}
\end{split}
\label{alg:feg} \tag{FEG}
\end{align}
for $k=0,1,\dots$. 
If $0 < \alpha \le \frac{1}{L}$, \ref{alg:feg} exhibits the rate
\begin{align*}
    \sqnorm{\nabla\vL(x_k)} = \sqnorm{\opA x_k} \le \frac{4\sqnorm{x_0 - x_\star}}{\alpha^2 k^2}
\end{align*}
for $k=1,2,\dots$ and a saddle point (solution) $x_\star = (u_\star, v_\star)$. 
To the best of our knowledge, this result with $\alpha=\frac{1}{L}$ is the fastest known rate.

We present the new method, Dual Fast Extragradient (\ref{alg:dual-feg}):
\begin{align}
% \begin{split}
    & x_{k+\half}  = x_k - \alpha z_k - \alpha \opA x_k \nonumber\\
    & x_{k+1}  = x_{k+1/2} - \frac{N-k-1}{N-k}\alpha\left(\opA x_{k+1/2} - \opA x_k \right) \nonumber\\
    & z_{k+1}  = \frac{N-k-1}{N-k} z_k - \frac{1}{N-k} \opA x_{k+\hf}
\label{alg:dual-feg}\tag{Dual-FEG}
\end{align}
for $k=0,1,\dots,N-1$, where $z_0 = 0$.
For $0 < \alpha \le \frac{1}{L}$, \ref{alg:dual-feg} exhibits the rate
\begin{align*}
    \sqnorm{\nabla\vL(x_N)} = \sqnorm{\opA x_N} \le \frac{4\sqnorm{x_0 - x_\star}}{\alpha^2 N^2}. 
\end{align*}
This rate exactly coincides with the rate of \ref{alg:feg} for $k=N$.
We discuss the detailed analysis in Section~\ref{section:minimax}. 

As shown in~\citep{YoonRyu2021_accelerated}, there exists $L$-smooth convex-concave $\vL\colon \reals^n \times \reals^n \to \reals$ with $m=n\ge 3N+2$ such that
\begin{align*}
    \sqnorm{\nabla\vL(x_N)} \ge \frac{L^2 \sqnorm{x_0 - x_\star}}{\left(2\lfloor N/2 \rfloor + 1\right)^2}
\end{align*}
for any deterministic $N$-step first-order algorithm.
Therefore, \ref{alg:feg} and \ref{alg:dual-feg} are optimal up to a constant factor.

\subsection{Continuous-time analysis}
\label{subsection:summary-continuous}

This section introduces continuous-time analyses corresponding to the algorithms of Sections~\ref{subsection:summary-fixed-point} and \ref{subsection:summary-smooth-minimax}.
In continuous-time limits, the algorithms for fixed-point problems and minimax optimization reduce to the same continuous-time ODE.
Let $\opA = 2(\opI + \opT)^{-1} - \opI$ for fixed-point problem~\eqref{def:fixed-point-problem} and $\opA = \nabla_\pm \vL$ for minimax problem~\eqref{def:minimax-problem}.

The (primal) Anchor ODE \citep{RyuYuanYin2019_ode, SuhParkRyu2023_continuoustime} is
\begin{align}
\label{ode:anchor}
    \dot{X}(t) = -\opA (X(t)) + \frac{1}{t} (X_0 - X(t)) 
\end{align}
which has an equivalent 2nd-order form
\[
    \ddot{X}(t) + \frac{2}{t} \dot{X} + \frac{1}{t} \opA(X(t)) + \frac{d}{dt} \opA(X(t)) = 0
\]
where $X(0) = X_0$ (and $\dot{X}(0) = -\frac{1}{2}\opA(X_0)$ for 2nd-order form) is the initial condition.
Anchor ODE exhibits the rate
\begin{align*}
    \sqnorm{\opA (X(t))} \le \frac{4 \sqnorm{X_0 - X_\star}}{t^2}
\end{align*}
for $t>0$, where $X_\star$ is a solution (zero of $\opA$).

We present the new Dual-Anchor ODE:
\begin{align}
\begin{aligned}
    \dot{X}(t) & = - Z(t) - \opA(X(t)) \\
    \dot{Z}(t) 
    & = - \frac{1}{T-t} Z(t) - \frac{1}{T-t} \opA(X(t)) ,
\end{aligned}
\label{ode:dual-anchor}
\end{align}
which has an equivalent 2nd-order form
\begin{align*} 
    \ddot{X}(t) + \frac{1}{T-t} \dot{X}(t) + \frac{d}{dt} \opA (X(t)) = 0
\end{align*}
for $t \in (0, T)$, where $T > 0$ is a pre-specified terminal time, and $X(0)=X_0$ and $Z(0)=0$ (or $\dot{X}(0) = -\opA(X_0)$ for the 2nd-order form) are initial conditions.
Dual-Anchor ODE exhibits the rate
\begin{align*}
    \sqnorm{\opA (X(T))} \le \frac{4 \sqnorm{X_0 - X_\star}}{T^2}.
\end{align*}
This rate exactly coincides with the rate of Anchor ODE for $t=T$.
We discuss the detailed analysis in Section~\ref{section:continuous-time}. 

\paragraph{``Dual'' in the sense of H-duality.} Our new algorithms are ``dual'' to the known primal algorithms in the sense of H-duality, a recently developed notion of duality between convex minimization algorithms \cite{KimOzdaglarParkRyu2023_timereversed}. We provide the detailed discussion of H-duality for fixed-point algorithms in \cref{section:H-duality}.

\section{Analysis of \ref{alg:dual_halpern}}
\label{section:Dual-Halpern}

In this section, we present the convergence analysis of \ref{alg:dual_halpern}, showing that it is another exact optimal algorithm for solving nonexpansive fixed-point problems.

\setcounter{subsection}{-1}
\subsection{Preliminaries: Monotone operators}

In the following, we express our analysis using the language of monotone operators. We quickly set up the notation and review the connections between fixed-point problems and monotone operators. 

\paragraph{Monotone operators.}
A set-valued (non-linear) operator $\opA\colon \reals^d \rightrightarrows \reals^d$ is monotone if $\inprod{g - g'}{x - x'} \ge 0$ for all $x,x' \in \reals^d$, $g \in \opA x$, and $g' \in \opA x'$.
If $\opA$ is monotone and there is no monotone operator $\opA'$ for which $\gra\opA \subset \gra\opA'$ properly, then $\opA$ is maximally monotone.
If $\opA$ is maximally monotone, then its resolvent $\JA := (\opI + \opA)^{-1}\colon \reals^d \to \reals^d$ is a well-defined single-valued operator. 

\paragraph{Fixed-point problems are monotone inclusion problems.}
There exists a natural correspondence between the classes of nonexpansive operators and maximally monotone operators in the following sense.
\begin{proposition} 
\label{proposition:monotone-nonexpansive-correspondence}
\citep[Theorem~2]{EcksteinBertsekas1992_Douglas} 
If $\opT\colon \reals^d \to \reals^d$ is a nonexpansive operator, then $\opA = 2(\opI + \opT)^{-1} - \opI$ is maximally monotone.
Conversely, if $\opA\colon \reals^d \rightrightarrows \reals^d$ is maximally monotone, then $\opT = 2\JA - \opI$ is nonexpansive.
\end{proposition}
When $\opT = 2\JA - \opI$, we have $x = \opT x \iff 0 \in \opA x$,
i.e., $\fix\opT = \zer\opA := \{x \in \reals^d \,|\, 0 \in \opA x\}$.
Therefore, \cref{proposition:monotone-nonexpansive-correspondence} induces a one-to-one correspondence between nonexpansive fixed-point problems and monotone inclusion problems.

\paragraph{Fixed-point residual norm and operator norm.}
Given $y \in \reals^d$, its accuracy as an approximate fixed-point solution is often measured by $\|y - \opT y\|$, the norm of fixed-point residual.
Let $\opA$ be the maximal monotone operator satisfying $\opT = 2\JA - \opI$, and let $x = \JA y$.
Then we see that
\begin{align*}
    y \in (\opI + \opA) (x) = x + \opA x \iff y - x \in \opA x .
\end{align*}
Denote $\tilA x = y - x \in \opA x$.
Then
\begin{align*}
    y - \opT y = y - (2\JA - \opI) (y) = 2 (y - \JA y) = 2 \tilA x.
\end{align*}
Therefore, $\|y - \opT y\|=2\|\tilA x\|$.

\paragraph{Minimax optimization and monotone operators.}
The minimax problem \eqref{def:minimax-problem} can also be recast as a monotone inclusion problem.
Precisely, for $L$-smooth convex-concave $\vL$, its saddle operator $\nabla_\pm \vL$ is monotone and $L$-Lipschitz.
In this case, $x_\star = (u_\star, v_\star)$ is a minimax solution for $\vL$ if and only if $\sop{x_\star} = 0$.

Finally, we quickly state a handy lemma used in the convergence analyses throughout the paper.
\begin{lemma}
\label{lemma:convergence-proof-last-step}
Let $\opA\colon \reals^d \rightrightarrows \reals^d$ be monotone and let $x,y \in \reals^d$, $\tilA x \in \opA x$.
Suppose, for some $\rho > 0$,
\begin{align*}
    \rho \sqnorm{\tilA x} + \inprod{\tilA x}{x - y} \le 0 
\end{align*}
holds. 
Then, for $x_\star \in \zer\opA$, $\sqnorm{\tilA x} \le \frac{\sqnorm{y - x_\star}}{\rho^2}$.
\end{lemma}

\begin{proof}
By monotonicity of $\opA$ and Young's inequality,
\begin{align*}
    0 &\geq \rho \sqnorm{\tilA x} + \inprod{\tilA x}{x - y} \\
    &\geq \rho\sqnorm{\tilA x} + \inprod{\tilA x}{x_\star - y} \\ % (\because \text{Monotonicity of $\opA$}) \\
    &\geq \frac{\rho}{2} \norm{\tilA x}^2 - \frac{1}{2\rho}\norm{x_\star-y}^2. % (\because \text{Young's Inequality})
    \vspace{-\baselineskip} \qedhere
\end{align*}
\end{proof}

\subsection{Convergence analyses of \ref{alg:halpern} and \ref{alg:dual_halpern}}
\label{subsection:family-Halpern}

We formally state the convergence result of \ref{alg:dual_halpern} and outline its proof.

\begin{theorem}
Let $\opT\colon \reals^d \to \reals^d$ be nonexpansive and $y_\star \in \fix\opT$. 
For $N \ge 1$, \ref{alg:dual_halpern} exhibits the rate
\begin{align*}
    \sqnorm{y_{N-1} - \opT y_{N-1}} \le \frac{4\|y_0 - y_\star\|^2}{N^2} .
\end{align*}
\end{theorem}

\begin{proof}[Proof outline]
Let $\opA$ be the unique maximally monotone operator such that $\opT = 2\JA - \opI$ (defined as in~\cref{proposition:monotone-nonexpansive-correspondence}).
Let $x_{k+1} = \JA (y_k)$ for $k=0,1,\dots$, so that $\tilA x_{k+1} = y_k - x_{k+1} \in \opA x_{k+1}$.
Recall the alternative form~\eqref{alg:dual_halpern_with_z} of \ref{alg:dual_halpern}.
Define
\begin{align*}
    V_k & = -\frac{N-k-1}{N-k} \sqnorm{z_k + 2 \tilA x_N} \\
    & \quad \quad + \frac{2}{N-k} \inprod{z_k + 2 \tilA x_N}{y_k - y_{N-1}}
\end{align*}
for $k=0,1,\dots,N-1$.
We show in \cref{section:appendix-Dual-Halpern-Lyapunov} that
\begin{align*}
    & V_k - V_{k+1} \\
    & = \frac{4}{(N-k)(N-k-1)} \inprod{x_N - x_{k+1}}{\tilA x_N - \tilA x_{k+1}} ,
\end{align*}
i.e., $V_k \ge V_{k+1}$ for $k=0,1,\dots,N-2$.
Observe that $V_{N-1}=0$ and because $z_0 = 0$,
\begin{align*}
    V_0 & = -\frac{4 (N-1)}{N} \sqnorm{\tilA x_N} + \frac{4}{N} \inprod{\tilA x_N}{y_0 - y_{N-1}} \\
    & = -4 \sqnorm{\tilA x_N} + \frac{4}{N} \inprod{\tilA x_N}{y_0 - x_N} 
\end{align*}
where the second line uses $x_N = y_{N-1} - \tilA x_N$.
Finally, divide both sides of $V_0 \ge \cdots \ge V_{N-1} = 0$ by $\frac{4}{N}$, apply \cref{lemma:convergence-proof-last-step} and the identity $y_{N-1} - \opT y_{N-1} = 2 \tilA x_N$:
\begin{align*}
    \sqnorm{y_{N-1} - \opT y_{N-1}} = 4 \sqnorm{\tilA x_N} \le \frac{4\|y_0 - y_\star\|^2}{N^2} .
\end{align*}
\vspace{-\baselineskip} \qedhere
\end{proof}

\vspace{.05cm}

We point out that the convergence analysis for \ref{alg:halpern} can be done in a similar style \cite{Lieder2021_convergence}.
Define $\opA$, $x_{k+1}$ and $\tilA x_{k+1}$ as above.
Define $U_0 = 0$ and
\begin{align*}
    U_k = k^2 \sqnorm{\tilA x_k} + k \inprod{\tilA x_k}{x_k - y_0}
\end{align*}
for $k=1,2,\dots$.
It can be shown that \cite{RyuYin2022_largescale} 
\begin{align*}
    U_j - U_{j+1} = j(j+1) \inprod{x_{j+1} - x_j}{\tilA{x_{j+1}} - \tilA{x_j}} \ge 0 
\end{align*}
for $j=0,1,\dots$.
Then $0 = U_0 \ge \cdots \ge U_k$, and dividing both sides by $k$ and applying~\cref{lemma:convergence-proof-last-step} gives the rate $\|\tilA x_k\|^2 \le \frac{\sqnorm{y_0 - x_\star}}{k^2}$.

\subsection{Proximal forms of \ref{alg:halpern} and \ref{alg:dual_halpern} }
\label{subsection:proximal-forms}

It is known that \ref{alg:halpern} can be equivalently written as
\begin{align*}
    x_{k+1} & = \JA (y_k) \\
    y_{k+1} & = x_{k+1} + \frac{k}{k+2} (x_{k+1} - x_k) - \frac{k}{k+2} (x_k - y_{k-1})
\end{align*}
for $k=0,1,\dots$, where $x_0 = y_0$. This proximal form is called Accelerated Proximal Point Method (APPM) \cite{Kim2021_accelerated}.
Likewise, \ref{alg:dual_halpern} can be equivalently written as
\begin{align*}
    x_{k+1} & = \JA (y_k) \\
    y_{k+1} & = x_{k+1} + \frac{N-k-1}{N-k} (x_{k+1} - x_k) \\
    & \quad - \frac{N-k-1}{N-k} (x_k - y_{k-1}) - \frac{1}{N-k} (x_{k+1} - y_k) 
\end{align*}
for $k=0,1,\dots,N-2$, where $x_{-1} = y_0 = x_0$.
We prove the equivalence in \cref{section:appendix-algorithm-equivalence}.

\section{Continuous family of exact optimal fixed-point algorithms}
\label{section:family-of-algorithms}

Upon seeing the two algorithms \ref{alg:halpern} and \ref{alg:dual_halpern} exhibiting the same exact optimal rate, it is natural to ask whether there are other exact optimal algorithms. In this section, we show that there is, in fact, an $(N-2)$-dimensional continuous family of exact optimal algorithms.

\subsection{H-matrix representation}
Fixed-point algorithm of $N-1$ iterations with fixed (non-adaptive) step-sizes can be written in the form 
\begin{align} \label{eqn::FPI_H_matrix}
\begin{split}
    y_{k+1} & = y_k  - \sum_{j=0}^{k} h_{k+1,j+1} \underbrace{(y_j - \opT y_j)}_{2\tilA x_{j+1}}
\end{split}
\end{align}
for $k=0,1,\dots,N-2$, where the $\Tilde{\opA}$ notation uses the convention of Section~\ref{section:Dual-Halpern}.
With this representation, the lower-triangular matrix $H\in \mathbb{R}^{(N-1)\times (N-1)}$, defined by $(H)_{k,j} = h_{k,j}$ if $j\leq k$ and $(H)_{k,j}=0$ otherwise, fully specifies the algorithm.

\subsection{Optimal algorithm family via H-matrices}
\label{subsection:family-characterization-via-H-matrices}

We now state our result characterizing a family of exact optimal algorithms.

\begin{theorem}[Optimal family]
\label{theorem:optimal-method-family}
There exist a nonempty open convex set $C \subset \mathbb{R}^{N-2}$ and a continuous injective mapping $\Phi\colon C \to \mathbb{R}^{(N-1)\times (N-1)}$ such that for all $w\in C$ $H=\Phi(w)$ is lower-triangular, and algorithm~\eqref{eqn::FPI_H_matrix} defined via the H-matrix $H=\Phi(w)$ exhibits the exact optimal rate (matching the lower bound)
\begin{align*}
    \sqnorm{y_{N-1} - \opT y_{N-1}} = 4 \sqnorm{\tilA x_N} \leq \frac{4 \norm{y_0-y_\star}^2}{N^2}  
\end{align*}
for nonexpansive  $\opT\colon \reals^d \to \reals^d$ and $y_\star \in \fix\opT$.
\end{theorem}  
We briefly outline the high-level idea of the proof while deferring the details to Appendix~\ref{section:appendix-optimal-family-theorem}.
From \cref{subsection:family-Halpern}, we observe that the convergence proofs for \ref{alg:halpern} and \ref{alg:dual_halpern} both work by establishing the identity
\begin{align}
\begin{aligned}
    0 & = \inprod{\topa x_N}{x_N - y_0} + N \|\topa x_N\|^2 \\
    & \quad \quad + \sum_{(i,j)\in I} \lambda_{i,j} \inprod{\topa x_i - \topa x_j}{x_i - x_j} ,
\end{aligned}
\label{eqn:optimal-algorithm-class-proof-template}
\end{align}
where $I$ is some set of tuple of indices $(i,j)$ with $i > j$ 
(for both algorithms, the consecutive differences of Lyapunov functions are of the form $\lambda_{i,j} \inprod{\topa x_i - \topa x_j}{x_i - x_j}$; sum them up to obtain~\eqref{eqn:optimal-algorithm-class-proof-template}) and $\lambda_{i,j} \ge 0$.
For \ref{alg:halpern}, $I$ consists of $(k+1,k)$ for $k=1,\dots,N-1$ while for \ref{alg:dual_halpern}, $(N,k)$ for $k=1,\dots,N-1$ are used.
In the proof of~\cref{theorem:optimal-method-family}, we identify algorithms (in terms of H-matrices) whose convergence can be proved via~\eqref{eqn:optimal-algorithm-class-proof-template} where $I$ is the union of the two sets of tuples, i.e.,
\begin{align*}
    I = \{(k+1,k) \,|\, k=1,\dots,N-1\} & \\
    \cup \, \{(N,k) \,|\, k=1,\dots,N-1\} & .
\end{align*}
Once~\eqref{eqn:optimal-algorithm-class-proof-template} is established, we obtain the desired convergence rate from monotonicity of $\opA$ and \cref{lemma:convergence-proof-last-step}.

To clarify, the algorithm family as defined in~\cref{theorem:optimal-method-family} does not include \ref{alg:halpern} and \ref{alg:dual_halpern} as $C$ is an open set and \ref{alg:halpern} and \ref{alg:dual_halpern} correspond to points on the boundary $\partial C=\overline{C}\backslash C$.
We choose not to incorporate $\partial C$ in \cref{theorem:optimal-method-family} because doing so leads to cumbersome divisions-by-zero. However, with a specialized analysis, % we can show that 
\cref{proposition:optimal-family-continuous-extension} shows that \ref{alg:halpern} and \ref{alg:dual_halpern} % are indeed in the boundary $\partial C$ 
are indeed a part of the parametrization $\Phi$ 
and therefore that the optimal family ``connects'' \ref{alg:halpern} and \ref{alg:dual_halpern}.
The proof is deferred to~\cref{subsection:appendix-optimal-family-continuous-extension}.

\begin{proposition}
\label{proposition:optimal-family-continuous-extension}
Consider $C$ and $\Phi$ defined in Theorem~\ref{theorem:optimal-method-family}. 
One can continuously extend $\Phi$ to some $u,v \in\partial C$ such that $\Phi(u)=H_{\text{OHM}}$ and $\Phi(v)=H_{\text{Dual-OHM}}$, where $H_{\text{OHM}}$ and $H_{\text{Dual-OHM}}$ are respectively the H-matrix representations of \ref{alg:halpern} and \ref{alg:dual_halpern}.
\end{proposition}

Our proof of~\cref{theorem:optimal-method-family} comes without an explicit characterization of $\Phi$ or the H-matrices (we only show existence\footnote{The H-matrices are, however, ``computable'' in the sense that they can be obtained by solving an explicit system of linear equations, as outlined in \cref{section:appendix-optimal-family-theorem}.}).
Fortunately, for $N=3$, we do have a simple explicit characterization, which directly illustrates the family interpolating between the \ref{alg:halpern} and \ref{alg:dual_halpern}. 
For any $2\times 2$ lower-triangular $\begin{bmatrix} h_{1,1} & 0 \\ h_{2,1} & h_{2,2} \end{bmatrix}$ satisfying $h_{2,1} = 1 - h_{1,1} - h_{2,2}$, $h_{1,1}h_{2,2} = \frac{1}{3}$ and $h_{1,1}, h_{2,2} \in \left[\frac{1}{2}, \frac{2}{3}\right]$, the iterate $y_2$ defined by \eqref{eqn::FPI_H_matrix} satisfies $\sqnorm{y_2 - \opT y_2} \hfill = 4 \sqnorm{\tilA x_3} \hfill \le \hfill \frac{4 \sqnorm{y_0 - y_\star}}{9}$.
In particular, \ref{alg:halpern} and \ref{alg:dual_halpern} each correspond to $(h_{1,1},h_{2,2})=\left(\frac{1}{2},\frac{2}{3}\right)$ and $(h_{1,1},h_{2,2})=\left(\frac{2}{3},\frac{1}{2}\right)$.

As a final comment, we clarify that \cref{theorem:optimal-method-family} is not the exhaustive parametrization of optimal algorithms; there are other exact optimal algorithms that \cref{theorem:optimal-method-family} does not cover.
We provide such examples in Appendix~\ref{section:appendix-optimal-family-remaining-proof_missing}.
The complete characterization of the set of exact optimal algorithms seems to be challenging, and we leave it to future work.

\section{H-duality for fixed-point algorithms}
\label{section:H-duality}

In this section, we present an H-duality theory for fixed-point algorithms. 
At a high level, H-duality states that an algorithm and its convergence proof can be \emph{dualized} in a certain sense.
This result provides a formal connection between \ref{alg:halpern} and \ref{alg:dual_halpern}.

\subsection{H-dual operation}

Before providing the formal definition, we observe the following.
The H-matrix of \ref{alg:halpern} is
\begin{align*}
\begin{split}
    & \left(H_{\text{OHM}}\right)_{k,j}=
    \begin{cases}
        -\frac{j}{k(k+1)} & \text{if } j < k \\
        \frac{k}{k+1}     & \text{if } j = k
    \end{cases}
\end{split} 
\end{align*}
while the H-matrix of \ref{alg:dual_halpern} is
\begin{align*}
\begin{split}
     \left(H_{\text{Dual-OHM}}\right)_{k,j} =
    \begin{cases}
        -\frac{N-k}{(N-j)(N-j+1)} & \text{if } j < k \\
        \frac{N-k}{N-k+1}         & \text{if } j = k.
    \end{cases}
\end{split} 
\end{align*}
We derive the H-matrix representations in \cref{section:appendix-algorithm-equivalence}.
For now, note the relationship $H_{\text{Dual-OHM}} = H_{\text{OHM}}^\at$, where superscript $\at$ denotes anti-diagonal transpose $\left(H^\at\right)_{k,j} = \left(H\right)_{N-j,N-k}$.
The anti-diagonal transpose operation reflects a square matrix along its anti-diagonal direction and preserves lower triangularity.

Given an algorithm represented by $H$, we define its \emph{H-dual} as the algorithm represented by $H^\at$.
In other words, given a general iterative algorithm defined as~\eqref{eqn::FPI_H_matrix}, its H-dual is the algorithm with the update rule
\begin{align} \label{eqn::FPI_H_dual}
\begin{split}
    y_{k+1} &= y_k  - \sum_{j=0}^{k} h_{N-j-1,N-k-1}\left( y_j -\opT y_j\right)  
\end{split}
\end{align}
for $k=0,1,\dots,N-2$.  

\begin{proposition}
\ref{alg:dual_halpern} and \ref{alg:halpern} are H-duals of each other.
\end{proposition}

\subsection{H-duality theorem}
\ref{alg:halpern} and \ref{alg:dual_halpern} are H-duals of each other, and they share the identical rates on $\sqnorm{y_{N-1} - \opT y_{N-1}}$.
This symmetry is not a coincidence; the following H-duality theorem explains the connection between them.

\begin{theorem}[Informal]  \label{thm::H-duality_FPI_informal} 
Let $H \in \mathbb{R}^{(N-1)\times (N-1)}$ be lower triangular and $\nu >0$. Then the following are equivalent.
\begin{itemize}
    \item $\norm{y_{N-1} - \opT y_{N-1}}^2 \leq \nu\norm{y_0-y_\star}^2$ for~\eqref{eqn::FPI_H_matrix} with $H$ can be proved with primal Lyapunov structure.
    \item $\norm{y_{N-1} - \opT y_{N-1}}^2 \leq \nu\norm{y_0-y_\star}^2$ for \eqref{eqn::FPI_H_matrix} with $H^\at$ (H-dual) can be proved with dual Lyapunov structure.
\end{itemize}
\end{theorem}
 
We defer the precise statement and the proof of \cref{thm::H-duality_FPI_informal} to \cref{sec::appendix_H_duality_pf}. The high-level takeaway of \cref{thm::H-duality_FPI_informal} is as follows: If an algorithm achieves a certain convergence rate with respect to $\norm{y_{N-1} - \opT y_{N-1}}^2$ (based on a certain proof structure), then the proof can be \emph{H-dualized} to guarantee the exact same convergence rate for its H-dual. The relationship between \ref{alg:halpern} and \ref{alg:dual_halpern} is an instance of this duality correspondence (with $\nu=1/N^2$).
 
Different from the prior H-duality result for convex minimization \cite{KimOzdaglarParkRyu2023_timereversed}, we show that algorithms in H-dual relationship share the convergence rate with respect to the same performance measure $\sqnorm{y_{N-1} - \opT y_{N-1}}$ and the same initial condition $\sqnorm{y_0 - y_\star}$. One could say that the performance measure $\sqnorm{y_{N-1} - \opT y_{N-1}}$ is ``self-dual''. This contrasts with the prior H-duality \cite{KimOzdaglarParkRyu2023_timereversed}, which showed that function values are ``dual'' to gradient magnitude in convex minimization. While the H-duality theory of this work and that of \cite{KimOzdaglarParkRyu2023_timereversed} share some superficial similarities, the unifying fundamental principle remains unknown, and finding one is an interesting subject of future research.

\section{Analysis of \ref{alg:dual-feg} for minimax problems\!}
\label{section:minimax}

\begin{figure*}[ht!]
\centering
\begin{subfigure}{0.45\textwidth}
\centering
    \includegraphics[width=\linewidth]{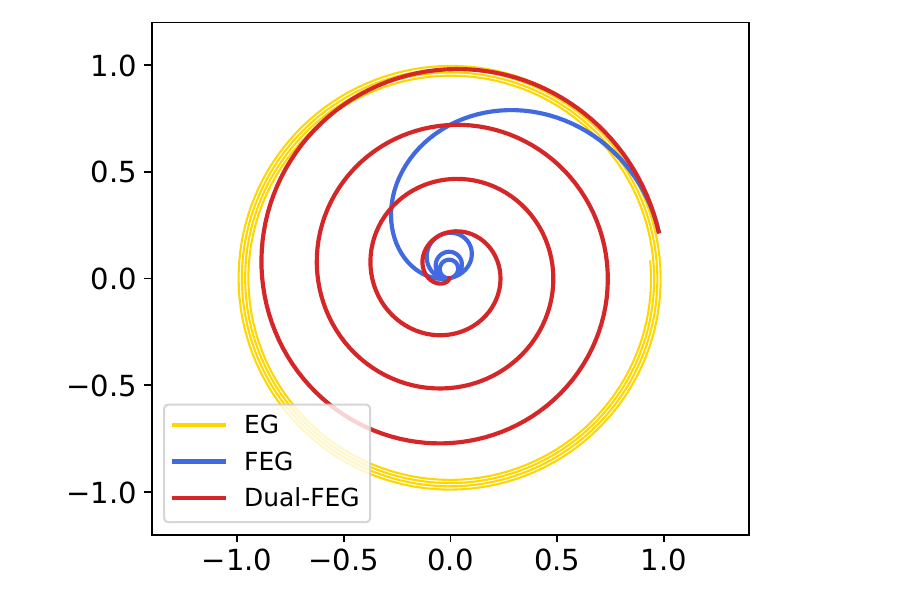}
    \caption{Trajectories from a 2D bilinear example}
    \label{subfig:2D_bilinear_trajectories}
\end{subfigure}
~
\begin{subfigure}{0.45\textwidth}
\centering
    \includegraphics[width=\linewidth]{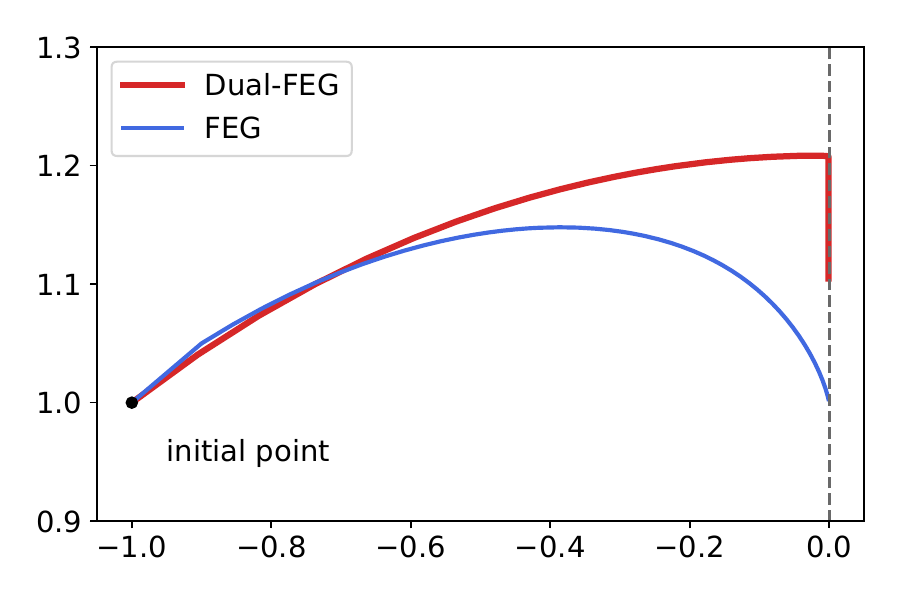}
    \caption{Trajectories from a 2D non-bilinear example} \label{subfig:u_square_v_trajectories}
\end{subfigure}
\caption{Trajectories generated by minimax optimization algorithms on \textbf{(left)} $\vL(u,v) = uv$ with a random initial point of norm 1 and \textbf{(right)} $\vL(u, v) = u^2 v$ with initial point $(-1, 1)$, where the dashed vertical line indicates the set of optima (saddle points).}
\label{fig:trajectory_plots}
\end{figure*}

In this section, we present the convergence analysis of \ref{alg:dual-feg} and additionally provide the observation that \ref{alg:dual-feg} is the H-dual of \ref{alg:feg}.

Algorithms such as \ref{alg:halpern} and \ref{alg:dual_halpern} are sometimes referred to as ``implicit methods'' since they can be expressed using resolvents, as discussed in Section~\ref{subsection:proximal-forms}. The results of this section show that \ref{alg:dual_halpern} has an ``explicit'' counterpart \ref{alg:dual-feg}, which uses direct gradient evaluations instead.

First, we formally state the convergence result.
\begin{theorem}
Let $\vL\colon \reals^n \times \reals^m \to \reals$ be convex-concave and $L$-smooth.
For $N \ge 1$ fixed, if $0 < \alpha \le \frac{1}{L}$, \ref{alg:dual-feg} exhibits the rate
\begin{align*}
    \sqnorm{\nabla\vL (x_N)} = \sqnorm{\sop{x_N}} \le \frac{4\sqnorm{x_0 - x_\star}}{\alpha^2 N^2},
\end{align*}
where $x_\star \in \zer\nabla_\pm \vL$ is a solution.
\end{theorem}

\begin{proof}[Proof outline]
Define $g_N = \sop{x_N}$. 
In \cref{subsection:appendix-Dual-FEG-Lyapunov}, we show that
\begin{align*}
    V_k = -\alpha \sqnorm{z_k + g_N} + \frac{2}{N-k} \inprod{z_k + g_N}{x_k - x_N}
\end{align*}
is nonincreasing in $k=0,1,\dots,N-1$, and $V_{N-1} \ge 0$.
This implies $0 \le V_0 = -\alpha \sqnorm{g_N} + \frac{2}{N} \inprod{g_N}{x_0 - x_N}$. 
Finally, divide both sides by $\frac{2}{N}$ and apply~\cref{lemma:convergence-proof-last-step}.
\end{proof}

\paragraph{\ref{alg:dual-feg} and \ref{alg:feg} are H-duals.}
\ref{alg:dual-feg} has intermediate iterates $x_{k+\hf}$, serving a role similar to intermediate iterates of \ref{alg:feg}.
Consider the following H-matrix representation for \ref{alg:dual-feg}, analogous to~\eqref{eqn::FPI_H_matrix}:
\begin{align*}
    x_{(\ell+1)/2} = x_{\ell/2} - \frac{1}{L} \sum_{i=0}^{\ell} h_{(\ell+1)/2, i/2} \sop{x_{i/2}}
\end{align*} 
for $\ell=0,1,\dots,2N-1$.
For \ref{alg:dual-feg}, the H-matrix is the $2N\times 2N$ lower-triangular matrix
\begin{align*}
    (H_{\text{Dual-FEG}})_{\ell,i} =     \begin{cases}
         h_{\ell/2,(i-1)/2} & \text{if } i \leq \ell \\
        0     & \text{if } i > \ell.
    \end{cases}
\end{align*} 
We can analogously define the H-matrix $H_{\text{FEG}}$ for \ref{alg:feg}, which also has extragradient-type intermediate iterates. 
It turns out that $H_{\text{FEG}}$ and $H_{\text{Dual-FEG}}$ are anti-diagonal transposes of each other.
 
\begin{proposition}
\label{proposition:dual-feg-H-dual}
\ref{alg:dual-feg} and \ref{alg:feg} are H-duals of each other.
\end{proposition}

We defer the proof to \cref{subsection:proof-of-dual-feg-H-dual}. 
This intriguing H-dual relationship and identical convergence rates of \ref{alg:feg} and \ref{alg:dual-feg} strongly indicate the possible existence of H-duality theory for smooth minimax optimization; we leave its formal treatment to future work.

\section{Continuous-time analysis of dual-anchoring}
\label{section:continuous-time}

\begin{figure*}[ht!]
\centering
\begin{subfigure}{0.45\textwidth}
\centering
    \includegraphics[width=\linewidth]{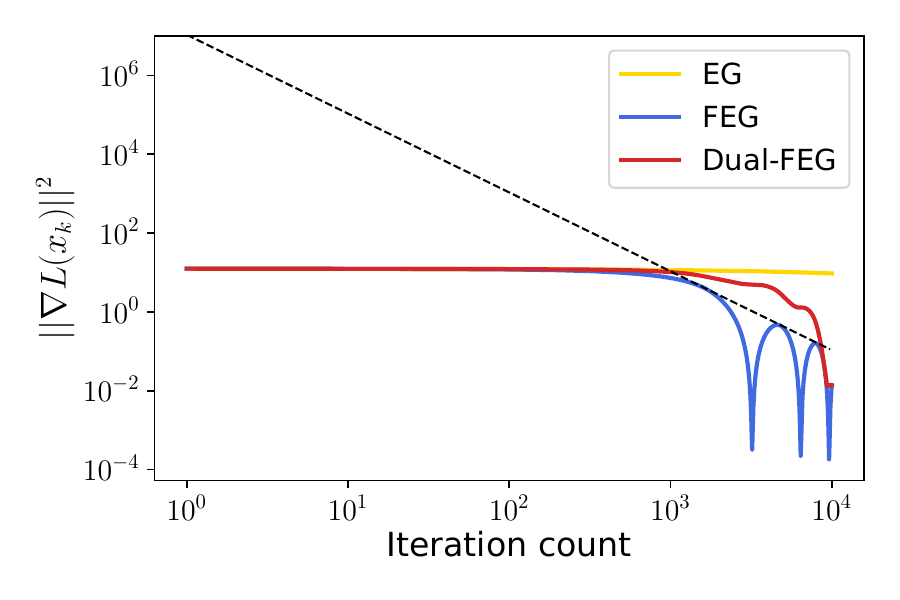}
    \caption{Worst case bilinear problem}
    \label{subfig:ouyang_construction_performance}
\end{subfigure}
\hfill
\begin{subfigure}{0.45\textwidth}
\centering
    \includegraphics[width=\linewidth]{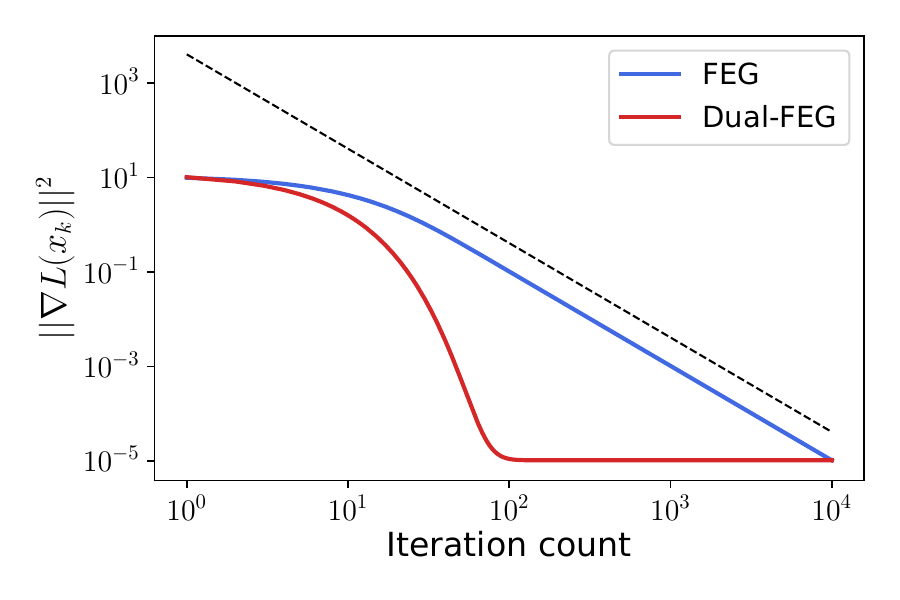}
    \caption{Strongly-convex-strongly-concave bilinear problem}
    \label{subfig:ouyang_strongly_monotone_performance}
\end{subfigure}
\caption{Performance of minimax algorithms in reducing $\sqnorm{\nabla \vL(x_k)}$ for bilinear problem instances. The dashed black line indicates the theoretical upper bounds for \ref{alg:feg}.} 
\label{fig:performance_plots}
\end{figure*}

In this section, we outline the analysis of the Dual-Anchor ODE \eqref{ode:dual-anchor}, which is the common continuous-time model for both \ref{alg:dual_halpern} and \ref{alg:dual-feg}, as derived in \cref{appendix:continuous_time_limit_algorithms}.
We then introduce the notion of H-dual for ODE models and show that Dual-Anchor ODE is the H-dual of Anchor ODE~\eqref{ode:anchor}, the continuous-time model for \ref{alg:halpern} and \ref{alg:feg}.

\begin{theorem} \label{thm:continuous_convergnece_rate}
Let $\opA \colon \reals^d \to \reals^d$ be Lipschitz continuous and monotone.  
For $T>0$, the solution $X\colon [0,T) \to \reals^d$ of the Dual-Anchor ODE \eqref{ode:dual-anchor} with initial conditions $X(0) = X_0, Z(0) = 0$ uniquely exists, and $X(T) = \lim_{t\to T^-} X(t)$ satisfies
\begin{align*}
    \norm{ \opA(X(T)) }^2 \le \frac{4 \norm{ X_0 - X_\star }^2 }{T^2} 
\end{align*}
where $X_\star \in \zer \opA$.
\end{theorem}

\begin{proof} [Proof outline.]
    Define $V\colon [0,T) \to \reals$ by
    \begin{align*}
        V(t) &= -\norm{ Z(t) + \opA(X(T)) }^2 \\
        & \quad \quad + \frac{2}{T-t} \inner{ Z(t) + \opA(X(T)) }{  X(t) - X(T) } .
    \end{align*}
    In \cref{appendix:proof_of_continuous_convergnece_rate}, we show that 
    \begin{align*}
        \dot{V}(t)
        &= -\frac{2 \inner{ X(t) - X(T) }{ \opA(X(t)) - \opA(X(T)) }}{(T-t)^2} \le 0.
    \end{align*} 
    \cref{lemma:ode_terminal_infomation} shows $\lim_{t\to T^-} V(t) = 0$. 
    From $Z(0) = 0$,
    \begin{align*}
        V(0) = -\norm{ \opA(X(T)) }^2 - \frac{2}{T} \inner{ \opA(X(T)) }{ X(T) - X_0 }. 
    \end{align*}
    Dividing both sides of $0 = \lim_{t\to T^-} V(t) \le V(0)$ by $\frac{2}{T}$ and applying \cref{lemma:convergence-proof-last-step} we get the desired inequality.
\end{proof} 

\paragraph{Dual-Anchor ODE and Anchor ODE are H-duals.}
The continuous-time analogue of the H-matrix representation \cite{KimYang2023_unifying,KimOzdaglarParkRyu2023_timereversed} is
\begin{align*} 
    \dot{X}(t) = - \int_{0}^{t} H(t,s) \opA(X(s)) ds ,
\end{align*}
where we refer to $H(t,s)$ as the \emph{H-kernel}. The \emph{H-dual ODE} is defined by
\[
    \dot{X}(t) = - \int_{0}^{t} H^\at(t,s) \opA(X(s)) ds
\]
where $H^\at(t,s) = H(T-s,T-t)$ is the analogue of the anti-diagonal transpose. 
We prove the following in \cref{appendix:continuous_time_H-matrix}.
\begin{proposition}
\label{proposition:continuous-time-H-dual}
Dual-Anchor ODE \eqref{ode:dual-anchor} and Anchor ODE \eqref{ode:anchor} are H-duals of each other.
\end{proposition}

\paragraph{Generalization to differential inclusion.}
\cref{thm:continuous_convergnece_rate} assumes Lipschitz continuity of $\opA$ for simplicity.
However, even if we only assume that $\opA$ is maximally monotone, the existence of a solution and the convergence result can be established for the differential inclusion
\begin{align*}
    \dot{X}(t) & \in -Z(t) - \opA(X(t)) \\
    \dot{Z}(t) & \in -\frac{1}{T-t} Z(t) - \frac{1}{T-t} \opA(X(t)) ,
\end{align*}
which is a generalized continuous-time model for possibly set-valued operators.
We provide the details in \cref{appendix:continuous-time-differential-inclusion}.

\section{Experiments}
\label{section:experiments}

In this section, we present some numerical simulations illustrating the dynamics of dual-anchor algorithm.
In \cref{fig:trajectory_plots}, we compare the trajectories of \ref{alg:feg}, \ref{alg:dual-feg}, and the Extragradient (EG) \cite{Korpelevich1976_extragradient} algorithms.
\cref{subfig:2D_bilinear_trajectories} uses a bilinear example $\vL(u,v) = uv$ with $\alpha=0.005$ and $N=5000$ as an approximation of the algorithms' behavior in the limit $\alpha\to 0$, and \cref{subfig:u_square_v_trajectories} uses a 
non-bilinear example $\vL(u, v) = u^2 v$ (which is convex-concave and smooth on $[-1,1] \times \reals_{\ge 0}$) with $\alpha=0.05$ and $N=10000$.
In Figures~\ref{fig:performance_plots} and \ref{fig:basis_pursuit}, 
we plot $\sqnorm{\nabla \vL(x_k)}$.
\cref{subfig:ouyang_construction_performance} uses a worst-case bilinear example due to \citet{OuyangXu2021_lower}:
\begin{align*} 
    \vL(u,v) = \frac{1}{2} u^\intercal G u - g^\intercal u - \inprod{Au - b}{v}
\end{align*}
where $u,v \in \reals^n$ with $n=200$.
The precise terms of $\vL$ are stated in \cref{subsection:appendix-experiment-Ouyang-construction}. 
We use initial points $u_0 = 0, v_0 = 0$ and use $\alpha=1.0$ and $N=10000$.
\cref{subfig:ouyang_strongly_monotone_performance} uses the same example of \citet{OuyangXu2021_lower} with additional $\mu$-strongly convex term in $u$ and $\mu$-strongly-concave term in $v$ (with $\mu=0.1$):
\begin{align*}
    \vL_\mu (u,v) = \frac{1}{2} u^\intercal (G + \mu I) u - g^\intercal u - \inprod{Au - b}{v} - \frac{\mu}{2} \|v\|^2 .
\end{align*}
\cref{fig:basis_pursuit} considers the linearly constrained convex minimization problem
\begin{align*}
\begin{array}{cc}
    \underset{u \in \reals^n}{\text{minimize}} & h_\delta(u) \\
    \text{subject to} & Au = b
\end{array}
\end{align*}
where $h_\delta (u) = \begin{cases} \frac{1}{2} \|u\|^2 & \text{if} \,\, \|u\| \le \delta \\ \delta \|u\| - \frac{1}{2} \delta^2 & \text{otherwise} \end{cases}$ is the Huber loss and $A \in \reals^{m \times n}$ ($m < n$).
We run \ref{alg:feg} and \ref{alg:dual-feg} on the Lagrangian $\vL (u,v) = h_\delta (u) + \inprod{Au - b}{v}$.
We choose $n = 100$, $m=20$, randomly generate entries of $A$ as i.i.d.\ $\cN(0,1/n^2)$, randomly choose $\frac{n}{10}$ coordinates where $\overline{u} \in \reals^n$ has nonzero values (uniform random in $[0,1]$) and set $b = A\overline{u}$.
We use initial points $u_0, v_0$ with i.i.d.\ standard normal coordinates and choose $\delta = 0.1$, $\alpha=0.5$ and $N=10^5$.

\begin{figure}[t]
\centering 
\includegraphics[width=\linewidth, trim=0 0 0 0.15in, clip]{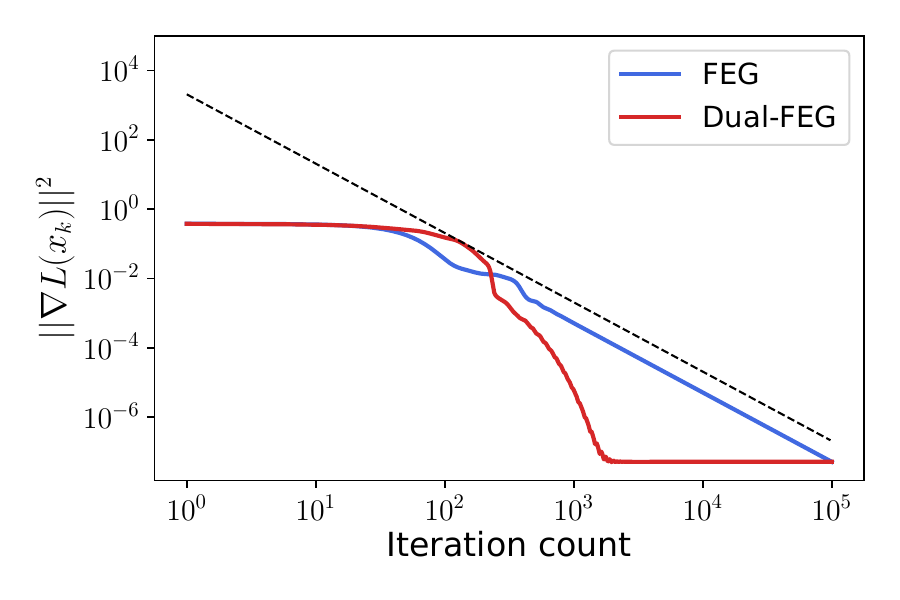}
\vspace{-.3in}
\caption{Performance of algorithms in reducing $\sqnorm{\nabla \vL(x_k)}$ for the Lagrangian of a linearly constrained smooth convex minimization problem. The dashed black line indicates the theoretical upper bounds for \ref{alg:feg}.}
\label{fig:basis_pursuit}
\end{figure}

The trajectories of \ref{alg:dual-feg} are qualitatively different from those of \ref{alg:feg}, indicating that the two acceleration mechanisms are genuinely different.
Interestingly, however, we find that the last iterates of \ref{alg:dual-feg} and \ref{alg:feg} are identical when the operator is linear, as we show in~\cref{subsection:appendix-experiment-identical-terminal-iterates}. Indeed, Figures~\ref{subfig:2D_bilinear_trajectories} and \ref{fig:performance_plots} show the two algorithms arriving at the same point at the terminal iteration. However, this is not true for nonlinear operators; in Figure~\ref{subfig:u_square_v_trajectories}, we observe that the terminal iterate of \ref{alg:dual-feg} does not coincide with the terminal iterate of \ref{alg:feg}.

In Figures~\ref{subfig:ouyang_strongly_monotone_performance} and \ref{fig:basis_pursuit}, we observe that \ref{alg:dual-feg} approaches the terminal accuracy much earlier compared to FEG, which progresses much more steadily.
In \cref{subsection:appendix-experiment-early-stopping}, we prove that in continuous-time, this phenomenon occurs when the objective function is strongly-convex-strongly-concave.
We anticipate that the discrete dual algorithms also exhibit a similar phenomenon, and we leave this investigation to future work.

\paragraph{Comparison of primal and dual algorithms.}
We observe that the dual algorithms require $N$ or $T$ to be specified in advance, while the primal ones do not. In this respect, primal algorithms are advantageous over dual algorithms.
On the other hand, we observe that for problems involving strongly monotone operators, dual algorithms exhibit much faster convergence than primal algorithms in the earlier iterations.
As the primal and dual algorithms share the same worst-case guarantee but display distinct characteristics, determining the best choice of algorithm may require considering other criterion that depend on the particular application scenario.

\section{Conclusion}
This work presents a new class of accelerations in fixed-point and minimax problems and provides the correct perspective that optimal acceleration is a family of algorithms. Our findings open several new avenues of research on accelerated algorithms. Since the worst-case guarantee, as a criterion, does not uniquely identify a single best method, additional criteria should be introduced to break the tie. Moreover, the role of H-duality in discovering the new dual accelerated algorithms hints at a deeper yet unexplored significance of H-duality in the theory of first-order algorithms.

% \section*{Acknowledgements}

\section*{Impact statement}
This paper presents a theoretical study of optimization. 
Given its abstract nature, we do not expect our work to raise any significant ethical or societal concerns and consider it neutral in terms of immediate real-world impact. 

\section*{Acknowledgements}
We thank Donghwan Kim and Felix Lieder for the discussion on their processes of using the performance estimation problem (PEP) methodology for finding the exact optimal algorithm presented as APPM and OHM. We thank the anonymous referees for inspiring the exploration of fast convergence of dual algorithms with strongly monotone operators.

\newpage

\bibliography{main}
\bibliographystyle{icml2024}

\newpage
\appendix
\onecolumn

\section{Related work}
\label{section:appendix-related-work}

\paragraph{Fixed-point algorithms.}
Iterative algorithms for solving fixed-point problems~\eqref{def:fixed-point-problem} have been extensively studied over a long period, and among them, Picard iteration, Krasnosel’skii–Mann (KM) Iteration and Halpern Iteration stand out as representative classes, each defined by:  
\begin{align}
    y_{k+1} &= \opT y_k, \label{eqn::Picard} \tag{Picard}  \\
    y_{k+1} &= \lambda_{k+1} y_k + (1-\lambda_{k+1}) \opT y_k,  \label{eqn::KM} \tag{KM} \\
     y_{k+1} &= \lambda_{k+1} y_0 + (1-\lambda_{k+1}) \opT y_k  \label{eqn::Halpern} \tag{Halpern}
\end{align}
The formal study of \ref{eqn::Picard} iteration dates up to \citet{Banach1922_operations}.
The class \ref{eqn::KM} generalizes the works of \citet{Krasnoselskii1955_two} and \citet{Mann1953_mean}; early works \cite{Ishikawa1976_fixed,BorweinReichShafrir1992_krasnoselskimann} focused on its asymptotic convergence $\norm{y_k - \opT y_k}^2 \to 0$, while quantitative rates of $\mathcal{O}(1/k)$ to $o(1/k)$ were respectively demonstrated by \citet{CominettiSotoVaisman2014_rate,LiangFadiliPeyre2016_convergence,BravoCominetti2018_sharp} and \citet{BaillonBruck1992_optimal,DavisYin2016_convergence,Matsushita2017_convergence}.
For the \ref{eqn::Halpern} class, devised by \citet{Halpern1967_fixed}, \citet{Wittmann1992_approximation,Xu2002_iterative} established asymptotic convergence. Concerning quantitative convergence, \cite{Leustean2007_rates} provided the rate $\sqnorm{y_k - \opT y_k} = \mathcal{O}(1/(\log k)^2)$, which was later improved to $\mathcal{O}(1/k)$ \cite{Kohlenbach2011_quantitative}, and to $\mathcal{O}(1/k^2)$, first by \citet{SabachShtern2017_first}, and then by \citet{Lieder2021_convergence} with a tighter constant, using the choice $\lambda_k=\frac{1}{k+1}$.
\citet{ParkRyu2022_exact} established the exact optimality of the convergence rate from \citet{Lieder2021_convergence} (and the independently discovered equivalent result of \citet{Kim2021_accelerated}) by constructing a matching complexity lower bound.
Some notable acceleration results for fixed-point problems not covered by the above list include Anderson-type acceleration \cite{Anderson1965_iterativea,WalkerNi2011_andersona,ZhangODonoghueBoyd2020_globally} and inertial-type acceleration \cite{Mainge2008_convergence,DongYuanChoRassias2018_modified,Shehu2018_convergence,ReichThongCholamjiakVanLong2021_inertial}.

\paragraph{Minimax optimization algorithms.}
Smooth convex-concave minimax optimization is a classical problem, whose investigation dates up to \citet{Korpelevich1976_extragradient, Popov1980_modification}, which respectively first developed the Extragradient (EG) and optimistic gradient (OG) whose variants have been studied extensively \cite{SolodovSvaiter1999_hybrid, Nemirovski2004_proxmethod, Nesterov2007_dual, RakhlinSridharan2013_online, DaskalakisIlyasSyrgkanisZeng2018_training} across the optimization and machine learning communities.
Recently there has been rapid development in the theory of reducing $\sqnorm{\nabla \vL(\cdot)}$; while EG and OG converge at the rate of $\sqnorm{\nabla \vL(u_k,v_k)}=\cO(1/k)$ \cite{GorbunovLoizouGidel2022_extragradient, CaiOikonomouZheng2022_finitetime}, using the anchoring mechanism \cite{RyuYuanYin2019_ode} motivated by the optimal Halpern iteration \cite{Halpern1967_fixed, Kim2021_accelerated, Lieder2021_convergence}, \citet{Diakonikolas2020_halpern} achieved the near-optimal rate $\Tilde{\cO}(1/k^2)$, which was finally accelerated to $\cO(1/k^2)$ via Extra Anchored Gradient (EAG) algorithm \cite{YoonRyu2021_accelerated}.
\citet{LeeKim2021_fast} provided a constant factor improvement over EAG while generalizing the acceleration to weakly nonconvex-nonconcave problems, and \citet{Tran-DinhLuo2021_halperntype, CaiZheng2023_accelerated} presented single-call versions of the acceleration.
\citet{BotCsetnekNguyen2023_fast} achieved asymptotic $o(1/k^2)$ convergence based on the continuous-time perspective.
The conceptual connection between smooth minimax optimization and fixed-point problems (proximal algorithms) have been formally studied in \citet{MokhtariOzdaglarPattathil2020_unified, YoonRyu2022_accelerated}.

\paragraph{Continuous-time analyses.} 
Taking the continuous-time limit of an iterative algorithm results in an ordinary differential equation (ODE), and they often more easily reveal the structure of convergence analysis.
The ODE models for OG and EG were respectively studied by \citet{CsetnekMalitskyTam2019_shadow, Lu2022_os}. 
The continuous-time analysis of accelerated algorithms was initiated by \citet{SuBoydCandes2014_differential, KricheneBayenBartlett2015_accelerated}, and the anchor acceleration ODE for fixed-point and minimax problems was first introduced by~\citet{RyuYuanYin2019_ode} and  
then generalized to a broader family of differential inclusion with more rigorous treatment by~\citet{SuhParkRyu2023_continuoustime}.
\citet{BotCsetnekNguyen2023_fast} studied a different family of ODE that achieves acceleration asymptotically. 
An ODE involving the coefficient of the form $\frac{1}{T-t}$ with fixed terminal time $T$ was first presented in~\citet{SuhRohRyu2022_continuoustime},  
as continuous-time model of the OGM-G algorithm \cite{KimFessler2021_optimizing}. 
The formal concept of H-duality in continuous-time has been proposed by \citet{KimOzdaglarParkRyu2023_timereversed}, while \citet{KimYang2023_convergence, KimYang2023_unifying} also presents a similar result.

\paragraph{Performance estimation problem (PEP) technique.}
At a high level, the PEP methodology \cite{DroriTeboulle2014_performance, TaylorHendrickxGlineur2017_smooth, TaylorBach2019_stochastic, RyuTaylorBergelingGiselsson2020_operator, DasGuptaVanParysRyu2023_branchandbound} provides a computer-assisted framework for finding tight convergence proofs and optimizing the step-sizes of optimization of algorithms, and many efficient algorithms and novel proofs have recently been discovered with the aid of this methodology \cite{KimFessler2016_optimized, Lieder2021_convergence, Kim2021_accelerated, KimFessler2021_optimizing, ParkRyu2021_optimal, ParkRyu2022_exact, GorbunovLoizouGidel2022_extragradient, TaylorDrori2022_optimal, JangGuptaRyu2023_computerassisted, ParkRyu2023_accelerated}.
\newpage

Specifically relevant to our work are the prior, concurrent work of \citet{Kim2021_accelerated} presenting APPM and of \citet{Lieder2021_convergence} presenting OHM. 
It was later shown that APPM and OHM are equivalent, and they represent one element within the family of exact optimal algorithms that we present in this work. We find it somewhat surprising that the two authors independently found the same algorithm when there is an infinitude of answers. Therefore, we personally asked Kim and Lieder to understand how their processes led to their discovery and not any other choice.

In a personal communication, Donghwan Kim explained that the search process in \cite{KimFessler2016_optimized, Kim2021_accelerated}, at a high level, involved the following steps:
\begin{enumerate}
    \item Identify the set of inequalities needed to establish a convergence guarantee for a simpler ``reference algorithm'' (such as gradient descent or proximal point method).
    \item Set the algorithm step-sizes as additional variables within the PEP formulation, and perform joint minimization of the convergence rate (worst-case complexity) with respect to the step-sizes. This may result in a nonconvex formulation, requiring heuristics such as alternating minimization and repetition of local minimization followed by perturbation (to escape from local optimum).
    \item Iterate Step~2 while restricting the set of inequalities available within the proof (that PEP formulation uses) based on the pattern detected from Step~1. This often facilitates convergence to the global minimum and encourages numerically cleaner solutions that are conducive to analysis by humans.
    In the search of APPM \cite{Kim2021_accelerated}, for instance, only the inequalities $\inprod{\tilA x_{i+1} - \tilA x_i}{x_{i+1} - x_i} \ge 0$ for $i=1,\dots,N-1$ and $\inprod{\tilA x_N}{x_N - y_\star} \ge 0$ were used.
\end{enumerate}
While the strategy of using a constrained set of inequalities is effective for identifying \emph{an} optimal algorithm, it would prevent the discovery of other optimal algorithms that require distinct proof structures.

In another personal communication, Felix Lieder explained that the search process of OHM in \citet{Lieder2021_convergence} involved the optimization of step-sizes as variables within the PEP formulation (as in the Step 2 in the case of APPM), but unlike \cite{Kim2021_accelerated}, the entire set of available inequalities was used (without restriction).
Local nonlinear programming solvers were employed to tackle this nonconvex optimization, and within multiple trials with random initialization, the step-sizes often converged near the step-size values for OHM (as detailed in \cite{Lieder2018}).
In this case, despite the numerical indication that optimal step-sizes were not unique, the authors' main goal was to identify the simplest, hence most easily interpretable solution, which turned out to be OHM in the end.

In our work, we take a mixture of deductive and numerical approaches.
Our initial motivation to investigate the H-dual of OHM (namely \ref{alg:dual_halpern}) came from the convex H-duality theory \cite{KimOzdaglarParkRyu2023_timereversed}, and PEP was mainly utilized for quick search of proofs (i.e.\ the correct way of combining the inequalities) and for numerical verification of the convergence guarantees.

\paragraph{Non-uniqueness of exact optimal algorithms.}
In the setup of minimizing a non-smooth convex function $f$ whose subgradient magnitude is bounded by $M$ (so the objective function is $M$-Lipschitz continuous), there are at least 4 different algorithms achieving the exact optimal complexity.
Suppose that $\|x_0 - x_\star\| \le R$, where $x_\star$ is a minimizer of $f$.
The subgradient method of \cite{Nesterov2004_introductory}, given by $x_{k+1} = x_k - \frac{R}{\sqrt{N+1}} \frac{g_k}{\norm{g_k}}$ where $g_k \in \partial f(x_k)$, exhibits the rate $f(x_N) - f_\star \le \frac{MR}{\sqrt{N+1}}$ where $f_\star = f(x_\star)$.
The exact same rate is achieved by the algorithm of \cite{DroriTeboulle2016_optimal} (which is a variant of the cutting-plane method), a fixed step-size algorithm of \cite{DroriTaylor2020_efficient} (without gradient normalization), and its line-search variant.
As there is a matching lower bound $f(x_N) - f_\star \ge \frac{MR}{\sqrt{N+1}}$ \cite{DroriTeboulle2016_optimal}, all these methods are exactly optimal in terms of the worst-case complexity.
% Among these, however, only the algorithm by \citet{DroriTaylor2020_efficient} uses fixed step sizes and admits an H-matrix representation. 

In the fixed-point setup (or the equivalent monotone inclusion setup), a Halpern-type algorithm by \citet{SuhParkRyu2023_continuoustime} that uses adaptive interpolation (anchoring) coefficients also achieves the exact optimal complexity (same rate as \ref{alg:halpern}).
Hence, in a strict sense, our work is not the first discovery of the non-uniqueness of exact optimal algorithms.
However, to the best of our knowledge, our work is the first to identify distinct \emph{fixed step-size} algorithms sharing the same exact optimal complexity for fixed-point problems.
Additionally, while the adaptive algorithm of \cite{SuhParkRyu2023_continuoustime} relies on a mechanism that is similar to \ref{alg:halpern}, our dual algorithms use an arguably different mechanism from that underlying \ref{alg:halpern}.

\newpage
\section{Equivalence of algorithms}
\label{section:appendix-algorithm-equivalence}

In this section, we show equivalences between distinct forms of algorithms.

\subsection{\ref{alg:halpern}}
\textbf{Form 1} (Anchoring form).
\begin{align} 
\label{eqn:appendix-Halpern-anchor-form}
    y_{k+1} = \frac{k+1}{k+2} \opT y_k + \frac{1}{k+2} y_0 
\end{align}

\textbf{Form 2} (Momentum form).
\begin{align}
\label{eqn:appendix-Halpern-momentum-form}
    y_{k+1} = y_k - \frac{1}{k+2} (y_k - \opT y_k) + \frac{k}{k+2} \left( \opT y_k - \opT y_{k-1} \right)    
\end{align}
where $\opT y_{-1} = x_0$.

\textbf{Form 3} (H-matrix form).
\begin{align*}
    y_{k+1} & = y_k  - \sum_{j=0}^{k} h_{k+1,j+1}(y_j - \opT y_j)  
\end{align*}
where
\begin{align}
\label{eqn:appendix-Halpern-H-matrix}
     h_{k,j} =
    \begin{cases}
        -\frac{j}{k(k+1)} & \text{if } j < k  , \\
        \frac{k}{k+1}         & \text{if } j = k .
    \end{cases}
\end{align}

\textbf{Form 4} (APPM).
\begin{align}
\label{eqn:appendix-appm-definition}
\begin{split}
    x_{k+1} & = \JA (y_k) \\
    y_{k+1} & = x_{k+1} + \frac{k}{k+2} (x_{k+1} - x_k) - \frac{k}{k+2} (x_k - y_{k-1})
\end{split}
\end{align}
where $y_{-1} = x_0 = y_0$ and $\opA = 2(\opT + \opI)^{-1} - \opI \iff \opT = 2\JA - \opI$.

\vspace{.1in}

\paragraph{Form 1 $\implies$ Form 2.} Multiplying $k+2$ to \eqref{eqn:appendix-Halpern-anchor-form} gives
\begin{align*}
    (k+2)y_{k+1} = (k+1)\opT y_k + y_0.
\end{align*}
Thus $y_0=(k+2)y_{k+1} - (k+1)\opT y_k$. 
Substituting this into $y_0=(k+1)y_k - k \opT y_{k-1}$ and rearranging:
\begin{align*}
    0 = (k+2)y_{k+1} - (k+1)\opT y_k - (k+1)y_k + k\opT y_{k-1} \iff \text{[Form 2]}.
\end{align*}

\paragraph{Form 2 $\iff $ Form 4.}
Direct substitution $x_{k+1} = \frac{1}{2} (y_k + \opT y_k)$ (which follows from $\JA = \frac{1}{2} (\opI + \opT)$) shows that~\eqref{eqn:appendix-Halpern-momentum-form} and \eqref{eqn:appendix-appm-definition} are the identical update rules.

\paragraph{Form 4 $\implies$ Form 1.}
Multiplying $k+2$ throughout the second line of~\eqref{eqn:appendix-appm-definition} gives
\begin{align*}
    & (k+2) y_{k+1} = (2k + 2) x_{k+1} - 2k x_k + k y_{k-1} \\
    \iff & (k+2) y_{k+1} - k y_{k-1} = (2k + 2) x_{k+1} - 2k x_k .
\end{align*}
Summing up the last line, we see that the terms telescope:
\begin{align*}
    & \sum_{j=0}^k \left[ (j+2) y_{j+1} - j y_{j-1} \right] = \sum_{j=0}^k \left[ (2j+2) x_{j+1} - 2j x_j \right] \\
    \iff & (k+2) y_{k+1} + (k+1) y_k - y_0 = (2k+2) x_{k+1} \\
    \iff & (k+2) y_{k+1} = y_0 + (k+1) (2x_{k+1} - y_k) = y_0 + (k+1) (2\JA y_k - y_k) \\
    \iff & y_{k+1} = \frac{1}{k+2} y_0 + \frac{k+1}{k+2} (2\JA - \opI)(y_k) = \frac{k+1}{k+2} \opT y_k + \frac{1}{k+2} y_0 .
\end{align*}
The last line is precisely~\eqref{eqn:appendix-Halpern-anchor-form}.

\paragraph{Form 1 $\implies$ Form 3.}
Observe that $y_1 = \frac{1}{2} \opT y_0 + \frac{1}{2} y_0 = y_0 - \frac{1}{2} (y_0 - \opT y_0)$, so we have $h_{1,1} = \frac{1}{2}$, which is agrees with~\eqref{eqn:appendix-Halpern-H-matrix}.
We use induction on $k$: Let $k \ge 1$ and suppose that~\eqref{eqn:appendix-Halpern-anchor-form} satisfies the representation~\eqref{eqn:appendix-Halpern-H-matrix} up to $y_0,\dots,y_k$.
Then
\begin{align}
    y_{k+1} & = \frac{k+1}{k+2} \opT y_k + \frac{1}{k+2} y_0 \nonumber \\
    & = y_k - \frac{k+1}{k+2} (y_k - \opT y_k) + \frac{1}{k+2} (y_0 - y_k) , \label{eqn:appendix-algorithm-equivalence-Halpern-H-matrix-induction}
\end{align}
and
\begin{align}
    y_0 - y_k & = \sum_{i=1}^{k} (y_{i-1} - y_i) \nonumber \\
    & = \sum_{i=1}^k \sum_{j=0}^{i-1} h_{i,j+1} (y_j - \opT y_j) \nonumber \\
    & = \sum_{i=1}^k \left( \sum_{j=0}^{i-2} -\frac{j+1}{i(i+1)} (y_j - \opT y_j) + \frac{i}{i+1} (y_{i-1} - \opT y_{i-1}) \right) \nonumber \\
    & = \sum_{i=0}^{k-1} \left( \sum_{j=0}^{i-1} -\frac{j+1}{(i+1)(i+2)} (y_j - \opT y_j) \right) + \sum_{j=0}^{k-1} \frac{j+1}{j+2} (y_j - \opT y_j) \nonumber \\
    & = \frac{k}{k+1} (y_{k-1} - \opT y_{k-1}) + \sum_{j=0}^{k-2} \left( \frac{j+1}{j+2} - \sum_{i=j+1}^{k-1} \frac{j+1}{(i+1)(i+2)} \right) (y_j - \opT y_j) \label{eqn:appendix-algorithm-equivalence-Halpern-H-matrix-induction-y0-yk}
\end{align}
where we change the order of double summation to obtain the last identity.
Applying the following formula
\begin{align*}
    \sum_{i=j+1}^{k-1} \frac{j+1}{(i+1)(i+2)} = (j+1) \sum_{i=j+1}^{k-1} \left( \frac{1}{i+1} - \frac{1}{i+2} \right) = (j+1) \left( \frac{1}{j+2} - \frac{1}{k+1} \right)
\end{align*}
to~\eqref{eqn:appendix-algorithm-equivalence-Halpern-H-matrix-induction-y0-yk} gives
\begin{align*}
    y_0 - y_k & = \frac{k}{k+1} (y_{k-1} - \opT y_{k-1}) + \sum_{j=0}^{k-2} \frac{j+1}{k+1} (y_j - \opT y_j) \\
    & = \sum_{j=0}^{k-1} \frac{j+1}{k+1} (y_j - \opT y_j) .
\end{align*}
Finally, plugging the last expression into~\eqref{eqn:appendix-algorithm-equivalence-Halpern-H-matrix-induction}, we obtain
\begin{align*}
    y_{k+1} = y_k - \frac{k+1}{k+2} (y_k - \opT y_k) + \sum_{j=0}^{k-1} \frac{j+1}{(k+1)(k+2)} (y_j - \opT y_j) ,
\end{align*}
so $h_{k+1,j+1} = -\frac{j+1}{(k+1)(k+2)}$ for $j=0,\dots,k-1$ and $h_{k+1,k+1} = \frac{k+1}{k+2}$, completing the induction.

\paragraph{Form 3 $\implies$ Form 4.} We have
\begin{align*}
    y_1 & = \frac{1}{2} y_0 + \frac{1}{2} \opT y_0 = \frac{1}{2} y_0 + \frac{1}{2} (2 x_1 - y_0) = x_1 ,
\end{align*}
which agrees with~\eqref{eqn:appendix-appm-definition} with $k=0$. 

Now let $k \ge 1$. 
Let $\tilA x_{j+1} = y_j - x_{j+1}$ for $j=0,1,\dots$. 
Then $y_j - \opT y_j = y_j - (2x_{j+1} - y_j) = 2 \tilA x_{j+1}$, and
\begin{align}
    y_{k+1} & = y_k - \sum_{j=0}^k 2 h_{k+1,j+1} \tilA x_{j+1} \nonumber \\
    & = y_k + \sum_{j=0}^{k-1} \frac{2(j+1)}{(k+1)(k+2)} \tilA x_{j+1} - \frac{2(k+1)}{k+2} \tilA x_{k+1} , \label{eqn:appendix-algorithm-equivalence-Halpern-H-matrix-expanded}
\end{align}
and
\begin{align*}
    x_{k+1} - x_k & = (y_k - \tilA x_{k+1}) - (y_{k-1} - \tilA x_k) \\
    & = y_k - y_{k-1} - \tilA x_{k+1} + \tilA x_k \\
    & = -\sum_{j=0}^{k-1} 2h_{k,j+1} \tilA x_{j+1} - \tilA x_{k+1} + \tilA x_k \\
    & = \sum_{j=0}^{k-2} \frac{2(j+1)}{k(k+1)} \tilA x_{j+1} + \left( 1 - \frac{2k}{k+1} \right) \tilA x_k - \tilA x_{k+1} \\
    & = \sum_{j=0}^{k-2} \frac{2(j+1)}{k(k+1)} \tilA x_{j+1} - \frac{k-1}{k+1} \tilA x_k - \tilA x_{k+1} .
\end{align*}
Thus,
\begin{align*}
    & x_{k+1} + \frac{k}{k+2} (x_{k+1} - x_k) - \frac{k}{k+2} (x_k - y_{k-1}) \\
    & = y_k - \tilA x_{k+1} + \frac{k}{k+2} \left( \sum_{j=0}^{k-2} \frac{2(j+1)}{k(k+1)} \tilA x_{j+1} - \frac{k-1}{k+1} \tilA x_k - \tilA x_{k+1} \right) + \frac{k}{k+2} \tilA x_k \\
    & = y_k + \sum_{j=0}^{k-2} \frac{2(j+1)}{(k+1)(k+2)} \tilA x_{j+1} + \left( \frac{k}{k+2} - \frac{k(k-1)}{(k+1)(k+2)} \right) \tilA x_k - \frac{2k+2}{k+2} \tilA x_{k+1} \\
    & = y_k + \sum_{j=0}^{k-1} \frac{2(j+1)}{(k+1)(k+2)} \tilA x_{j+1} - \frac{2(k+1)}{k+2} \tilA x_{k+1} \\
    & = y_{k+1} 
\end{align*}
and the proof is complete (the last line follows from~\eqref{eqn:appendix-algorithm-equivalence-Halpern-H-matrix-expanded}).

\subsection{\ref{alg:dual_halpern}}
\textbf{Form 1} (Momentum form).
\begin{align*}
    y_{k+1} = y_k + \frac{N-k-1}{N-k} \left( \opT y_k - \opT y_{k-1} \right) 
\end{align*}
for $k=0,1,\dots,N-2$, where $\opT y_{-1} = y_0$.

\textbf{Form 2} ($z$-form).
\begin{align}
\label{eqn:appendix-dual-halpern-z-form}
\begin{split}
    z_{k+1} & = \frac{N-k-1}{N-k} z_k - \frac{1}{N-k} \left( y_k - \opT y_k \right) \\
    y_{k+1} & = \opT y_k - z_{k+1}
\end{split}
\end{align}
for $k=0,1,\dots,N-2$, where $z_0 = 0$.

\textbf{Form 3} (H-matrix form).
\begin{align*}
    y_{k+1} & = y_k  - \sum_{j=0}^{k} h_{k+1,j+1}(y_j - \opT y_j) \\
    & = y_k  - \sum_{j=0}^{k} 2 h_{k+1,j+1} \tilA x_{j+1}   
\end{align*}
for $k=0,1,\dots,N-2$, where
\begin{align}
     h_{k,j} =
    \begin{cases}
        -\frac{N-k}{(N-j)(N-j+1)} & \text{if } j < k  , \\
        \frac{N-k}{N-k+1}         & \text{if } j = k .
    \end{cases}
    \label{eqn:appendix-dual-Halpern-H-matrix}
\end{align}

\textbf{Form 4} (Proximal form of~\cref{subsection:proximal-forms}). 
\begin{align}
\begin{split}
    x_{k+1} & = \JA (y_k) \\
    y_{k+1} & = x_{k+1} + \frac{N-k-1}{N-k} (x_{k+1} - x_k)  - \frac{N-k-1}{N-k} (x_k - y_{k-1}) - \frac{1}{N-k} (x_{k+1} - y_k) 
\end{split}
\label{eqn:appendix-dual-Halpern-APPM-form}
\end{align}
for $k=0,1,\dots,N-2$,
where $x_0 = y_0$ and $\opA = 2(\opT + \opI)^{-1} - \opI \iff \opT = 2\JA - \opI$.

\paragraph{Form 1 $\implies$ Form 2.}
Let $z_0 = 0$ and $z_{k} = \opT y_{k-1} - y_k$ for $k=1,\dots,N-1$, so that the second line of~\eqref{eqn:appendix-dual-halpern-z-form} holds by definition.
Following establishes the recursion for $z_k$, which is the first line of~\eqref{eqn:appendix-dual-halpern-z-form}:
\begin{align*}
    z_{k+1} &= \opT y_{k}  - y_{k+1} \\
    &= \opT y_k - \left(y_k + \frac{N-k-1}{N-k}(\opT y_k- \opT y_{k-1}) \right) \\
    &= \opT y_k - \left(y_k + \frac{N-k-1}{N-k}(\opT y_k- z_k-y_k) \right) \\
    &= \frac{N-k-1}{N-k} z_k - \frac{1}{N-k} \left( y_k - \opT y_k \right).
\end{align*}

\paragraph{Form 2 $\implies$ Form 3.}
The first line of~\eqref{eqn:appendix-dual-halpern-z-form} is equivalent to
\begin{align*}
    z_{j+1} = \frac{N-j-1}{N-j}z_j - \frac{2}{N-j} \tilA x_{j+1}.
\end{align*}
Dividing both sides by $N-j-1$ and rearranging, we obtain
\begin{align*}
    \frac{1}{N-j-1} z_{j+1} - \frac{1}{N-j}z_j =  - \frac{2}{(N-j)(N-j-1)} \tilA x_{j+1}.
\end{align*}
Summing this up from $j=0$ to $k$ and multiplying $N-k-1$ to the both sides we have
\begin{align*}
    z_{k+1} = -(N-k-1) \sum_{j=0}^{k} \frac{2}{(N-j)(N-j-1)} \tilA x_{j+1}
\end{align*}
(note that $z_0 = 0$). 
Next, we substitute the last expression into the second line of~\eqref{eqn:appendix-dual-halpern-z-form}:
\begin{align*}
    y_{k+1} &= y_k  + (\opT y_k -y_k)- z_{k+1} \\
    &= y_k - 2\tilA x_{k+1} + (N-k-1) \sum_{j=0}^{k} \frac{2}{(N-j)(N-j-1)} \tilA x_{j+1} \\
    &=y_k - \frac{2(N-k-1)}{N-k} \tilA x_{k+1} + \sum_{j=1}^{k} \frac{2(N-k-1)}{(N-j+1)(N-j)}\tilA x_j .
\end{align*}
This shows that $h_{k+1,j} = -\frac{(N-k-1)}{(N-j)(N-j+1)}$ for $j=1,\dots,k$ and $h_{k+1,k+1} = \frac{N-k-1}{N-k}$, which agrees with~\eqref{eqn:appendix-dual-Halpern-H-matrix}.

\paragraph{Form 3 $\implies$ Form 4.}
From the definition of Form 3,
\begin{align}
\begin{split} \label{appendix_eq_pf_1}
    y_{k+1} &= y_k - \frac{2(N-k-1)}{N-k} \tilA x_{k+1} + \sum_{j=0}^{k-1} \frac{2(N-k-1)}{(N-j-1)(N-j)}\tilA x_{j+1}.
\end{split}
\end{align}
Putting $k-1$ in place of $k$, we have
\begin{align*}
    y_{k} &= y_{k-1} - \frac{2(N-k)}{N-k+1} \tilA x_k + \sum_{j=0}^{k-2} \frac{2(N-k)}{(N-j-1)(N-j)} \tilA x_{j+1} \\
    &= y_{k-1} - 2\tilA x_{k} + \sum_{j=0}^{k-1} \frac{2(N-k)}{(N-j-1)(N-j)}\tilA x_{j+1}.
\end{align*}
Using the last equation, we replace the summation within~\eqref{appendix_eq_pf_1}:
\begin{align*}
    y_{k+1} 
    &=  y_k - \frac{2(N-k-1)}{N-k}\tilA x_{k+1} + \frac{N-k-1}{N-k}\left(y_k-y_{k-1}+2\tilA x_{k} \right) .
\end{align*}
Finally, substitute $\tilA x_{k+1} = y_k - x_{k+1}$, $\tilA x_k = y_{k-1} - x_k$ and rearrange to obtain Form 4:
\begin{align*}
    y_{k+1} &= y_k - \frac{2(N-k-1)}{N-k}\left(y_k -x_{k+1} \right) + \frac{N-k-1}{N-k}\left(y_k-y_{k-1}+2(y_{k-1}-x_k) \right) \\
    &=  x_{k+1} + \frac{N-k-1}{N-k} (x_{k+1} - x_k)  - \frac{N-k-1}{N-k} (x_k - y_{k-1}) - \frac{1}{N-k} (x_{k+1} - y_k) .
\end{align*}

\paragraph{Form 4 $\implies$ Form 1.}
Simply substitute $x_{k+1}=\frac{1}{2}(y_k + \opT y_k)$ into the second line of~\eqref{eqn:appendix-dual-Halpern-APPM-form} and rearrange. In detail:
\begin{align*}
    y_{k+1} & = x_{k+1} + \frac{N-k-1}{N-k} (x_{k+1} - x_k) - \frac{N-k-1}{N-k} (x_k - y_{k-1}) - \frac{1}{N-k} (x_{k+1} - y_k)  \\
    &= \frac{2(N-k-1)}{N-k}x_{k+1} - \frac{2(N-k-1)}{N-k}x_k +\frac{N-k-1}{N-k}y_{k-1} + \frac{1}{N-k}y_k \\
    &=\frac{N-k-1}{N-k}\left(y_k + \opT y_k \right)-\frac{N-k-1}{N-k}\left(y_{k-1} + \opT y_{k-1} \right) + \frac{N-k-1}{N-k} y_{k-1} + \frac{1}{N-k}y_k \\
    &=  y_k + \frac{N-k-1}{N-k} \left( \opT y_k - \opT y_{k-1} \right).
\end{align*}

\subsection{\ref{alg:feg}}
\textbf{Form 1} (Original form).
\begin{align}
\begin{split}
    x_{k+\hf} & = x_k + \frac{1}{k+1} (x_0 - x_k) - \frac{k}{k+1} \alpha \sop{x_k}\\
    x_{k+1} & = x_k + \frac{1}{k+1} (x_0 - x_k) - \alpha \sop{x_{k+\hf}}
\end{split}
\label{eqn:appendix-FEG-original-form}
\end{align}

\textbf{Form 2} (H-matrix form).
\begin{align*}
    x_{(\ell+1)/2} = x_{\ell/2} - \frac{1}{L} \sum_{i=0}^{\ell} h_{(\ell+1)/2, i/2} \sop{x_{i/2}}
\end{align*} 
where for $k=0,1,\dots,N-1$,
\begin{align}
   \frac{1}{\alpha L} h_{(\ell+1)/2,i/2} = \begin{cases}
      \frac{k}{k+1}  & \text{if } \ell=2k,\, i=2k \\
      -\frac{j+1}{k(k+1)} & \text{if } \ell=2k,\, i=2j+1, \, j=0,\dots,k-1 \\
      0 & \text{if } \ell=2k,\, i=2j,\, j=0,\dots,k-1 \\
      1 & \text{if } \ell=2k+1,\, i=2k+1 \\
      -\frac{k}{k+1} & \text{if } \ell=2k+1,\, i=2k \\
      0 & \text{if } \ell=2k+1,\, i=0,\dots,2k-1.
    \end{cases}
    \label{eqn:appendix-FEG-H-matrix-form}
\end{align}

It suffices to inductively check that the two forms indicate the same update rule of generating $x_{k+1/2}, x_{k+1}$ provided that they are equivalent for all $x_0,\dots,x_k$.
They are clearly equivalent for $k=0$.
Assume that the equivalence holds for $j=0,\dots,k$.
Now we show that the update rule by Form~1 agrees with~\eqref{eqn:appendix-FEG-H-matrix-form}.
Multiplying $k+1$ to the second line of~\eqref{eqn:appendix-FEG-original-form} and switching the index from $k$ to $j$, we have
\begin{align*}
    (j+1) x_{j+1} = j x_j + x_0 - \alpha (j+1) \sop{x_{j+\hf}}.
\end{align*}
Summing this up from $j=0$ to $k$ and dividing by $k+1$ we have
\begin{align}
\label{eqn::appendix_pf_2}
    x_{k+1} - x_0 = - \frac{1}{k+1} \alpha \sum_{j=0}^{k} (j+1) \sop{x_{i+\hf}}.
\end{align}
Also, by subtracting the first line of~\eqref{eqn:appendix-FEG-original-form} from the second line of~\eqref{eqn:appendix-FEG-original-form}, we obtain
\begin{align*} 
    x_{k+1} = x_{k+1/2} - \alpha \sop{x_{k+\hf}} + \frac{k}{k+1}\alpha\sop{x_k}.
\end{align*}
Then, by applying \eqref{eqn::appendix_pf_2} with $k$ and $k \leftarrow k-1$ we can write
\begin{align*}
    x_{k+1/2} - x_k & = x_{k+1} +\alpha\sop{x_{k+\hf}} - \frac{k}{k+1}\alpha\sop{x_k} - x_k \\
    & = (x_{k+1}-x_k) +\alpha\sop{x_{k+\hf}} - \frac{k}{k+1}\alpha\sop{x_k} \\
    & = (x_{k+1}-x_0)-(x_k-x_0) +\alpha\sop{x_{k+\hf}} - \frac{k}{k+1} \alpha\sop{x_k} \\
    & = \left( - \frac{1}{k+1} \alpha \sum_{j=0}^{k} (j+1) \sop{x_{i+\hf}} \right) - \left( - \frac{1}{k} \alpha \sum_{j=0}^{k-1} (j+1) \sop{x_{i+\hf}} \right) \\
    & \quad\quad + \alpha\sop{x_{k+\hf}} - \frac{k}{k+1} \alpha\sop{x_k}   \\ 
    & = -\frac{k}{k+1} \alpha\sop{x_k} + \sum_{j=0}^{k-1} \left( \frac{j+1}{k} - \frac{j+1}{k+1} \right) \alpha \sop{x_{j+1/2}} \\
    & = -\frac{k}{k+1} \alpha\sop{x_k} + \sum_{j=0}^{k-1} \frac{j+1}{k(k+1)} \alpha\sop{x_{j+1/2}}.
\end{align*}
This shows that
\begin{align*} 
    h_{k+1/2,k} = \frac{k}{k+1}\alpha L, \quad h_{k+1/2,j+1/2} = -\frac{j+1}{k(k+1)}\alpha L, \quad h_{k+1/2, j} = 0 \quad (j=0,\dots,k-1) .
\end{align*}
Next, because
\begin{align*} 
    x_{k+1} &= x_{k+1/2} -\alpha \sop{x_{k+\hf}} + \frac{k}{k+1} \alpha\sop{x_k} 
\end{align*}
we obtain
\begin{align*} 
    h_{k+1,k+1/2} = \alpha L, \quad h_{k+1,k} = -\frac{k}{k+1} \alpha L, \quad h_{k+1,i/2} = 0 \quad (i=0,\dots,2k-1)
\end{align*}
which agrees with~\eqref{eqn:appendix-FEG-H-matrix-form}.

\subsection{\ref{alg:dual-feg}}

\textbf{Form 1} (Original form).
\begin{align}
\begin{split}
    x_{k+\half} & = x_k - \alpha z_k - \alpha \sop{x_k} \\ 
    x_{k+1}  & = x_{k+1/2}-\frac{N-k-1}{N-k}\alpha\left( \sop{x_{k+1/2}} - \sop{x_k} \right) \\
    z_{k+1} & = \frac{N-k-1}{N-k} z_k - \frac{1}{N-k} \sop{x_{k+\hf}}
\end{split}
\label{eqn:appendix-dual-FEG-original-form}
\end{align}
for $k=0,1,\dots,N-1$, where $z_0 = 0$.

\textbf{Form 2} (H-matrix form).
\begin{align*}
    x_{(\ell+1)/2} = x_{\ell/2} - \frac{1}{L} \sum_{i=0}^{\ell} h_{(\ell+1)/2, i/2} \sop{x_{i/2}}
\end{align*} 
where
\begin{align}
   \frac{1}{\alpha L} h_{(\ell+1)/2,i/2} = \begin{cases}
       1 & \text{if } \ell=2k,\, i=2k \\
      -\frac{N-k}{(N-j-1)(N-j)} & \text{if } \ell=2k,\, i=2j+1, \, j=0,\dots,k-1 \\
      0  & \text{if } \ell=2k,\, i=2j,\, j=0,\dots,k-1 \\
      \frac{N-k-1}{N-k} & \text{if } \ell=2k+1,\, i=2k+1 \\
      -\frac{N-k-1}{N-k} & \text{if } \ell=2k+1,\, i=2k \\
      0 & \text{if } \ell=2k+1,\, i=0,\dots,2k-1.
    \end{cases}
    \label{eqn:appendix-dual-FEG-H-matrix-form}
\end{align}
for $k=0,1,\dots,N-1$.

As in the case of FEG, we check that update rule~\eqref{eqn:appendix-dual-FEG-original-form} of Form 1 defines the identical update rule for $x_{k+1/2}, x_{k+1}$ provided that they are equivalent for all $x_0,\dots,x_k$.
Subtracting the first line of~\eqref{eqn:appendix-dual-FEG-original-form} from the second line of~\eqref{eqn:appendix-dual-FEG-original-form} we get
\begin{align*}
    x_{k+1} = x_{k+1/2} - \frac{N-k-1}{N-k} \alpha\left(\sop{x_{k+1/2}}-\sop{x_k} \right) .
\end{align*}
Dividing the the third line of~\eqref{eqn:appendix-dual-FEG-original-form} by $N-k$ and switching the index from $k$ to $j$ we obtain
\begin{align*}
    \frac{1}{N-j}z_j & = \frac{1}{N-j+1} z_{j-1} - \frac{1}{(N-j)(N-j+1)}\sop{x_{j-1/2}} .
\end{align*}
Summing this up from $j=0$ to $k-1$ and multiplying $N-k$ to both sides gives
\begin{align*}
    z_k = -\sum_{j=0}^{k-1} \frac{N-k}{(N-j-1)(N-j)}\sop{x_{j+1/2}}
\end{align*}
(note that $z_0 = 0$).
Now substituting the above expression for $z_k$ into the first line of~\eqref{eqn:appendix-dual-FEG-original-form} gives
\begin{align*}
    x_{k+\half} & = x_k - \alpha z_k - \alpha \sop{x_k} \\
    & = x_k- \alpha \sop{x_k} + \sum_{j=0}^{k-1} \frac{N-k}{(N-j-1)(N-j)} \alpha \sop{x_{j+1/2}}
\end{align*}
and thus,
\begin{align*}
    h_{k+1/2,k} = \alpha L, \quad h_{k+1/2,j+1/2} = -\frac{N-k}{(N-j-1)(N-j)} \alpha L , \quad h_{k+1/2,j} = 0 \quad (j=0,\dots,k-1).
\end{align*}
Finally, the second line of~\eqref{eqn:appendix-dual-FEG-original-form} is
\begin{align*}
      x_{k+1} &=  x_{k+1/2} - \frac{N-k-1}{N-k} \alpha\sop{x_{k+1/2}} + \frac{N-k-1}{N-k} \alpha\sop{x_k} 
\end{align*}
which gives
\begin{align*}
    h_{k+1,k+1/2} = \frac{N-k-1}{N-k} \alpha L , \quad h_{k+1,k} = -\frac{N-k-1}{N-k} \alpha L , \quad h_{k+1,i/2} = 0 \quad (i=0,\dots,2k-1) .
\end{align*}
This is precisely~\eqref{eqn:appendix-dual-FEG-H-matrix-form}.

\subsection{Anchor ODE \eqref{ode:anchor}} 

\textbf{Form 1} (Original form).
\begin{align} \label{ode:anchor-appendix}  
    \dot{X}(t) = -\opA(X(t)) - \frac{1}{t} (X(t) - X_0)   
\end{align}
where $X(0) = X_0$.

\textbf{Form 2} (Second-order form).
\begin{align} \label{ode:anchor-2nd-appendix}
    \ddot{X}(t) + \frac{2}{t} \dot{X}(t) + \frac{1}{t}\opA(X(t)) + \frac{d}{dt} \opA(X(t)) = 0
\end{align}
where $X(0) = X_0$ and $\dot{X}(0) = -A(X_0)$.

\paragraph{Form 1 $\implies$ Form 2.}
Let $X\colon [0,\infty) \to \reals^d$ be the solution to~\eqref{ode:anchor-appendix} with initial condition $X(0)=X_0$. 
First observe that taking the limit $t\to 0^+$ in \eqref{ode:anchor-appendix} gives $\dot{X}(0) = -\opA(X_0) - \dot{X}(0)$, which is equivalent to the initial velocity condition $\dot{X}(0) = -\frac{1}{2}\opA(X_0)$ for~\eqref{ode:anchor-2nd-appendix}. 
Differentiating both sides of \eqref{ode:anchor-appendix}, we have
\begin{align*}
    \ddot{X}(t) = -\frac{d}{dt} \opA(X(t)) - \frac{1}{t} \dot{X}(t) + \frac{1}{t^2} (X(t)-X_0). 
\end{align*}
Rearranging the defining equation~\eqref{ode:anchor-appendix} gives $\frac{1}{t^2}(X(t) - X_0) = -\frac{1}{t}\dot{X}(t) - \frac{1}{t}\opA(X(t))$.
Substituting this into the last equation and reorganizing, we obtain~\eqref{ode:anchor-2nd-appendix}:
\begin{align*}
    & \ddot{X}(t) = - \frac{d}{dt} \opA(X(t)) - \frac{1}{t} \dot{X}(t) + \left(-\frac{1}{t} \dot{X}(t) - \frac{1}{t} \opA(X(t)) \right) \\
    & \iff \ddot{X}(t) + \frac{2}{t} \dot{X}(t) + \frac{1}{t}\opA(X(t)) + \frac{d}{dt} \opA(X(t)) = 0 .
\end{align*}

\paragraph{Form 2 $\implies$ Form 1.}
Suppose $X\colon [0,\infty) \to \reals^d$ is the solution to~\eqref{ode:anchor-2nd-appendix} with initial conditions $X(0) = X_0$ and $\dot{X}(0) = -\frac{1}{2}\opA(X_0)$. 
Multiplying $t$ throughout \eqref{ode:anchor-2nd-appendix}, we have
\begin{align*}
    0 
    &= t\ddot{X}(t) + 2 \dot{X}(t) + \opA(X(t)) +  t\frac{d}{dt} \opA(X(t)) 
    = \frac{d}{dt} \pr{ t \dot{X}(t) } + \dot{X}(t) + \frac{d}{dt} \pr{ t \opA(X(t)) }.
\end{align*}
Integrating both sides from $0$ to $t$ gives
\begin{align*}
    0 = t \dot{X}(t) + X(t) - X_0 + t \opA(X(t)).
\end{align*}
Dividing both sides by $t$ and reorganizing, we get \eqref{ode:anchor-appendix}.

The above result holds given a minimal assumption that $\opA$ is Lipschitz continuous (with the equalities holding for almost every $t$ if differentiability is not assumed).
For the rigorous discussion on this point, we refer the readers to~\citep[Appendix~B]{SuhParkRyu2023_continuoustime}.

\subsection{Dual-Anchor ODE \eqref{ode:dual-anchor}}
\label{subsection:appendix-equivalence-dual-anchor-ode}

\textbf{Form 1} (Original form).
\begin{align}
\begin{aligned}
    \dot{X}(t) & = - Z(t) - \opA(X(t)) \\
    \dot{Z}(t) 
    & = - \frac{1}{T-t} Z(t) - \frac{1}{T-t} \opA(X(t)) 
\end{aligned}
 \label{ode:dual-anchor-appendix}
\end{align}
where $X(0)=X_0$ and $Z(0)=0$.

\textbf{Form 2} (Second-order form).
\begin{align}
    \ddot{X}(t) + \frac{1}{T-t} \dot{X}(t) -\frac{d}{dt} \opA(X(t)) = 0
    \label{ode:dual-anchor-2nd-appendix}
\end{align}
where $X(0)=X_0$ and $\dot{X}(0)=-\opA(X_0)$.

\paragraph{Form 1 $\implies$ Form 2.}
Let $\pmat{X \\Z} \colon [0,T) \to \reals^d\times \reals^d$ be the solution to~\eqref{ode:dual-anchor-appendix} with initial conditions $X(0)=X_0$ and $Z(0)=0$. 
Plugging $t=0$ into the first line of \eqref{ode:dual-anchor} gives $\dot{X}(0) = 0 - \opA(X_0)$, which is the initial velocity condition for~\eqref{ode:dual-anchor-2nd-appendix}. 
Now observe that
\begin{align*}
    \dot{Z}(t) = \frac{1}{T-t} \left( - Z(t) - \opA(X(t)) \right) = \frac{1}{T-t} \dot{X}(t)
\end{align*}
where the last equality comes from the first line of~\eqref{ode:dual-anchor-appendix}.
Differentiating the first line of~\eqref{ode:dual-anchor-appendix} and plugging in the above identity we obtain Form 2:
\begin{align*}
    0 
    = \ddot{X}(t) + \dot{Z}(t) + \frac{d}{dt} \opA(X(t)) 
    = \ddot{X}(t) + \frac{1}{T-t} \dot{X}(t) -\frac{d}{dt} \opA(X(t)).
\end{align*}

\paragraph{Form 2 $\implies$ Form 1.}
Suppose $X\colon [0,T) \to \reals^d$ is the solution to~\eqref{ode:dual-anchor-2nd-appendix} with initial conditions $X(0) = X_0$ and $\dot{X}(0) = -\opA(X_0)$. 
Define $Z\colon[0,T) \to \reals^d$ by $Z(t) = -\dot{X}(t) - \opA(X(t))$. 
Then $\dot{X}(t) = -Z(t) - \opA(X(t))$ by definition (this is the first line of~\eqref{ode:dual-anchor-appendix}).
Also note that $Z(0) = -\dot{X}(0) - \opA(X_0) = 0$, which is the $Z$-initial condition for~\eqref{ode:dual-anchor-appendix}.
Now differentiating $Z$ we have
\begin{align*}
    \dot{Z}(t) = - \ddot{X}(t) - \frac{d}{dt} \opA(X(t)) = \frac{1}{T-t} \dot{X}(t) = -\frac{1}{T-t} \pr{ Z(t) + \opA(X(t)) }. 
\end{align*}
The second equality directly follows from the defining equation~\eqref{ode:dual-anchor-2nd-appendix}. 
This shows that $\pmat{X \\Z}$ is the solution of \eqref{ode:dual-anchor-appendix}.

The above result holds given a minimal assumption that $\opA$ is Lipschitz continuous (with the equalities holding for almost every $t$ if differentiability is not assumed).
For the rigorous discussion on existence, uniqueness of the solutions and almost everywhere differntiability of the involved quantities, we refer the readers to~\cref{section:appendix-continuous}.

\newpage
\section{Lyapunov analysis of \ref{alg:dual_halpern}}
\label{section:appendix-Dual-Halpern-Lyapunov}

Recall the following form of \ref{alg:dual_halpern}:
\begin{align}
    z_{k+1} & = \frac{N-k-1}{N-k} z_k - \frac{1}{N-k} \left( y_k - \opT y_k \right) \label{eqn:appendix-dual-Halpern-z-definition} \\
    y_{k+1} & = \opT y_k - z_{k+1} . \label{eqn:appendix-dual-Halpern-y-definition}
\end{align}
Substituting \eqref{eqn:appendix-dual-Halpern-z-definition} into \eqref{eqn:appendix-dual-Halpern-y-definition} we get
\begin{align}
    y_{k+1} & = \opT y_k - \frac{N-k-1}{N-k} z_k + \frac{1}{N-k} \left( y_k - \opT y_k \right) \nonumber \\
    & = \frac{1}{N-k} y_k + \frac{N-k-1}{N-k} \opT y_k - \frac{N-k-1}{N-k} z_k \label{eqn:appendix-dual-Halpern-y-alternative-form}
\end{align}
With the substitution $\opT = 2\JA - \opI$ and $x_{k+1} = \JA y_k$ we can write \eqref{eqn:appendix-dual-Halpern-z-definition} and \eqref{eqn:appendix-dual-Halpern-y-alternative-form} as
\begin{align}
    z_{k+1} & = \frac{N-k-1}{N-k} z_k - \frac{2}{N-k} \tilA x_{k+1} \label{eqn:appendix-dual-Halpern-z-update-monotone-operator} \\
    y_{k+1} & = \frac{1}{N-k} y_k + \frac{N-k-1}{N-k} \left( y_k - 2 \tilA x_{k+1} \right) - \frac{N-k-1}{N-k} z_k \nonumber \\
    & = y_k - \frac{2(N-k-1)}{N-k} \tilA x_{k+1} - \frac{N-k-1}{N-k} z_k \label{eqn:appendix-dual-Halpern-y-update-monotone-operator}
\end{align}
where we have used $\opT y_k = 2 \JA y_k - y_k = 2x_{k+1} - y_k = y_k - 2(y_k - x_{k+1}) = y_k - 2\tilA x_{k+1}$.
For simplicity, write $g_j = \tilA x_j$ for $j=1,\dots,N$.
To complete the proof, it remains to show that for
\begin{align*}
    V_k = \underbrace{-\frac{N-k-1}{N-k} \sqnorm{z_k + 2 g_N}}_{:=V_k^{(1)}} + \underbrace{\frac{2}{N-k} \inprod{z_k + 2 g_N}{y_k - y_{N-1}}}_{:=V_k^{(2)}}
\end{align*}
the following holds:
\begin{align*}
    V_k - V_{k+1} = \frac{4}{(N-k)(N-k-1)} \inprod{x_{k+1} - x_N}{g_{k+1} - g_N} .
\end{align*} 
Rewriting the right hand side, we have 
\begin{align*}
    & \frac{4}{(N-k)(N-k-1)} \inprod{x_{k+1} - x_N}{g_{k+1} - g_N} \\
    &= \frac{4}{(N-k)(N-k-1)} \inprod{y_k - y_{N-1} - (g_{k+1} - g_N)}{g_{k+1} - g_N} \\
    &= \frac{2}{N-k-1} \inprod{y_k - y_{N-1}}{ \frac{2}{N-k} g_{k+1} - \frac{2}{N-k} g_N} - \frac{4}{(N-k)(N-k-1)} \sqnorm{g_{k+1} - g_N} \\  
    &\!\stackrel{\eqref{eqn:appendix-dual-Halpern-z-update-monotone-operator}}{=}  
    \frac{2}{N-k-1} \inprod{y_k - y_{N-1}}{ \frac{N-k-1}{N-k}(z_k+ 2g_N) - (z_{k+1} + 2g_N)} - \frac{4}{(N-k)(N-k-1)} \sqnorm{g_{k+1} - g_N} \\
    &= \frac{2}{N-k} \inprod{y_k - y_{N-1}}{ z_k+ 2g_N} 
        - \frac{2}{N-k-1} \inprod{y_k - y_{N-1}}{ z_{k+1} + 2g_N} 
        - \frac{4}{(N-k)(N-k-1)} \sqnorm{g_{k+1} - g_N} \\
    &\!\stackrel{\eqref{eqn:appendix-dual-Halpern-y-update-monotone-operator}}{=} V_k^{(2)}
        - \frac{2}{N-k-1} \inprod{y_{k+1} - y_{N-1} + \frac{2(N-k-1)}{N-k} g_{k+1} + \frac{N-k-1}{N-k} z_{k} }{ z_{k+1} + 2g_N} 
        - \frac{4}{(N-k)(N-k-1)} \sqnorm{g_{k+1} - g_N} \\
    &= V_k^{(2)} - V_{k+1}^{(2)} \underbrace{ - \frac{2}{N-k} \inprod{ 2 g_{k+1} + z_k }{z_{k+1} + 2g_N} - \frac{1}{(N-k)(N-k-1)} \sqnorm{2g_{k+1} - 2g_N}}_{:=R_k}
\end{align*}
and the proof is done once we show $R_k = V_k^{(1)} - V_{k+1}^{(1)}$.
From \eqref{eqn:appendix-dual-Halpern-z-update-monotone-operator} we have $2 g_{k+1} = (N-k-1)z_k -  (N-k)z_{k+1}$,
and plugging this into $R_k$ we obtain
\begin{align*}
    R_k
    &= - 2\inprod{ z_{k+1} - z_k}{z_{k+1} + 2g_N} - \frac{1}{(N-k)(N-k-1)} \sqnorm{ (N-k-1)z_k -  (N-k)z_{k+1} - 2g_N} \\
    &= 2 \inprod{ (z_{k+1}+2 g_N)-(z_k+2 g_N) }{z_{k+1} + 2g_N}  \\ &\quad 
        - \frac{1}{(N-k)(N-k-1)} \sqnorm{ (N-k-1) (z_k+2 g_N)-(N-k) (z_{k+1}+2 g_N) } \\
    &= 2 \sqnorm{ z_{k+1}+2 g_N } - \frac{N-k-1}{N-k} \sqnorm{z_k+2 g_N} - \frac{N-k}{N-k-1} \sqnorm{ z_{k+1}+2 g_N } \\
    &= -\frac{N-k-1}{N-k} \sqnorm{z_k + 2 g_N} + \frac{N-k-2}{N-k-1} \sqnorm{z_{k+1} + 2 g_N} \\
    &= V_k^{(1)} - V_{k+1}^{(1)},
\end{align*}
which proves that indeed, $R_k = V_k^{(1)} - V_{k+1}^{(1)}$.

\newpage
\section{Fixed-point H-duality theory}
\label{sec::appendix_H_duality_pf}

\subsection{The precise statement of H-duality theorem}
\label{subsection:appendix-H-duality-statement}
To state the theorem, we first need to set up some notations and concepts. 

\paragraph{Primal Lyapunov structure.}
Consider the H-matrix representation of an algorithm
\begin{align}
\label{eqn:FPI-H-matrix-appendix}
    y_{k+1} & = y_k - \sum_{j=0}^{k} h_{k+1,j+1} (y_j - \opT y_j) = y_k - \sum_{j=0}^{k} 2h_{k+1,j+1} \tilA x_{j+1}
\end{align}
for $k=0,1,\dots,N-2$, where $x_{k+1} = \JA(y_k)$ for $k=0,1,\dots,N-1$.
Consider a convergence proof for~\eqref{eqn:FPI-H-matrix-appendix} with respect to the performance measure $\norm{y_{N-1}-\opT y_{N-1}}^2 = 4\norm{\tilde{\opA} x_N}^2$, structured in the following way: 
Take a sequence $\{u_j\}_{j=1}^{N-1}$ of positive numbers, and define the sequence $\{U_j\}_{j=1}^{N}$ by $U_1 = 0$ and 
\begin{align*}
    U_{j+1} = U_j  - u_j\inprod{x_{j+1} - x_j}{\tilde{\opA} x_{j+1} - \tilde{\opA} x_j}
\end{align*}
for $j=1,\dots,N-1$. 
As $\opA$ is monotone, $\{U_j\}_{j=1}^{N}$ is a nonincreasing sequence. 
To clarify, while we restrict the proof structure to use a specific combination of monotonicity inequalities, we let $\{u_j\}_{j=1}^{N-1}$ as free variables which can be appropriately chosen to make the convergence analysis work.

Assume that we can show that for some $\tau_U > 0$,
\begin{align} \label{eqn::primal_cond_FPI} \tag{C1}
    \tau_U \norm{\tilA x_N}^2 + \inprod{\tilde{\opA}x_N}{x_N-y_0} \leq U_N 
\end{align}
holds for arbitrary $\tilA x_1, \dots, \tilA x_N$. 
To put this precisely, $(\text{RHS}) - (\text{LHS})$ in \eqref{eqn::primal_cond_FPI} is a ``vector quadratic form'' of $\{\tilde{\opA}x_j\}_{j=1}^N$, i.e., a function $\cQ \colon \prod_{j=1}^N \reals^d \to \reals$ of the form 
\begin{align*}
    \cQ \left( g_1, \dots, g_N \right)  = \sum_{i=1}^N \sum_{j=1}^N s_{i,j} \inprod{g_i}{g_j}  = \mathrm{Trace} \left( \vG^\intercal \vS_N \vG \right)
\end{align*}
where $\vG = \begin{bmatrix} g_1 & \cdots & g_N \end{bmatrix}, \vS_N = \left( s_{i,j} \right)_{1\le i,j \le N} \in \reals^{N\times N}$.  
Here $s_{i,j} = s_{i,j} \left( H, \{u_j\}_{j=1}^{N-1}, \tau_U \right)$ has a hidden dependency on the entries of H-matrix, $u_j$'s and $\tau_U$.
If $\cQ (g_1,\dots,g_N) \ge 0$ for any $g_1,\dots,g_N \in \reals^d$ (which is equivalent to $\vS_N \succeq 0$) then we informally say $\cQ \ge 0 \iff $ \eqref{eqn::primal_cond_FPI}.
If that is the case, we can establish
\begin{align*}
   \tau_U \norm{\opA x_N}^2 + \inprod{\tilde{\opA}x_N}{x_N-y_0} \leq  U_N \leq \dots \leq U_1 = 0.
\end{align*}
By Lemma~\ref{lemma:convergence-proof-last-step}, this implies $\norm{\tilde{\opA} x_N}^2 \leq \frac{\norm{y_0-x_\star}^2}{\tau_U^2}$. 
\ref{alg:halpern} is an example where this holds; selecting $\tau_U = N$ and $u_j = \frac{j(j+1)}{N}$ for $j=1,\dots,N-1$ ensures \eqref{eqn::primal_cond_FPI}, leading to the final convergence rate $\norm{\tilde{\opA} x_N}^2 \leq \frac{\norm{y_0-x_\star}^2}{N^2}$.
We refer to this proof strategy as \emph{primal Lyapunov proof}.

\paragraph{Dual Lyapunov structure.} 
Consider a convergence proof for~\eqref{eqn:FPI-H-matrix-appendix} with respect to the same performance measure $\norm{y_{N-1}-\opT y_{N-1}}^2 = 4\norm{\tilde{\opA} x_N}^2$, but now using a positive sequence $\{v_j\}_{j=1}^{N-1}$ and $\{V_j\}_{j=0}^{N-1}$, defined by $V_{N-1} = 0$ and the backward recursion
\begin{align*}
    V_{j} = V_{j+1}  + v_{j+1}\inprod{x_N - x_{j+1}}{\tilde{\opA} x_N - \tilde{\opA} x_{j+1}}
\end{align*}
for $j=0,\dots,N-2$ (note that we are using a different set of inequalities than $U_j$).
Then $\{V_j\}_{j=0}^{N-1}$ is nonincreasing because $\opA$ is monotone.
As before, $\{v_j\}_{j=1}^{N-1}$ can be seen as free variables.
Assume that we can show that for some $\tau_V > 0$, 
\begin{align} \label{eqn::dual_cond_FPI} \tag{C2}
  V_0 \leq  -\tau_V \norm{\tilA x_N}^2 - \inprod{\tilde{\opA}x_N}{x_N-y_0}
\end{align}
holds  
(in the sense of vector quadratic form with coefficients on $H$, $v_j$'s and $\tau_V$). 
Then
\begin{align*}
  0 = V_{N-1} \leq \dots \leq V_0 \leq -\tau_V \norm{\tilA x_N}^2 - \inprod{\tilde{\opA}x_N}{x_N-y_0} .
\end{align*}
By Lemma~\ref{lemma:convergence-proof-last-step}, this implies $\norm{\tilde{\opA} x_N}^2 \leq \frac{\norm{y_0-x_\star}^2}{\tau_V^2}$. 
\ref{alg:dual_halpern} is an example satisfying this; selecting $\tau_V = N$ and $v_j = \frac{N}{(N-j)(N-j+1)}$ for $j=1,\dots,N-1$ ensures \eqref{eqn::dual_cond_FPI}, which implies $\norm{\tilde{\opA} x_N}^2 \leq \frac{\norm{y_0-x_\star}^2}{N^2}$.
We refer to this proof strategy as \emph{dual Lyapunov proof}.

Finally, we are ready to state the H-duality theorem,  
which shows that the primal Lyapunov proof for a ``primal'' algorithm can be transformed into a dual Lyapunov proof establishing the same rate for its H-dual algorithm.
\begin{theorem}[H-duality] \label{thm::H-duality_FPI} 
Consider sequences of positive real numbers $\{u_j\}_{j=1}^{N-1}$ and $\{v_j\}_{j=1}^{N-1}$ related through $v_j=\frac{1}{u_{N-j}}$ for $j=1,\dots,N-1$. 
Let $H\in \mathbb{R}^{(N-1)\times(N-1)}$ be a lower triangular matrix and $\tau > 0$. Then
\begin{align*}
\left[\text{\eqref{eqn::primal_cond_FPI} is satisfied with $H, \{u_j\}_{j=1}^{N-1}$ and $\tau_U = \tau$}\right] 
\quad\Leftrightarrow\quad\left[\text{\eqref{eqn::dual_cond_FPI} is satisfied with $H^{\at}, \{v_j\}_{j=1}^{N-1}$ and $\tau_V = \tau$}\right].    
\end{align*}
\end{theorem}

\subsection{Proof of Theorem~\ref{thm::H-duality_FPI}}
 
Define the two vector quadratic forms $\cS$ and $\cT$ by
\begin{align*}
    \cS(\tilA x_1, \dots, \tilA x_N) & = U_N \pr{H, \{u_j\}_{j=1}^{N-1}} - \tau \sqnorm{\tilA x_N} - \inprod{\tilA x_N}{x_N \pr{H} - y_0} \\
    \cT(\tilA x_1, \dots, \tilA x_N) & = -V_0 \pr{H^\at, \{v_j\}_{j=1}^{N-1}} - \tau \sqnorm{\tilA x_N} - \inprod{\tilA x_N}{x_N \pr{H^\at} - y_0} ,
\end{align*}
where we write $U_N \pr{H, \{u_j\}_{j=1}^{N-1}}$ and $x_N(H)$ to specify that these quantities are defined through the rules of~\cref{subsection:appendix-H-duality-statement} while using the H-matrix $H$ to define $x_1,\dots,x_N$ (according to~\eqref{eqn:FPI-H-matrix-appendix}) and using those $x_1,\dots,x_N$ and the specified sequence $\{u_j\}_{j=1}^{N-1}$ to define $U_N$.
Similarly, $V_0 \pr{H^\at, \{v_j\}_{j=1}^{N-1}}$ and $x_N \pr{H^\at}$ specifies that they are defined using $H^\at$ as H-matrix and using the sequence $\{v_j\}_{j=1}^{N-1}$.

\paragraph{Rewriting $\cS$.} 
We first expand $\cS$ without explicitly expressing its dependency on $H$:
\begin{align}
    & \cS (\tilA x_1, \dots, \tilA x_N) \nonumber \\
    & = U_N - \tau \norm{\tilA x_N}^2 - \inprod{\tilA x_N}{x_N-y_0} \nonumber \\
    & = -\sum_{k=1}^{N-1} u_k\inprod{\tilA x_{k+1}-\tilA x_k}{x_{k+1 }-x_k} - \tau \norm{\tilA x_N}^2 - \inprod{\tilA x_N}{x_N - y_0} \nonumber \\
    & = -\tau \norm{\tilA x_N}^2 - \inprod{\tilA x_N}{y_{N-1} - \tilA x_N - y_0} - \sum_{k=1}^{N-1} u_k\inprod{\tilA x_{k+1}-\tilA x_k}{y_k - \tilA x_{k+1} - y_{k-1} + \tilA x_k} \nonumber \\
    & = -(\tau - 1) \norm{\tilA x_N}^2 + \inprod{\tilA x_N}{y_0-y_{N-1}} + \sum_{k=1}^{N-1} u_k\norm{\tilA x_{k+1}-\tilA x_k}^2 - \sum_{k=1}^{N-1} u_k \inprod{\tilA x_{k+1}-\tilA x_k}{y_k - y_{k-1}}  \label{eqn:appendix-H-duality-S-expansion-first}
\end{align}
Now rewrite $\inprod{\tilA x_N}{y_0-y_{N-1}} = -\sum_{k=1}^{N-1} \inprod{\tilA x_N}{y_k - y_{k-1}}$ and group this summation with the last summation within~\eqref{eqn:appendix-H-duality-S-expansion-first} to obtain
\begin{align}
    & \cS (\tilA x_1, \dots, \tilA x_N) \nonumber \\
    & = -(\tau - 1) \norm{\tilA x_N}^2   + \sum_{k=1}^{N-1} u_k\norm{\tilA x_{k+1}-\tilA x_k}^2 - \sum_{k=1}^{N-1} \inprod{u_k(\tilA x_{k+1}-\tilA x_k)+\tilA x_N}{y_k-y_{k-1}} \nonumber \\
    & = -(\tau - 1) \norm{\tilA x_N}^2   + \sum_{k=1}^{N-1} u_k\norm{\tilA x_{k+1}-\tilA x_k}^2 + \sum_{k=1}^{N-1} \inprod{u_k(\tilA x_{k+1}-\tilA x_k)+\tilA x_N}{\sum_{j=0}^{k-1} 2h_{k,j+1}\tilA x_{j+1}} \label{eqn:appendix-H-duality-S-expansion-second}
\end{align}
where in the last line we substitute $y_k-y_{k-1}$ using the update rule~\eqref{eqn:FPI-H-matrix-appendix} (which finally reveals the $H$-dependency of $\cS$).

\paragraph{Rewriting $\cT$.}
We similarly expand $\cT$:
\begin{align}
    & \cT (\tilA x_1, \dots, \tilA x_N) \nonumber \\
    & = - V_0 - \tau \sqnorm{\tilA x_N} - \inprod{\tilA x_N}{x_N - y_0} \nonumber \\
    & = - \sum_{k=1}^{N-1} v_k \inprod{\tilA x_N - \tilA x_k}{x_N-x_k} - \tau \norm{\tilA x_N}^2 - \inprod{\tilA x_N}{x_N - y_0} \nonumber \\
    & = -\tau \norm{\tilA x_N}^2 - \inprod{\tilA x_N}{y_{N-1}-\tilA x_N - y_0} - \sum_{k=1}^{N-1} v_k\inprod{\tilA x_N-\tilA x_k}{y_{N-1}-\tilA x_N - y_{k-1} + \tilA x_k} \nonumber \\
    & = -(\tau-1)\norm{\tilA x_N}^2 + \inprod{\tilA x_N}{y_0-y_{N-1}} + \sum_{k=1}^{N-1} v_k\norm{\tilA x_N - \tilA x_k}^2 - \sum_{k=1}^{N-1} v_k \inprod{\tilA x_N- \tilA x_k}{y_{N-1}-y_{k-1}} \label{eqn:appendix-H-duality-T-expansion-first}
\end{align}
We can rewrite the last summation in~\eqref{eqn:appendix-H-duality-T-expansion-first} as
\begin{align*}
    -\sum_{k=1}^{N-1} v_k \inprod{\tilA x_N- \tilA x_k}{y_{N-1}-y_{k-1}} & = -\sum_{k=1}^{N-1} \sum_{j=k}^{N-1} \inprod{v_k (\tilA x_N- \tilA x_k)}{y_{j}-y_{j-1}} \\
    & = -\sum_{j=1}^{N-1} \sum_{k=1}^{j} \inprod{v_k (\tilA x_N- \tilA x_k)}{y_{j}-y_{j-1}} \\
    & = -\sum_{k=1}^{N-1} \sum_{j=1}^{k} \inprod{v_j (\tilA x_N- \tilA x_k)}{y_{k}-y_{k-1}} \\
    & = -\sum_{k=1}^{N-1} \inprod{v_1 (\tilA x_N - \tilA x_1) + \cdots + v_k (\tilA x_N - \tilA x_k)}{y_{k}-y_{k-1}}
\end{align*}
where in the second-last equality, we switch the roles of indices $k$ and $j$.
Now plug this back into~\eqref{eqn:appendix-H-duality-T-expansion-first} together with the identity $\inprod{\tilA x_N}{y_0-y_{N-1}} = -\sum_{k=1}^{N-1} \inprod{\tilA x_N}{y_k - y_{k-1}}$ to obtain
\begin{align}
    & \cT (\tilA x_1, \dots, \tilA x_N) \nonumber \\
    & = -(\tau - 1) \norm{\tilA x_N}^2 + \sum_{k=1}^{N-1} v_k\norm{\tilA x_N - \tilA x_k}^2 \nonumber \\
    & \qquad - \sum_{k=1}^{N-1}  \inprod{\tilA x_N + v_1(\tilA x_N - \tilA x_1) + \dots + v_k(\tilA x_N - \tilA x_k) }{y_{k}-y_{k-1}} \nonumber \\
    & = -(\tau - 1) \norm{\tilA x_N}^2 + \sum_{k=1}^{N-1} v_k\norm{\tilA x_N - \tilA x_k}^2 \nonumber \\
    & \qquad + \sum_{k=1}^{N-1}  \inprod{\tilA x_N + v_1(\tilA x_N - \tilA x_1) + \dots + v_k(\tilA x_N - \tilA x_k) }{\sum_{j=0}^{k-1} 2h_{N-j-1,N-k}\tilA x_{j+1}} \label{eqn:appendix-H-duality-T-expansion-second}
\end{align}
where in the last line we substitute $y_k-y_{k-1}$ using the update rule~\eqref{eqn:FPI-H-matrix-appendix} but with entries of $H^\at$, i.e., $h_{N-j-1,N-k-1}$ in place of $h_{k+1,j+1}$, which reveals the $H^\at$-dependency of $\cT$.

\paragraph{Equivalence of $\cS \ge 0$ and $\cT \ge 0$.}
Now we construct a one-to-one correspondence $\mathcal{F}\colon \prod_{j=1}^N \reals^d \to \prod_{j=1}^N \reals^d$ satisfying $\cS(g_1,\dots,g_N) = \cT(\mathcal{F}(g_1,\dots,g_N))$. 
Note that once this is established, it immediately follows that
\begin{align*}
   \left[ \cS (g_1,\dots,g_N) \geq 0 ,\quad \forall \, g_1,\dots,g_N \in \mathbb{R}^d \right] \quad \iff \quad \left[ \cT(g_1,\dots,g_N) \geq 0 ,\quad \forall \, g_1,\dots,g_N \in \mathbb{R}^d \right] .
\end{align*}
Specifically, define
\begin{align*}
    \mathcal{F}\left(g_1,\dots,g_N \right) = \big( u_{N-1}(g_N-g_{N-1}) + g_N , \dots , u_1(g_2-g_1) + g_N , g_N \big) = (g_1', \dots, g_N'). 
\end{align*}
Then $\cF$ is a bijection because it has the explicit inverse mapping
\begin{align*}
    \cF^{-1} (g_1',\dots,g_N') = \left( g_N' + \sum_{k=1}^{N-1} \frac{1}{u_{N-k}} (g_N' - g_k') , \dots , g_N' + \frac{1}{u_{N-1}} (g_N' - g_1') , g_N'\right) .
\end{align*} 
It remains to verify that $\cS(g_1,\dots,g_N) = \cT(\mathcal{F}(g_1,\dots,g_N))$ indeed holds true.
We directly plug in the transformed set of vectors $g_1',\dots,g_N'$ to the expansion~\eqref{eqn:appendix-H-duality-T-expansion-second} of $\cT$, and substitute $v_k = \frac{1}{u_{N-k}}$ for $k=1,\dots,N$:
\begin{align}
    \cT(g_1',\dots,g_N') 
    & = -(\tau - 1) \norm{g_N'}^2 + \sum_{k=1}^{N-1} v_k\norm{g_N' - g_k'}^2 \nonumber \\
    & \qquad + \sum_{k=1}^{N-1}  \inprod{g_N' + v_1(g_N' - g_1') + \dots + v_k(g_N' - g_k') }{\sum_{j=0}^{k-1} 2h_{N-j-1,N-k} g_{j+1}'} \nonumber \\
    \begin{split}
        & = -(\tau - 1) \norm{g_N'}^2 + \sum_{k=1}^{N-1} \frac{1}{u_{N-k}}\norm{g_N' - g_k'}^2 \\
        & \qquad + \sum_{k=1}^{N-1}  \inprod{g_N' +\frac{1}{u_{N-1}}(g_N' - g_1') + \dots + \frac{1}{u_{N-k}}(g_N' - g_k') }{\sum_{j=0}^{k-1} 2h_{N-j-1,N-k} g_{j+1}'}     
    \end{split}
    \label{eqn:appendix-H-duality-T-transformed-expansion}
\end{align}
Note that for $k=1,\dots,N-1$, it follows directly from the definition of $\cF$ that $g_N'-g_k' = u_{N-k}(g_{N-k}-g_{N-k+1})$.
Therefore, the first two terms of~\eqref{eqn:appendix-H-duality-T-transformed-expansion} can be rewritten as
\begin{align*}
    -(\tau - 1) \norm{g_N'}^2 + \sum_{k=1}^{N-1} \frac{1}{u_{N-k}}\norm{g_N' - g_k'}^2 & = -(\tau - 1) \norm{g_N}^2 + \sum_{k=1}^{N-1} u_{N-k}\norm{g_{N-k} - g_{N-k+1}}^2 \\
    & = -(\tau - 1) \norm{g_N}^2 + \sum_{k=1}^{N-1} u_{k}\norm{g_{k+1} - g_{k}}^2
\end{align*}
and the last expression coincides with the first two terms within~\eqref{eqn:appendix-H-duality-S-expansion-second}.
Next, rewrite the second line of~\eqref{eqn:appendix-H-duality-T-transformed-expansion} as following:
\begin{align*}
    & \sum_{k=1}^{N-1}  \inprod{g_N' +\frac{1}{u_{N-1}}(g_N' - g_1') + \dots + \frac{1}{u_{N-k}}(g_N' - g_k') }{\sum_{j=0}^{k-1} 2h_{N-j-1,N-k}g_{j+1}'}  \\
    & = \sum_{k=1}^{N-1} \inprod{g_N + (g_{N-1} - g_N) + \dots + (g_{N-k} - g_{N-k+1}) }{\sum_{j=0}^{k-1} 2h_{N-j-1,N-k} g_{j+1}'} \\
    & = \sum_{k=1}^{N-1} \inprod{g_{N-k} }{\sum_{j=0}^{k-1} 2h_{N-j-1,N-k} g_{j+1}'} \\
    & = \sum_{k=1}^{N-1}\sum_{j=0}^{k-1} 2h_{N-j-1,N-k}\inprod{g_{N-k}}{g_N + u_{N-j-1}(g_{N-j} - g_{N-j-1})} \\
    & = \sum_{k=1}^{N-1}\sum_{i=N-k}^{N-1} 2h_{i,N-k}\inprod{g_{N-k}}{g_N + u_i(g_{i+1} - g_{i})}  
\end{align*}
where to obtain the last equality, we make an index substitution $i=N-j-1$.
From the last line, make another substitution of index $\ell=N-k$:
\begin{align}
    \sum_{k=1}^{N-1}\sum_{i=N-k}^{N-1} 2h_{i,N-k}\inprod{g_{N-k}}{g_N + u_i(g_{i+1} - g_{i})} & = \sum_{\ell=1}^{N-1}\sum_{i=\ell}^{N-1} 2h_{i,\ell} \inprod{g_\ell}{g_N + u_i(g_{i+1} - g_{i})} \nonumber \\
    & = \sum_{i=1}^{N-1}\sum_{\ell=1}^{i} 2h_{i,\ell} \inprod{g_\ell}{g_N + u_i(g_{i+1} - g_{i})} . \label{eqn:appendix-H-duality-final-index-change}
\end{align} 
Finally, changing the name of the indices $(i,\ell)$ to $(k,j+1)$ in~\eqref{eqn:appendix-H-duality-final-index-change} gives
\begin{align*}
    \sum_{i=1}^{N-1}\sum_{\ell=1}^{i} 2h_{i,\ell} \inprod{g_\ell}{g_N + u_i(g_{i+1} - g_{i})} & = \sum_{k=1}^{N-1}\sum_{j=0}^{k-1} 2h_{k,j+1} \inprod{g_{j+1}}{g_N + u_k (g_{k+1} - g_k)} \\
    & = \sum_{k=1}^{N-1} \inprod{g_N + u_k (g_{k+1} - g_k)}{\sum_{j=0}^{k-1} 2h_{k,j+1} g_{j+1}}
\end{align*}
where the last expression coincides with the last summation within~\eqref{eqn:appendix-H-duality-S-expansion-second}.
This shows that $\cT(\cF(g_1,\dots,g_N)) = \cT(g_1',\dots,g_N') = \cS(g_1,\dots,g_N)$, completing the proof of~\cref{thm::H-duality_FPI}.

\newpage
\section{Proof of the optimal family theorem}
\label{section:appendix-optimal-family-theorem}

\subsection{Overview of the proof and description of the parametrization $\Phi\colon C \to \reals^{(N-1)\times(N-1)}$}
\label{subsection:appendix-optimal-family-proof-outline}
In this section, we prove \cref{theorem:optimal-method-family}.
The full proof is long and complicated, so we first outline the structure of the proof.
\begin{enumerate}
    \item We consider the following proof strategy: If there is some index set $I \subseteq \{1,\dots,N\} \times \{1,\dots,N\}$ and nonnegative real numbers $\lambda_{i,j}$ for each $(i,j) \in I$ such that
    \begin{align}
        \label{eqn:proof-template-identity-appendix}
        0 = \inprod{\topa x_N}{x_N - y_0} + N \|\topa x_N\|^2 + \sum_{(i,j)\in I} \lambda_{i,j} \inprod{\topa x_i - \topa x_j}{x_i - x_j} 
    \end{align}
    then the algorithm exhibits the rate
    \begin{align*}
        \sqnorm{y_{N-1} - \opT y_{N-1}} = 4 \sqnorm{\tilA x_N} \leq \frac{4 \norm{y_0-y_\star}^2}{N^2} .
    \end{align*}
    The final convergence rate is a direct consequence of~\eqref{eqn:proof-template-identity-appendix}, monotonicity of $\opA$ and \cref{lemma:convergence-proof-last-step}.
    
    \item We choose $I = \{(k+1,k)\,|\,k=1,\dots,N-1\} \cup \{(N,k)\,|\,k=1,\dots,N-1\}$, and we use the H-matrix representation~\eqref{eqn:FPI-H-matrix-appendix} to eliminate $x_1,\dots,x_N$ within~\eqref{eqn:proof-template-identity-appendix}.
    Then~\eqref{eqn:proof-template-identity-appendix} becomes a vector quadratic form in $\tilA x_1, \dots, \tilA x_N$, i.e.,
    \begin{align*}
        & \inprod{\topa x_N}{x_N - y_0} + N \|\topa x_N\|^2 + \sum_{(i,j)\in I} \lambda_{i,j} \inprod{\topa x_i - \topa x_j}{x_i - x_j}
        = \sum_{k=1}^N \sum_{\ell=k}^N s_{\ell,k} \inprod{\topa x_\ell}{\topa x_k} 
    \end{align*}
    where the coefficients $s_{\ell,k} = s_{\ell,k} \left( H, \lambda \right)$ are functions depending on $H$ and $\lambda = \left(\lambda_{i,j}\right)_{(i,j) \in I}$ 
    (note that we only keep $s_{\ell,k}$ with $\ell \ge k$ to avoid redundancy).
    We characterize the explicit expressions for $s_{\ell,k}$ in terms of $\lambda$ and entries of $H$, reducing the problem of establishing the identity~\eqref{eqn:proof-template-identity-appendix} to the problem of solving the system of equations
    \begin{align}
    \label{eqn:appendix-optimal-family-s-ell-k-system}
        s_{\ell,k} \left(H, \lambda \right) = 0 \quad (k=1,\dots,N, \,\, \ell=k,\dots,N) .
    \end{align}

    \item We provide explicit solutions for $\lambda$, in terms of diagonal entries $h_{k,k}$ ($k=1,\dots,N-1$) of the H-matrix.
    With these explicit values of $\lambda$, the system~\eqref{eqn:appendix-optimal-family-s-ell-k-system} becomes nonlinear in $h_{k,k}$ ($k=1,\dots,N-1$), but it is a linear system in non-diagonal entries $h_{\ell,k}$ ($\ell=k+1,\dots,N-1$) of $H$.

    \item (This is the most technical core step of analysis.) We show that with the expressions of $\lambda$ from the previous step, under certain conditions on $h_{k,k}$ ($k=1,\dots,N-1$), the linear system~\eqref{eqn:appendix-optimal-family-s-ell-k-system} is uniquely solvable in the non-diagonal H-matrix entries $h_{\ell,k}$ ($\ell=k+1,\dots,N-1$), and these unique solutions
    \begin{align}
    \label{eqn:appendix-optimal-family-h-ell-k-as-functions}
        h_{\ell,k}^* = h_{\ell,k}^*(h_{1,1},h_{2,2},\dots,h_{N-1,N-1})
    \end{align}
    can be expressed as continuous functions of the diagonal $h_{k,k}$ ($k=1,\dots,N-1$).
\end{enumerate}

We now describe how $\Phi\colon C \to \reals^{(N-1)\times(N-1)}$ is constructed.
Define
\begin{align}
\label{eqn:appendix-pk-definition}
    p_k = \prod_{\ell=k}^{N-1} h_{\ell,\ell}, \quad k=1,\dots,N-1 
\end{align}
(these quantities play significant role in all analyses of this section).
It turns out that the affine constraints 
\begin{align}
    p_1 & = \frac{1}{N} \label{eqn:appendix-family-p1-equality-constraint} \\
    p_k & \ge \frac{1}{N-k+1} \quad (k=2,\dots,N-1) \label{eqn:appendix-family-pk-affine-constraint-1} \\
    p_k & \ge \frac{N-k}{N-k-1} p_{k+1} - \frac{1}{N-k-1} \quad (k=1,\dots,N-2) \label{eqn:appendix-family-pk-affine-constraint-2} 
\end{align}
are the key conditions making the construction of Steps~1 through 4 work.
Specifically, we have $\lambda_{i,j} > 0$ for explicit expressions of $\lambda$ from Step~3, if the constraints \eqref{eqn:appendix-family-p1-equality-constraint} holds, and \eqref{eqn:appendix-family-pk-affine-constraint-1} and \eqref{eqn:appendix-family-pk-affine-constraint-2} holds with strict inequality.
The set $C$ of all $(p_2, \dots, p_{N-1})$ satisfying the inequality constraints \eqref{eqn:appendix-family-pk-affine-constraint-1} and \eqref{eqn:appendix-family-pk-affine-constraint-2} strictly is an open convex subset of $\reals^{N-2}$.
To check that $C \ne \emptyset$, consider the H-matrices of \ref{alg:halpern} and \ref{alg:dual_halpern}.
From \ref{alg:halpern}, we have
\begin{align*}
    p_k^{(0)} = \prod_{\ell=k}^{N-1} h_{\ell,\ell} = \prod_{\ell=k}^{N-1} \frac{\ell}{\ell+1} = \frac{k}{N}
\end{align*}
and from \ref{alg:dual_halpern},
\begin{align*}
    p_k^{(1)} = \prod_{\ell=k}^{N-1} h_{\ell,\ell} = \prod_{\ell=k}^{N-1} \frac{N-\ell}{N-\ell+1} = \frac{1}{N-k+1} .
\end{align*}
It can be checked by direct calculations that $p_1^{(0)} = \frac{1}{N} = p_1^{(1)}$ and that $p_k^{(0)}$'s satisfy \eqref{eqn:appendix-family-pk-affine-constraint-1} strictly and \eqref{eqn:appendix-family-pk-affine-constraint-2} with equality, while $p_k^{(1)}$'s satisfy \eqref{eqn:appendix-family-pk-affine-constraint-1} with equality and \eqref{eqn:appendix-family-pk-affine-constraint-2} strictly.
Therefore, $\left( p_2^{(0)}, \dots , p_{N-1}^{(0)} \right), \left( p_2^{(0)}, \dots , p_{N-1}^{(0)} \right) \in \partial C$.
Additionally, if we define $p_k^{(\gamma)} := \gamma p_k^{(0)} + (1-\gamma) p_k^{(1)}$ for $\gamma \in (0,1)$ and $k=1,\dots,N-1$, then $p_1^{(\gamma)} = \frac{1}{N}$, and $\left( p_2^{(\gamma)}, \dots , p_{N-1}^{(\gamma)} \right)$ satisfies both \eqref{eqn:appendix-family-pk-affine-constraint-1}, \eqref{eqn:appendix-family-pk-affine-constraint-2} strictly, i.e.,
\begin{align*}
    \left( p_2^{(\gamma)}, \dots , p_{N-1}^{(\gamma)} \right) \in C .
\end{align*}
This shows that $C$ is nonempty.

Note that provided that $p_1 = \frac{1}{N}$, we can recover the diagonal $h_{k,k}$'s from $(p_2,\dots,p_{N-1}) \in C$ as $h_{1,1} = \frac{p_1}{p_2} = \frac{1}{Np_2}$, $h_{k,k} = \frac{p_k}{p_{k+1}}$ for $k=2,\dots,N-2$ and $h_{N-1,N-1} = p_{N-1}$.
Now assuming that Step~4 is done so that we have $h_{\ell,k}^*$ determined as functions of $h_{1,1},\dots,h_{N-1,N-1}$ as in~\eqref{eqn:appendix-optimal-family-h-ell-k-as-functions},
we define $\Phi: C \to \reals^{(N-1)\times(N-1)}$ by
\begin{align*}
    \Phi (p_2, \dots, p_{N-1}) = 
    \begin{bmatrix}
        \frac{1}{N p_2} & 0 & \cdots & 0 \\
        h_{2,1}^* \left( \frac{1}{N p_2} , \frac{p_2}{p_3} , \dots , p_{N-1} \right) & \frac{p_2}{p_3} & \cdots & 0 \\
        \vdots & \vdots & \ddots & \vdots \\
        h_{N-1,1}^* \left( \frac{1}{N p_2} , \frac{p_2}{p_3} , \dots , p_{N-1} \right) & h_{N-1,2}^* \left( \frac{1}{N p_2} , \frac{p_2}{p_3} , \dots , p_{N-1} \right) & \cdots & p_{N-1}
    \end{bmatrix} .
\end{align*}
This map is injective because one can recover the values of $p_2, \dots, p_{N-1}$ from the diagonal entries of $\Phi (p_2, \dots, p_{N-1})$.
Additionally, it is continuous provided that $h_{\ell,k}^*$ are continuous.  
Because $\Phi(w)$ is designed to satisfy~\eqref{eqn:proof-template-identity-appendix} (when it is used as H-matrix) for any $w = (p_2, \dots, p_{N-1}) \in C$, by monotonicity of $\opA$, the algorithm with $\Phi(w)$ as H-matrix exhibits the exact optimal rate of $\sqnorm{y_{N-1} - \opT y_{N-1}} = 4 \sqnorm{\tilA x_N} \leq \frac{4 \norm{y_0-y_\star}^2}{N^2}$ by Step~1.
In the subsequent sections, we go through each steps outlined above, and explain how the quantities $p_k$ defined in~\eqref{eqn:appendix-pk-definition} and the constraints \eqref{eqn:appendix-family-p1-equality-constraint}, \eqref{eqn:appendix-family-pk-affine-constraint-1} and \eqref{eqn:appendix-family-pk-affine-constraint-2} become relevant to analysis.

\subsection{Step 2: Computation of coefficient functions $s_{\ell,k}$ in vector quadratic form}

We rewrite the terms $x_N - y_0$ and $x_i - x_j$ in the right hand side of~\eqref{eqn:proof-template-identity-appendix} as linear combinations of $\tilA x_k$ ($k=1,\dots,N$) according to the definition~\eqref{eqn::FPI_H_matrix} and $x_k = \JA(y_{k-1}) = y_{k-1} - \tilA x_k$, and then expand everything.
Denoting $g_k = \tilA x_k$ for simplicity, the expansion goes:
\begin{align*}
    & \inprod{\topa x_N}{x_N - y_0} + N \|\topa x_N\|^2 + \sum_{(i,j)\in I} \lambda_{i,j} \inprod{\topa x_i - \topa x_j}{x_i - x_j} \\
    & = \inprod{\topa x_N}{x_N - y_0} + N \|\topa x_N\|^2 + \sum_{k=1}^{N-2} \lambda_{k+1,k} \inprod{\topa x_{k+1} - \topa x_k}{x_{k+1} - x_k} + \sum_{k=1}^{N-2} \lambda_{N,k} \inprod{\topa x_N - \topa x_k}{x_N - x_k} \\
    & \quad + \lambda_{N,N-1} \inprod{\topa x_N - \topa x_{N-1}}{x_N - x_{N-1}} \\
    & = \inprod{g_N}{-\sum_{k=1}^{N-1} \sum_{j=1}^k 2 h_{k,j} g_j - g_N} + N \|g_N\|^2 + \sum_{k=1}^{N-2} \lambda_{k+1,k} \inprod{g_{k+1} - g_k}{-\sum_{j=1}^k 2 h_{k,j} g_j - (g_{k+1} - g_k)} \\
    & \quad + \sum_{k=1}^{N-2} \lambda_{N,k} \inprod{g_N - g_k}{-\sum_{i=k}^{N-1}\sum_{j=1}^i 2 h_{i,j} g_j - (g_N - g_k)} + \lambda_{N,N-1} \inprod{g_N - g_{N-1}}{-\sum_{j=1}^{N-1} 2 h_{N-1,j} g_j - (g_N - g_{N-1})}  \\
    & = \inprod{g_N}{-\sum_{j=1}^{N-1} \sum_{k=j}^{N-1} 2 h_{k,j} g_j - g_N} +  N \|g_N\|^2 - \sum_{k=1}^{N-2} \lambda_{k+1,k} \|g_{k+1}-g_k \|^2 - \sum_{k=1}^{N-2} \lambda_{N,k}\| g_N - g_k\|^2 - \lambda_{N,N-1}\|g_N-g_{N-1} \|^2 \\
    & \quad -\sum_{k=1}^{N-2}\lambda_{N,k}\inprod{g_N-g_k}{\sum_{i=k}^{N-1}\sum_{j=1}^i 2 h_{i,j} g_j}
    -\sum_{j=1}^{N-2} \inprod{\sum_{k=j}^{N-2} 2 h_{k,j}\lambda_{k+1,k}(g_{k+1}-g_k)}{g_j}
    -\lambda_{N,N-1}\inprod{g_N-g_{N-1}}{\sum_{j=1}^{N-1} 2 h_{N-1,j} g_j}.
\end{align*} 
We carefully gather the terms within the above expansion and group the coefficients attached to the same inner product terms, into the form
\begin{align}
\label{eqn:appendix-optimal-family-vector-quadratic-form}
    \sum_{k=1}^N \sum_{\ell=k}^N s_{\ell,k} \inprod{g_\ell}{g_k} .
\end{align} 
The result of the computation is as following:
\begin{align}
\begin{split}
    s_{N,N} & = N - 1 - \sum_{k=1}^{N-1} \lambda_{N,k} \\
    s_{N-1,N-1} & = -\lambda_{N-1,N-2} + \lambda_{N,N-1} (2 h_{N-1,N-1} - 1) \\
    s_{N,N-1} & = -2 h_{N-1,N-1} - \sum_{k=1}^{N-2} 2\lambda_{N,k} h_{N-1,N-1} - 2 \lambda_{N,N-1} (h_{N-1,N-1} - 1)
\end{split}
\label{eqn:appendix-optimal-family-s-computation-1}
\end{align}
and for $k=1,\dots,N-2$,
\begin{align}
\begin{split}
    s_{k,k} & = \begin{cases}
        \lambda_{k+1,k} (2 h_{k,k} - 1) - \lambda_{k,k-1} + \lambda_{N,k} \left( \sum_{i=k}^{N-1}2  h_{i,k} - 1 \right) & \text{if } \,\, k > 1 \\
        \lambda_{k+1,k} (2 h_{k,k} - 1) + \lambda_{N,k} \left( \sum_{i=k}^{N-1} 2 h_{i,k} - 1 \right) & \text{if } \,\, k = 1
    \end{cases} \\
    s_{k+1,k} & = 2 \lambda_{k+1,k} (1-h_{k,k}) + 2 \lambda_{k+2,k+1} h_{k+1,k} + 2 \lambda_{N,k} \sum_{i=k+1}^{N-1} h_{i,k+1} + 2 \lambda_{N,k+1} \sum_{i=k+1}^{N-1} h_{i,k} \quad (k < N - 2) \\ 
    s_{\ell,k} & = - 2 \lambda_{\ell,\ell-1} h_{\ell-1,k} + 2 \lambda_{\ell+1,\ell} h_{\ell,k} + 2 \lambda_{N,k} \sum_{i=\ell}^{N-1} h_{i,\ell} + 2 \lambda_{N,\ell} \sum_{i=\ell}^{N-1} h_{i,k} \quad (\ell=k+2,\dots,N-2) \\
    s_{N-1,k} & = 2 \lambda_{N,k} h_{N-1,N-1} - 2 \lambda_{N-1,N-2} h_{N-2,k} + 2 \lambda_{N,N-1} h_{N-1,k} \quad (k < N - 2) \\ 
    s_{N,k} & = 2 \lambda_{N,k} - \sum_{\ell=k}^{N-1} \left( 1 + \sum_{j=1}^\ell \lambda_{N,j} \right) 2 h_{\ell,k} 
\end{split}
\label{eqn:appendix-optimal-family-s-computation-2}
\end{align}
and finally,
\begin{align}
    s_{N-1,N-2} = 2 \lambda_{N-1,N-2} (1-h_{N-2,N-2}) + 2 \lambda_{N,N-1} h_{N-1,N-2} + 2 \lambda_{N,N-2} h_{N-1,N-1} .
    \label{eqn:appendix-optimal-family-s-computation-3}
\end{align}

\subsection{Step 3: Explicit characterization of $\lambda$}

Recall the definition~\eqref{eqn:appendix-pk-definition}: $p_k = \prod_{\ell=k}^{N-1} h_{\ell,\ell}$ for $k=1,\dots,N-1$. 
Assume $p_1 = \frac{1}{N}$.
We define
\begin{align}
\begin{split}
    \lambda_{N,N-1} & = N h_{N-1,N-1} \\
    \lambda_{k+1,k} & = \frac{N}{N-k-1} p_{k+1} \left( (N-k) p_{k+1} - 1 \right) \quad (k=1,\dots,N-2) \\
    \lambda_{N,k} & = \frac{N}{(N-k)(N-k-1)} - \frac{N}{N-k-1} p_{k+1} + \frac{N}{N-k} p_k \quad (k=1,\dots,N-2) 
\end{split}
\label{eqn:appendix-optimal-family-dual-variables}
\end{align}

We prove the following handy identities for $\lambda$.

\begin{proposition}
\label{proposition:lambda_Nj-telescoping}
For $\lambda$ defined as in~\eqref{eqn:appendix-optimal-family-dual-variables}, the following holds, provided that $p_1 = \frac{1}{N}$.
\begin{align*}
    & \sum_{i=k}^{N-1} \lambda_{N,i} = \frac{N(N-k-1)}{N-k} + \frac{N}{N-k} p_k \quad (k=2,\dots,N-1) \\
    & \sum_{i=1}^{N-1} \lambda_{N,i} = N - 1.  
\end{align*}
\end{proposition}

\begin{proof}
Observe that the first identity, when $k=N-1$, is $\lambda_{N,N-1} = N p_{N-1} = N h_{N-1,N-1}$, which holds true by definition.
Next, for $k=1,\dots,N-1$, we can rewrite
\begin{align*}
    \lambda_{N,k} = \frac{N}{N-k-1} - \frac{N}{N-k} - \frac{N}{N-k-1} p_{k+1} + \frac{N}{N-k} p_k
\end{align*}
so they telescope:
\begin{align*}
    \sum_{i=k}^{N-1} \lambda_{N,i} & = \lambda_{N,N-1} + \sum_{i=k}^{N-2} \lambda_{N,i} \\
    & = N p_{N-1} + \sum_{i=k}^{N-2} \left( \frac{N}{N-i-1} - \frac{N}{N-i} - \frac{N}{N-i-1} p_{i+1} + \frac{N}{N-i} p_i \right) \\
    & = N p_{N-1} + \left( N - \frac{N}{N-k} - N p_{N-1} + \frac{N}{N-k} p_k \right) \\
    & = \frac{N(N-k-1)}{N-k} + \frac{N}{N-k} p_k .
\end{align*} 
In the case $k=1$, we have
\begin{align*}
    \sum_{i=1}^{N-1} \lambda_{N,i} = \frac{N(N-2)}{N-1} + \frac{N}{N-1} p_1 = \frac{N^2 - 2N + 1}{N - 1} = N - 1 .
\end{align*}
\end{proof}

Observe that $\lambda_{k+1,k} \ge 0$ if $p_{k+1} \ge \frac{1}{N-k}$ for $k=1,\dots,N-2$, which is equivalent to~\eqref{eqn:appendix-family-pk-affine-constraint-1}.
Moreover,
\begin{align*}
    & \lambda_{N,k} = \frac{N}{(N-k)(N-k-1)} - \frac{N}{N-k-1} p_{k+1} + \frac{N}{N-k} p_k \ge 0 \\
    & \iff p_k \ge \frac{N-k}{N} \left( \frac{N}{N-k-1} p_{k+1} - \frac{N}{(N-k)(N-k-1)} \right) = \frac{N-k}{N-k-1} p_{k+1} - \frac{1}{N-k-1}
\end{align*}
which is~\eqref{eqn:appendix-family-pk-affine-constraint-2}.
Clearly, if these inequalities are satisfied strictly, then $\lambda_{k+1,k}, \lambda_{N,k} > 0$.

\subsection{Step 4: Solving the linear system $s_{\ell,k} = 0$}

We first observe that by definition of $\lambda$ from the previous section and \cref{proposition:lambda_Nj-telescoping},
\begin{align*}
    s_{N,N} & = N - 1 - \sum_{k=1}^{N-1} \lambda_{N,k} = 0 \\
    s_{N-1,N-1} & = -\lambda_{N-1,N-2} + \lambda_{N,N-1} (2 h_{N-1,N-1} - 1) \\
    & = - N p_{N-1} (2p_{N-1} - 1) + N h_{N-1,N-1} (2h_{N-1,N-1} - 1) = 0
\end{align*}
and
\begin{align*}
    s_{N,N-1} & = -2 h_{N-1,N-1} - \sum_{k=1}^{N-2} 2\lambda_{N,k} h_{N-1,N-1} - 2 \lambda_{N,N-1} (h_{N-1,N-1} - 1) \\
    & = -2h_{N-1,N-1} \left(1 + \sum_{k=1}^{N-1} \lambda_{N,k} \right) + 2\lambda_{N,N-1} \\
    & = -2N h_{N-1,N-1} + 2\lambda_{N,N-1} \\
    & = 0
\end{align*}
directly follows.
However, for $s_{\ell,k}$ with $k\le N-2$ and $\ell=k,\dots,N$, we need to determine $h_{k+1,k},\dots,h_{N-1,k}$ that achieves $s_{\ell,k} = 0$.
We first characterize the condition for $s_{k,k} = 0$.
\begin{proposition}
\label{proposition:appendix-skk-zero-condition}
For $k=1,\dots,N-2$, provided that $p_1 = \frac{1}{N}$,
\begin{align}
\label{eqn:appendix-optimal-family-skk-zero-condition}
    s_{k,k} = 0 \iff \sum_{i=k}^{N-1} 2h_{i,k} = 1 - (N-k) p_{k+1} + (N-k+1) p_k .
\end{align}
\end{proposition}

\begin{proof}
Recall from~\eqref{eqn:appendix-optimal-family-s-computation-2} that
\begin{align*}
    s_{k,k} = 
    \begin{cases}
        \lambda_{k+1,k} (2 h_{k,k} - 1) - \lambda_{k,k-1} + \lambda_{N,k} \left( \sum_{i=k}^{N-1}2  h_{i,k} - 1 \right) & \text{if } \,\, k > 1 \\
        \lambda_{k+1,k} (2 h_{k,k} - 1) + \lambda_{N,k} \left( \sum_{i=k}^{N-1} 2 h_{i,k} - 1 \right) & \text{if } \,\, k = 1
    \end{cases}
\end{align*}
To eliminate the need to deal with the case $k=1$ separately, define the dummy variable $\lambda_{1,0} = \frac{N}{N-1} p_1 (Np_1 - 1)$, which is consistent with the rule~\eqref{eqn:appendix-optimal-family-dual-variables} for defining $\lambda_{k+1,k}$ for $k=1,\dots,N-2$.
Given that $p_1 = \frac{1}{N}$ we have $\lambda_{1,0} = 0$, so we can write
\begin{align}
    s_{k,k} = 0 \, \iff \, \sum_{i=k}^{N-1} 2h_{i,k} - 1 = \frac{\lambda_{k,k-1}-\lambda_{k+1,k}(2h_{k,k}-1)}{\lambda_{N,k}} \label{eqn:optimal-method-family-induction-S_kk-equivalence}
\end{align}
where
\begin{align*}
    \lambda_{k+1,k} & = \frac{N}{N-k-1} p_{k+1} \left( (N-k)p_{k+1} - 1 \right) \\
    \lambda_{k,k-1} & = \frac{N}{N-k} p_k \left((N-k+1) p_k - 1 \right) \\
    \lambda_{N,k} & = \frac{N}{(N-k)(N-k-1)} - \frac{N}{N-k-1} p_{k+1} + \frac{N}{N-k} p_k \\
    h_{k,k} & = \frac{\prod_{\ell=k}^{N-1} h_{\ell,\ell}}{\prod_{\ell=k+1}^{N-1} h_{\ell,\ell}} = \frac{p_k}{p_{k+1}}
\end{align*} 
Plugging the above identities into the right hand side of~\eqref{eqn:optimal-method-family-induction-S_kk-equivalence} gives
\begin{align*}
    \sum_{i=k}^{N-1} 2 h_{i,k} - 1 & = \frac{\lambda_{k,k-1}-\lambda_{k+1,k}(2h_{k,k}-1)}{\lambda_{N,k}} \\
    & = \frac{\frac{N}{N-k} p_k \left( (N-k+1) p_k -1 \right) - \frac{N}{N-k-1} p_{k+1} \left( (N-k)p_{k+1} -1 \right) \left( \frac{2p_k}{p_{k+1}} - 1 \right)}{\frac{N}{(N-k)(N-k-1)} - \frac{N}{N-k-1}p_{k+1} + \frac{N}{N-k}p_k} .
\end{align*}
Multiply $\frac{(N-k)(N-k-1)}{N}$ throughout both numerator and denominator, and rearrange the expressions using the substitution $d_k = 1 - (N-k) p_{k+1} + (N-k-1) p_k$:
\begin{align*}
    \sum_{i=k}^{N-1} 2h_{i,k} - 1
    & = \frac{(N-k-1) p_k \left((N-k+1)p_k - 1 \right) - (N-k) \left((N-k)p_{k+1}-1 \right) \left(2p_k - p_{k+1}\right)}{1 - (N-k)p_{k+1} + (N-k-1)p_k} \\
    & = \frac{(N-k-1) p_k \left((N-k+1)p_k - 1 \right) - \left( (N-k)p_{k+1}-1 \right) \left( 2(N-k) p_k - (N-k) p_{k+1} \right)}{1 - (N-k)p_{k+1} + (N-k-1)p_k} \\
    & = \frac{(N-k-1) p_k \left((N-k+1)p_k - 1 \right) - \left( (N-k-1)p_k - d_k \right) \left( 2(N-k)p_k + d_k - 1 - (N-k-1)p_k \right)}{d_k} \\
    & = \frac{(N-k-1) p_k \left((N-k+1)p_k - 1 \right) - \left( (N-k-1)p_k - d_k \right) \left( (N-k+1)p_k + d_k - 1 \right)}{d_k} \\
    & = \frac{d_k^2 + (2p_k - 1) d_k}{d_k} \\
    & = d_k + 2p_k - 1.
\end{align*}
This is again equivalent to
\begin{align*}
    \sum_{i=k}^{N-1} 2h_{i,k} = d_k + 2p_k = 1 - (N-k)p_{k+1} + (N-k+1)p_k .
\end{align*}
\end{proof}

We name the quantity
\begin{align}
\label{eqn:appendix-optimal-family-qk-definition}
    q_k = 1 - (N-k)p_{k+1} + (N-k+1)p_k .
\end{align}
from~\eqref{eqn:appendix-optimal-family-skk-zero-condition} as it will be frequently used in the subsequent analysis.

In~\eqref{eqn:appendix-optimal-family-s-computation-2}, we see that for each fixed $k=1,\dots,N-2$, the equations $s_{\ell,k} = 0$ ($\ell \ge k$) involve only $h_{i,j}$'s with $i\ge j \ge k$.
Therefore, rather than trying to solve the whole system $s_{\ell,j} = 0$ ($j=1,\dots,N-2$, $\ell=k,\dots,N$) at once, we iteratively solve the partial systems $s_{\ell,k} = 0$ ($\ell=k,\dots,N$) one-by-one with fixed value of $k$, starting from $k=N-2$ and progressively lowering it to $k=1$. 

\begin{proposition}
For each $k=1,\dots,N-2$ fixed, the system of equations $s_{\ell,k} = 0$ ($\ell=k,\dots,N$), viewing only $h_{k+1,k}, \dots, h_{N-1,k}$ as variables, is a uniquely solvable linear system, if: all $\lambda_{i,j}$'s defined in~\eqref{eqn:appendix-optimal-family-dual-variables} are positive and
\begin{align}
    \sum_{i=j}^{N-1} 2 h_{i,j} = q_j
    \label{eqn:family-existence-proof-induction-assumption}
\end{align}
holds for $j=k+1,\dots,N-2$.
In this case, the solutions $h_{k+1,k}^*, \dots, h_{N-1,k}^*$ satisfying $s_{\ell,k} = 0$ ($\ell=k,\dots,N$) can be written as continuous functions of $h_{1,1}, \dots, h_{N-1,N-1}$, i.e.,
\begin{align*}
    h_{i,k}^* = h_{i,k}^* \left( h_{1,1}, \dots, h_{N-1,N-1} \right) .
\end{align*}
\label{proposition:appendix-optimal-family-solvability-up-to-k=2}
\end{proposition}

\begin{proof}
First consider the case $k=N-2$, which requires a separate treatment because $s_{N-1,N-2}$ is defined separately as~\eqref{eqn:appendix-optimal-family-s-computation-3}.
In this case, we have 3 equations to solve, namely
\begin{align}
\label{eqn:appendix-optimal-family-system-Nm2-case}
    s_{N-2,N-2} = 0 , \quad s_{N-1,N-2} = 0 , \quad s_{N,N-2} = 0 .
\end{align}
By~\cref{proposition:appendix-skk-zero-condition}, we have 
\begin{align}
\label{eqn:appendix-optimal-family-sNm2,Nm2-zero-condition}
\begin{split}
    s_{N-2,N-2} = 0 \, & \iff \, 2(h_{N-2,N-2} + h_{N-1,N-2}) = q_{N-1} = 1 - 2p_{N-1} + 3p_{N-2} \\
    & \iff \, h_{N-1,N-2} = \frac{1 - 2h_{N-1} + 3h_{N-2,N-2}h_{N-1,N-1}}{2} - h_{N-2,N-2} .
\end{split}
\end{align}
Plug the last expression for $h_{N-1,N-2}$ into 
\begin{align*}
    s_{N-1,N-2} & = 2 \lambda_{N-1,N-2} (1-h_{N-2,N-2}) + 2 \lambda_{N,N-1} h_{N-1,N-2} + 2 \lambda_{N,N-2} h_{N-1,N-1} \\
    & = 2N h_{N-1,N-1} (2h_{N-1,N-1} - 1) (1 - h_{N-2,N-2}) + 2N h_{N-1,N-1} h_{N-1,N-2} \\
    & \qquad + 2 \left( \frac{N}{2} - Nh_{N-1,N-1} + \frac{N}{2} h_{N-2,N-2}h_9{N-1,N-1} \right) h_{N-1,N-1}
\end{align*}
gives $s_{N-1,N-2} = 0$.
Next, provided that~\eqref{eqn:appendix-optimal-family-sNm2,Nm2-zero-condition} holds true, we have
\begin{align*}
    s_{N,N-2} & = 2\lambda_{N,N-2} - \left( 1 + \sum_{i=1}^{N-2} \lambda_{N,i} \right) 2h_{N-2,k} - \left( 1 + \sum_{i=1}^{N-1} \lambda_{N,i} \right) 2h_{N-1,k} \\
    & = 2\lambda_{N,N-2} - (N - \lambda_{N,N-1}) 2h_{N-2,N-2} - 2N h_{N-1,N-2} \\
    & = 2 \left( \frac{N}{2} - Nh_{N-1,N-1} + \frac{N}{2} h_{N-2,N-2}h_{N-1,N-1} \right) + 2 \lambda_{N,N-1} h_{N-2,N-2} - N (2h_{N-2,N-2} + 2h_{N-1,N-1}) \\
    & = 2 \left( \frac{N}{2} - Nh_{N-1,N-1} + \frac{N}{2} h_{N-2,N-2}h_{N-1,N-1} \right) + 2 Nh_{N-1,N-1} h_{N-2,N-2} - N (1 - 2h_{N-1} + 3h_{N-2,N-2}h_{N-1,N-1}) \\
    & = 0 .
\end{align*}
This shows the value of $h_{N-1,N-2}$ characterized by~\eqref{eqn:appendix-optimal-family-sNm2,Nm2-zero-condition} is the unique solution solving the equations~\eqref{eqn:appendix-optimal-family-system-Nm2-case}.
The expression determining $h_{N-1,N-2}$ is clearly continuous in $h_{N-2,N-2}, h_{N-1,N-1}$.

Now suppose $k < N-2$.
We first consider the following system of equations (without $s_{N-1,k} = 0$ and $s_{N,k} = 0$):
\begin{align}
\begin{split}
    s_{k,k} & = \lambda_{N,k} \left( \sum_{i=k+1}^{N-1} 2h_{i,k} + 2h_{k,k} - q_k \right) = 0 \\
    s_{k+1,k} & = 2 \lambda_{k+1,k} (1-h_{k,k}) + 2 \lambda_{k+2,k+1} h_{k+1,k} + 2 \lambda_{N,k} \sum_{i=k+1}^{N-1} h_{i,k+1} + 2 \lambda_{N,k+1} \sum_{i=k+1}^{N-1} h_{i,k} = 0 \\
    s_{\ell,k} & = - 2 \lambda_{\ell,\ell-1} h_{\ell-1,k} + 2 \lambda_{\ell+1,\ell} h_{\ell,k} + 2 \lambda_{N,k} \sum_{i=\ell}^{N-1} h_{i,\ell} + 2 \lambda_{N,\ell} \sum_{i=\ell}^{N-1} h_{i,k} = 0 \quad (\ell=k+2,\dots,N-2) 
\end{split}
\label{eqn:appendix-optimal-family-invertible-system-except-for-two-eqns}
\end{align}
where the first identity has been reorganized in a simpler form using~\cref{proposition:appendix-skk-zero-condition}.
We show that the system~\eqref{eqn:appendix-optimal-family-invertible-system-except-for-two-eqns} form an invertible linear system in the variables $h_{k+1,k}, \dots, h_{N-1,k}$, and thus have unique solutions.
Then we show that these solutions are consistent with the remaining two equations to complete the proof.

\textbf{Claim 1. } The system of equations~\eqref{eqn:appendix-optimal-family-invertible-system-except-for-two-eqns} is uniquely solvable.

In~\eqref{eqn:appendix-optimal-family-invertible-system-except-for-two-eqns}, the non-diagonal $h_{i,j}$'s with $j=k+1,\dots,N-2$ are involved only through the summation $\sum_{i=j}^{N-1} h_{i,j}$, which can be replaced with $q_j/2$ by the assumption~\eqref{eqn:family-existence-proof-induction-assumption}. 
Now for each $\ell=k+1,\dots,N-2$ we subtract a multiple of $s_{k,k}$ from each $s_{\ell,k}$ to eliminate the variables $h_{\ell,k},\dots,h_{N-1,k}$ while retaining an equivalent system of linear equations:
\begin{align*}
    s'_{k+1,k} & := s_{k+1,k} - \frac{\lambda_{N,k+1}}{\lambda_{N,k}} s_{k,k} \\
    & = 2 \lambda_{k+2,k+1} h_{k+1,k} + \underbrace{\lambda_{N,k} q_{k+1} + 2 \lambda_{k+1,k} (1 - h_{k,k}) + \lambda_{N,k+1} (q_k - 2h_{k,k})}_{:= t_{k+1,k} = \, \text{Terms not depending on }h_{k+1,k}, \dots, h_{N-1,k}} \\
    s'_{\ell,k} & := s_{\ell,k} - \frac{\lambda_{N,\ell}}{\lambda_{N,k}} s_{k,k} \\
    & = -2\lambda_{\ell,\ell-1} h_{\ell-1,k} + 2\lambda_{\ell+1,\ell} h_{\ell,k} - 2\lambda_{N,\ell} \sum_{i=k+1}^{\ell-1} h_{i,k} + \underbrace{\lambda_{N,\ell} (q_k - 2h_{k,k}) + \lambda_{N,k} q_\ell}_{:= t_{\ell,k} = \, \text{Terms not depending on }h_{k+1,k}, \dots, h_{N-1,k}} \quad (\ell=k+2,\dots,N-2) .
\end{align*}
For the sake of conciseness, we separately define the terms within $s'_{\ell,k}$ not depending on $h_{k+1,k}, \dots, h_{N-1,k}$:
\begin{align}
\label{eqn:appendix-optimal-family-tlk-definition}
    t_{\ell,k} = \begin{cases}
        \lambda_{N,k} q_{k+1} + 2 \lambda_{k+1,k} (1 - h_{k,k}) + \lambda_{N,k+1} (q_k - 2h_{k,k}) & \text{if } \ell = k + 1\\
        \lambda_{N,\ell} (q_k - 2h_{k,k}) + \lambda_{N,k} q_\ell & \text{if } \ell = k+2, \dots, N-2
    \end{cases}
\end{align}
so that we can write
\begin{align*}
    s'_{k+1,k} & = 2 \lambda_{k+2,k+1} h_{k+1,k} + t_{k+1,k} \\
    s'_{\ell,k} & = -2\lambda_{\ell,\ell-1} h_{\ell-1,k} + 2\lambda_{\ell+1,\ell} h_{\ell,k} - 2\lambda_{N,\ell} \sum_{i=k+1}^{\ell-1} h_{i,k} + t_{\ell,k} .
\end{align*}
Next, we express the system of equations $\frac{s'_{\ell,k}}{2} = 0$ ($\ell=k+1,\dots,N-1$) in matrix form.
To do this, we define the vectors within $\reals^{N-k-1}$, holding the coefficients attached to the variables $h_{k+1,k}, \dots, h_{N-1,k}$ within each $\frac{s'_{\ell,k}}{2}$:
\[
\begin{array}{ccccccccc}
   \va_1 & =  \big[ & \lambda_{k+2,k+1} & 0 & 0 & \cdots & 0 & 0 & 0 \big]^\intercal \smallskip \\
   \va_2 & = \big[ & -\lambda_{N,k+2} - \lambda_{k+2,k+1} & \lambda_{k+3,k+2} & 0 & \cdots & 0 & 0 & 0 \big]^\intercal \smallskip \\
   \va_3 & = \big[ & -\lambda_{N,k+3} & -\lambda_{N,k+3} - \lambda_{k+3,k+2} & \lambda_{k+4,k+3} & \cdots & 0 & 0 & 0 \big]^\intercal \\
   \vdots &  &  &  & \vdots &  &  & \\
   \va_{N-k-2} & = \big[ & -\lambda_{N,N-2} & -\lambda_{N,N-2} & -\lambda_{N,N-2} & \cdots & -\lambda_{N,N-2} - \lambda_{N-2,N-3} & \lambda_{N-1,N-2} & 0 \big]^\intercal
\end{array}
\]
so that if we define $\vh = \begin{bmatrix} h_{k+1,k} & \cdots & h_{N-1,k} \end{bmatrix}^\intercal$, then
\begin{align}
\label{eqn:appendix-optimal-family-slk-prime-vector-form}
    \frac{s'_{\ell,k}}{2} = \va_{\ell-k}^\intercal \vh + \frac{t_{\ell,k}}{2} 
\end{align}
for $\ell = k+1,\dots,N-2$.
More precisely,
\begin{align}
\begin{split}
    \va_1 & = \lambda_{k+2,k+1} \ve_1 \\
    \va_{\ell-k} & = - \lambda_{N,\ell} \sum_{i=1}^{\ell-k-1} \ve_i -\lambda_{\ell,\ell-1} \ve_{\ell-k-1} + \lambda_{\ell+1,\ell} \ve_{\ell-k}, \quad \ell = k+2,\dots,N-2
\end{split}
\label{eqn:appendix-optimal-family-va-definition-equation}
\end{align}
where $\ve_1, \dots, \ve_{N-k-2} \in \reals^{N-k-1}$ are elementary basis vectors.
Finally, define
\[
    \va_{N-k-1} = \sum_{i=1}^{N-k-1} \ve_i = \begin{bmatrix} 1 & 1 & \cdots & 1 \end{bmatrix}
\]
so that
\begin{align}
\label{eqn:appendix-optimal-family-skk-vector-form}
    \frac{s_{k,k}}{2\lambda_{N,k}} = \sum_{i=k+1}^{N-1} h_{i,k} + h_{k,k} - \frac{q_k}{2} 
    & = \va_{N-k-1}^\intercal \vh + h_{k,k} - \frac{q_k}{2} .
\end{align}
Now, by defining the matrix
\begin{align*}
    \vA = \begin{bmatrix} \va_1 & \va_2 & \cdots & \va_{N-k-1} \end{bmatrix}^\intercal \in \reals^{(N-k-1) \times (N-k-1)}
\end{align*}
we can write
\begin{align*}
    \begin{bmatrix}
        s_{k,k} \\
        s_{k+1,k} \\
        \vdots \\
        s_{N-2,k}
    \end{bmatrix} = 0
    \,\iff\,
    \frac{1}{2}\begin{bmatrix}
        s_{k+1,k}' \\
        \vdots \\
        s_{N-2,k}' \\
        s_{k,k}/\lambda_{N,k}
    \end{bmatrix}
    = 0 \,\iff\, \vA \vh = \frac{1}{2} \begin{bmatrix}
        -t_{k+1,k} \\
        \vdots \\
        -t_{N-2,k} \\
        q_k - 2h_{k,k}
    \end{bmatrix} .
\end{align*}
Note that $\vA$ is lower-triangular:
\begin{align*}
    \vA = \begin{bmatrix}
        \lambda_{k+2,k+1} & 0 & \cdots & 0 & 0 \\
        -\lambda_{N,k+2} - \lambda_{k+2,k+1} & \lambda_{k+3,k+2} & \cdots & 0 & 0 \\
        \vdots & \vdots & \ddots & \vdots & \vdots \\
        -\lambda_{N,N-2} & -\lambda_{N,N-2} & \cdots & \lambda_{N-1,N-2} & 0 \\
        1 & 1 & \cdots & 1 & 1 
    \end{bmatrix} 
\end{align*} 
Therefore, it is invertible provided that all diagonal entries $\lambda_{k+2,k+1}, \dots, \lambda_{N-1,N-2}$ are positive.
Because all entries of $\vA$ and $t_{k+1,k}, \dots, t_{N-2,k}$ are expressed as continuous functions of $h_{1,1}, \dots, h_{N-1,N-1}$ (because $\lambda$ is defined by~\eqref{eqn:appendix-optimal-family-dual-variables} and $q_\ell$'s are defined by~\eqref{eqn:appendix-optimal-family-qk-definition}), applying Cramer's rule to the above system, we can write $h_{k+1,k}, \dots, h_{N-1,k}$ as continuous functions of $h_{k,k}, \dots, h_{N-1,N-1}$.

\textbf{Claim 2. } Assuming that the solutions of~\eqref{eqn:appendix-optimal-family-invertible-system-except-for-two-eqns} also satisfy $s_{N-1,k} = 0$, we have $s_{N,k} = 0$.

Recall that
\begin{align}
    s_{N-1,k} & = 2 \lambda_{N,k} h_{N-1,N-1} - 2 \lambda_{N-1,N-2} h_{N-2,k} + 2 \lambda_{N,N-1} h_{N-1,k} \nonumber \\
    s_{N,k} & = 2 \lambda_{N,k} - \sum_{\ell=k}^{N-1} \left( 1 + \sum_{j=1}^\ell \lambda_{N,j} \right) 2 h_{\ell,k} . \label{eqn:appendix-optimal-family-sNk-definition}
\end{align}
Assume $s_{N-1,k} = 0$ (as well as $s_{\ell,k} = 0$ for $\ell=k,\dots,N-2$).
Then
\begin{align*}
    0 = s'_{N-1,k} & := s_{N-1,k} - \frac{\lambda_{N,N-1}}{\lambda_{N,k}} s_{k,k} \\
    & = 2 \lambda_{N,k} h_{N-1,N-1} + 2 \lambda_{N,N-1} h_{N-1,k} - 2 \lambda_{N-1,N-2} h_{N-2,k} - \lambda_{N,N-1} \left( \sum_{i=k+1}^{N-1} 2h_{i,k} + 2h_{k,k} - q_k \right) \\
    & = - \lambda_{N,N-1} \sum_{i=k+1}^{N-3} 2h_{i,k} - (\lambda_{N,N-1} + \lambda_{N-1,N-2}) 2h_{N-2,k} + \underbrace{ 2\lambda_{N,k} h_{N-1,N-1} + \lambda_{N,N-1} (q_k - 2h_{k,k})}_{:= t_{N-1,k} = \, \text{Terms not depending on }h_{k+1,k}, \dots, h_{N-1,k}} .
\end{align*}
Now sum up all $s'_{\ell,k}$'s for $\ell=k+1,\dots,N-1$ (up to scalar multiplication by~2, the coefficients of $h_{k+1,k},\dots,h_{N-1,k}$ correspond to $\va_1 + \cdots + \va_{N-k-2}$ (recall~\eqref{eqn:appendix-optimal-family-va-definition-equation}) plus the vector $\begin{bmatrix} -\lambda_{N,N-1} & \cdots & -\lambda_{N,N-1} & -\lambda_{N,N-1} - \lambda_{N-1,N-2} & 0 \end{bmatrix}$) to obtain
\begin{align*}
    0 & = \sum_{\ell=k+1}^{N-1} s'_{\ell,k} \\
    & = \sum_{\ell=k+1}^{N-1} \left( -\sum_{i=\ell+1}^{N-1} \lambda_{N,i} \right) 2h_{\ell,k} + \sum_{\ell=k+1}^{N-1} t_{\ell,k} .
\end{align*}
Using the fact $1 + \sum_{i=1}^{N-1} \lambda_{N,i} = N$ (from Proposition~\ref{proposition:lambda_Nj-telescoping}), if we add
\begin{align*}
    0 = \frac{N}{\lambda_{N,k}} s_{k,k} & = \left( 1 + \sum_{i=1}^{N-1} \lambda_{N,i} \right) \left( \sum_{\ell=k+1}^{N-1} 2h_{\ell,k} + 2h_{k,k} - q_k \right) \\
    & = \sum_{\ell=k+1}^{N-1} \left( 1 + \sum_{i=1}^{N-1} \lambda_{N,i} \right) 2h_{\ell,k} + N (2h_{k,k} - q_k)
\end{align*} 
to both sides, we get
\begin{align}
    0 & = \sum_{\ell=k+1}^{N-1} \left( 1 + \sum_{i=1}^\ell \lambda_{N,i} \right) 2h_{\ell,k} + N (2h_{k,k} - q_k) + \sum_{\ell=k+1}^{N-1} t_{\ell,k} \nonumber \\
    & \!\stackrel{\eqref{eqn:appendix-optimal-family-sNk-definition}}{=} -s_{N,k} + 2\lambda_{N,k} - \left( 1+ \sum_{i=1}^k \lambda_{N,i} \right) 2h_{k,k} + N (2h_{k,k} - q_k) + \sum_{\ell=k+1}^{N-1} t_{\ell,k} \nonumber \\
    & = -s_{N,k} + \underbrace{2\lambda_{N,k} + \left( \sum_{i=k+1}^{N-1} \lambda_{N,i} \right) 2h_{k,k} - N q_k + \sum_{\ell=k+1}^{N-1} t_{\ell,k}}_{:= t_{N,k} = \, \text{Terms not depending on }h_{k+1,k}, \dots, h_{N-1,k}} . \label{eqn:appendix-optimal-family-proof-tNk-definition}
\end{align}
We show that $t_{N,k} = 0$, which then implies $s_{N,k} = 0$.
We expand and rearrange the summation $\sum_{\ell=k+1}^{N-1} t_{\ell,k}$ as following:
\begin{align}
    \sum_{\ell=k+1}^{N-1} t_{\ell,k} & = t_{k+1,k} + \sum_{\ell=k+2}^{N-2} t_{\ell,k} + t_{N-1,k} \nonumber \\
    & = \lambda_{N,k} q_{k+1} + 2\lambda_{k+1,k} (1 - h_{k,k}) + \lambda_{N,k+1} (q_k - 2h_{k,k}) + \sum_{\ell=k+2}^{N-2} \left( \lambda_{N,\ell} (q_k - 2h_{k,k}) + \lambda_{N,k} q_\ell \right) \nonumber \\
    & \quad + 2\lambda_{N,k} h_{N-1,N-1} + \lambda_{N,N-1} (q_k - 2h_{k,k}) \nonumber \\
    & = \lambda_{N,k} \left( \sum_{\ell=k+1}^{N-2} q_\ell + 2h_{N-1,N-1} \right) + 2\lambda_{k+1,k} (1 - h_{k,k}) + \left( \sum_{\ell=k+1}^{N-1} 
    \lambda_{N,\ell} \right) (q_k - 2h_{k,k}) . \label{eqn:appendix-optimal-family-proof-tlk-summation}
\end{align}
Because of the definition~\eqref{eqn:appendix-optimal-family-qk-definition} of $q_\ell$, their sum telescopes:
\begin{align}
    \sum_{\ell=k+1}^{N-2} q_\ell + 2h_{N-1,N-1} & = \left( \sum_{\ell=k+1}^{N-2}  1 - (N-\ell) p_{\ell + 1} + (N-\ell+1) p_\ell \right) + 2h_{N-1,N-1} \nonumber \\
    & = (N-k-2) - 2 p_{N-1} + (N-k) p_{k+1} + 2h_{N-1,N-1} \nonumber \\
    & = (N-k-2) + (N-k) p_{k+1}
    \label{eqn:appendix-optimal-family-proof-ql-sum-telescoping}
\end{align}
where the last line follows from $p_{N-1} = \prod_{\ell=N-1}^{N-1} h_{\ell,\ell} = h_{N-1,N-1}$.
Now plugging in the expression~\eqref{eqn:appendix-optimal-family-proof-tlk-summation} into~\eqref{eqn:appendix-optimal-family-proof-tNk-definition} and applying simplification by~\eqref{eqn:appendix-optimal-family-proof-ql-sum-telescoping} gives
\begin{align*}
    t_{N,k} = \underbrace{\lambda_{N,k} \left( (N-k) + (N-k) p_{k+1} \right)}_{(\mathrm{I})} + \underbrace{2\lambda_{k+1,k} (1 - h_{k,k})}_{(\mathrm{II})} + \underbrace{\left( \sum_{\ell=k+1}^{N-1} \lambda_{N,\ell} - N \right) q_k}_{(\mathrm{III})} .
\end{align*}
We expand each term $(\mathrm{I}), (\mathrm{II}), (\mathrm{III})$: first, plugging in the expression for $\lambda_{N,k}$ from~\eqref{eqn:appendix-optimal-family-dual-variables}, we obtain
\begin{align}
    (\mathrm{I}) & = (N-k) \lambda_{N,k} \left( 1 + p_{k+1} \right) \nonumber \\
    & = \left( \frac{N}{N-k-1} - \frac{N(N-k)}{N-k-1} p_{k+1} + Np_k \right) (1 + p_{k+1}) \nonumber \\ 
    & = \frac{N}{N-k-1} - \frac{N(N-k)}{N-k-1} p_{k+1} + N p_k + \frac{N}{N-k-1} p_{k+1} - \frac{N(N-k)}{N-k-1} p_{k+1}^2 + N p_k p_{k+1} \nonumber \\
    & = \frac{N}{N-k-1} + N p_k - N p_{k+1} + N p_k p_{k+1} - \frac{N(N-k)}{N-k-1} p_{k+1}^2 . \label{eqn:appendix-optimal-family-proof-I-expansion}
\end{align}
Next, again from~\eqref{eqn:appendix-optimal-family-dual-variables}, plug in the expression for $\lambda_{k+1,k}$ into $(\mathrm{II})$ to get
\begin{align}
    (\mathrm{II}) & = \frac{2N}{N-k-1} p_{k+1} \left( (N-k) p_{k+1} - 1 \right) (1 - h_{k,k}) \nonumber \\
    & = \frac{2N(N-k)}{N-k-1} p_{k+1}^2 - \frac{2N(N-k)}{N-k-1} p_{k+1}^2 h_{k,k} - \frac{2N}{N-k-1} p_{k+1} + \frac{2N}{N-k-1} p_{k+1} h_{k,k} \nonumber \\
    & = \frac{2N}{N-k-1} p_k - \frac{2N}{N-k-1} p_{k+1} - \frac{2N(N-k)}{N-k-1} p_k p_{k+1} + \frac{2N(N-k)}{N-k-1} p_{k+1}^2 \label{eqn:appendix-optimal-family-proof-II-expansion}
\end{align}
where in the last line, we use $p_k = \prod_{\ell=k}^{N-1} h_{\ell,\ell} = h_{k,k} \prod_{\ell=k+1}^{N-1} h_{\ell,\ell} = p_{k+1} h_{k,k}$.
Finally, applying~\cref{proposition:lambda_Nj-telescoping}, we rewrite $(\mathrm{III})$ as
\begin{align}
    (\mathrm{III}) & = \left( \frac{N(N-k-2)}{N-k-1} + \frac{N}{N-k-1} p_{k+1} - N \right) q_k \nonumber \\
    & = \left( -\frac{N}{N-k-1} + \frac{N}{N-k-1} p_{k+1} \right) \left( 1 - (N-k) p_{k+1} + (N-k+1) p_k \right) \nonumber \\
    & = -\frac{N}{N-k-1} + \frac{N}{N-k-1} p_{k+1} + \frac{N(N-k)}{N-k-1} p_{k+1} - \frac{N(N-k)}{N-k-1} p_{k+1}^2 - \frac{N(N-k+1)}{N-k-1} p_k + \frac{N(N-k+1)}{N-k-1} p_k p_{k+1} \nonumber \\
    & = -\frac{N}{N-k-1} - \frac{N(N-k+1)}{N-k-1} p_k + \frac{N(N-k+1)}{N-k-1} p_{k+1} + \frac{N(N-k+1)}{N-k-1} p_k p_{k+1} - \frac{N(N-k)}{N-k-1} p_{k+1}^2 . \label{eqn:appendix-optimal-family-proof-III-expansion}
\end{align}
Summing up the expressions~\eqref{eqn:appendix-optimal-family-proof-I-expansion}, \eqref{eqn:appendix-optimal-family-proof-II-expansion}, \eqref{eqn:appendix-optimal-family-proof-III-expansion}, the coefficients of $p_k, p_{k+1}, p_k p_{k+1}, p_{k+1}^2$ terms and the constant term all vanish, and we conclude
\begin{align*}
    t_{N,k} = (\mathrm{I}) + (\mathrm{II}) + (\mathrm{III}) = 0 .
\end{align*}

\textbf{Claim 3. } The solutions of~\eqref{eqn:appendix-optimal-family-invertible-system-except-for-two-eqns} satisfy $s_{N-1,k} = 0$.

To prove this, we first have to match the coefficients of $h_{k+1,k}, \dots, h_{N-1,k}$ within $s_{N-1,k}$ via linear combination of rows of $\vA$ through Gaussian elimination.
The following proposition and its proof outlines this process.

\begin{proposition}
\label{proposition:appendix-optimal-family-proof-gaussian-elimination}
The following identity holds:
\begin{align}
    s_{N-1,k} - 2\lambda_{N,k} h_{N-1,N-1} 
    & = \lambda_{N,N-1} \left( \frac{1}{\lambda_{N,k}} s_{k,k} - 2h_{k,k} + q_k \right) - \sum_{\ell=k+1}^{N-2} \frac{a_{\ell+1}}{\lambda_{\ell+1,\ell}} \left( s'_{\ell,k} - t_{\ell,k} \right)
    \label{eqn:interpolation-S_k,Nm1-gaussian-elimination-coefficients}
\end{align}
where the sequences $\{a_\ell, b_\ell\}_{\ell=k+1}^{N-1}$ are defined by $a_{N-1} = \lambda_{N,N-1}+\lambda_{N-1,N-2}, b_{N-1} = \lambda_{N,N-1}$ and the reverse recursion
\begin{align}
\begin{split}
    a_\ell & = b_{\ell+1} + \frac{\lambda_{N,\ell} + \lambda_{\ell,\ell-1}}{\lambda_{\ell+1,\ell}} a_{\ell+1} \\
    b_\ell & = b_{\ell+1} + \frac{\lambda_{N,\ell}}{\lambda_{\ell+1,\ell}} a_{\ell+1}
\end{split}
\label{eqn:appendix-optimal-family-proof-aj-bj-recursion}
\end{align}
for $\ell=N-1,\dots,k+1$.
\end{proposition}

\begin{proof}
Define $\vv = \begin{bmatrix} 0 & \cdots & 0 & -\lambda_{N-1,N-2} & \lambda_{N,N-1} \end{bmatrix}^\intercal$, so that
\begin{align*}
    s_{N-1,k} - 2 \lambda_{N,k} h_{N-1,N-1} & = - 2\lambda_{N-1,N-2} h_{N-2,k} + 2\lambda_{N,N-1} h_{N-1,k}  
    = 2\vv^\intercal \vh
\end{align*} 
Recall the vector identities~\eqref{eqn:appendix-optimal-family-slk-prime-vector-form}, \eqref{eqn:appendix-optimal-family-skk-vector-form}
\begin{align*}
    s'_{\ell,k} - t_{\ell,k} & = 2 \va_{\ell-k}^\intercal \vh \\
    \frac{1}{\lambda_{N,k}} s_{k,k} - 2h_{k,k} + q_k & = 2 \va_{N-k-1}^\intercal \vh 
\end{align*}
where $\va_1^\intercal, \dots, \va_{N-k-1}^\intercal$ are row vectors of $\vA$.
Because $\vA$ is lower-triangular, we can express $\vv$ as a linear combination of $\va_j$ by matching the entries from backwards.
We start with
\begin{align*}
    \vv_{N-1} := \vv - \lambda_{N,N-1} \va_{N-k-1} = \begin{bmatrix}
        -\lambda_{N,N-1} & -\lambda_{N,N-1} & \cdots & -\lambda_{N,N-1} & -\lambda_{N,N-1} - \lambda_{N-1,N-2} & 0 
    \end{bmatrix}^\intercal .
\end{align*}
We recursively define a sequence of vectors $\vv_\ell$ ($\ell=N-2, \dots, k+1$) by adding a constant multiple of $\va_{\ell-k}$ from $\vv_{\ell+1}$ which eliminates the $(\ell-k)$th entry of $\vv_{\ell+1}$.
We denote the value of the $(\ell-k)$th entry of $\vv_{\ell+1}$ by $-a_{\ell+1}$, and the common value of the remaining entries by $-b_{\ell+1}$, so:
\begin{align*}
    a_{N-1} & = \lambda_{N,N-1} + \lambda_{N-1,N-2} \\
    b_{N-1} & = \lambda_{N,N-1} .
\end{align*}
Note that for $\ell = N-2, \dots, k+1$, the last nonzero entry of $\va_{\ell-k}$ is its $(\ell-k)$th entry $\lambda_{\ell+1,\ell}$.
Hence, to cancel out the $(\ell-k)$th entry $-a_{\ell+1}$ of $\vv_{\ell+1}$, we proceed as:
\begin{align*}
    \vv_\ell & = \vv_{\ell+1} + \frac{a_{\ell+1}}{\lambda_{\ell+1,\ell}} \va_{\ell-k} \\
    & = \begin{bmatrix}
        -b_{\ell+1} & \cdots & -b_{\ell+1} & -b_{\ell+1} & -a_{\ell+1} & 0 & \cdots & 0
    \end{bmatrix}^\intercal \\
    & \quad\quad + \frac{a_{\ell+1}}{\lambda_{\ell+1,\ell}} \begin{bmatrix}
        -\lambda_{N,\ell} & \cdots & -\lambda_{N,\ell} & -\lambda_{N,\ell} - \lambda_{\ell,\ell-1} & \lambda_{\ell+1,\ell} & 0 & \cdots & 0
    \end{bmatrix}^\intercal \\
    & = \begin{bmatrix}
        -b_{\ell+1} - \frac{\lambda_{N,\ell}}{\lambda_{\ell+1,\ell}} a_{\ell+1} & \cdots & -b_{\ell+1} - \frac{\lambda_{N,\ell}}{\lambda_{\ell+1,\ell}} a_{\ell+1} & -b_{\ell+1} - \frac{\lambda_{N,\ell}+\lambda_{\ell,\ell-1}}{\lambda_{\ell+1,\ell}} a_{\ell+1} & 0 & 0 & \cdots & 0
    \end{bmatrix}^\intercal .
\end{align*}
Because $\vv_\ell = \begin{bmatrix} -b_\ell & \cdots & -b_\ell & -a_\ell & 0 & 0 & \cdots & 0 \end{bmatrix}^\intercal$, we obtain the recursion
\begin{align*}
    a_\ell = b_{\ell+1} + \frac{\lambda_{N,\ell} + \lambda_{\ell,\ell-1}}{\lambda_{\ell+1,\ell}} a_{\ell+1} , \quad b_\ell = b_{\ell+1} + \frac{\lambda_{N,\ell}}{\lambda_{\ell+1,\ell}} a_{\ell+1} \quad (\ell = N-1, \dots, k+2)
\end{align*}
as defined in the statement of \cref{proposition:appendix-optimal-family-proof-gaussian-elimination}.
Repeating this process until we reach $\ell=k+1$, we arrive at the identity
\begin{align*}
    \vv_{k+1} = \vv - \lambda_{N,N-1} \va_{N-k-1} + \sum_{\ell=k+1}^{N-2} \frac{a_{\ell+1}}{\lambda_{\ell+1,\ell}} \va_{\ell-k} = 0 .
\end{align*}
This implies
\begin{align*}
    s_{N-1,k} - 2 \lambda_{N,k} h_{N-1,N-1} & = 2\vv^\intercal \vh \\
    & = \lambda_{N,N-1} \cdot 2 \va_{N-k-1}^\intercal \vh - \sum_{\ell=k+1}^{N-2} \frac{a_{\ell+1}}{\lambda_{\ell+1,\ell}} 2\va_{\ell-k}^\intercal \vh \\
    & = \lambda_{N,N-1} \left( \frac{1}{\lambda_{N,k}} s_{k,k} - 2h_{k,k} + q_k \right) - \sum_{\ell=k+1}^{N-2} \frac{a_{\ell+1}}{\lambda_{\ell+1,\ell}} \left( s'_{\ell,k} - t_{\ell,k} \right)
\end{align*}
which is the desired conclusion.
\end{proof}

\begin{proposition}
The sequences $a_\ell, b_\ell$ defined in~\cref{proposition:appendix-optimal-family-proof-gaussian-elimination} have the closed-form expressions
\begin{align}
    a_\ell = \frac{(N-\ell-1)! \cdot N (N-\ell+1) p_\ell \prod_{i=\ell}^{N-1} p_i}{\prod_{j=\ell+1}^{N-1} \left( (N-j+1) p_j - 1 \right)} , \quad
    b_\ell = \frac{(N-\ell-1)! \cdot N \prod_{i=\ell}^{N-1} p_i}{\prod_{j=\ell+1}^{N-1} \left( (N-j+1) p_j - 1 \right)}
    \label{eqn:aj-bj-closed-form-expressions}
\end{align}
for $\ell=N-1,\dots,k+1$.
\label{proposition:aj-bj-has-closed-forms}
\end{proposition}

\begin{proof}
We use induction on $\ell = N-1, \dots, k+1$.
In the case $\ell=N-1$,
\begin{align*}
    b_{N-1} & = \lambda_{N,N-1} = N h_{N-1,N-1} = N p_{N-1} \\
    a_{N-1} & = \lambda_{N,N-1} + \lambda_{N-1,N-2} \\
    & = N h_{N-1,N-1} + N p_{N-1} ( 2p_{N-1} - 1 ) \\
    & = N p_{N-1} + N p_{N-1} ( 2p_{N-1} - 1 ) \\
    & = 2N p_{N-1}^2
\end{align*}
which is consistent with~\eqref{eqn:aj-bj-closed-form-expressions} (with the vacuous product in the denominators being 1).
Now let $k+1 \le \ell \le N-2$ and assume that the formula~\eqref{eqn:aj-bj-closed-form-expressions} holds for $\ell+1$.
Recall that
\begin{align*}
    \lambda_{N,\ell} & = \frac{N}{(N-\ell)(N-\ell-1)} - \frac{N}{N-\ell-1} p_{\ell+1} + \frac{N}{N-\ell} p_\ell \\
    \lambda_{\ell+1,\ell} & = \frac{N}{N-\ell-1} p_{\ell+1} \left( (N-\ell) p_{\ell+1} - 1 \right) \\
    \lambda_{\ell,\ell-1} & = \frac{N}{N-\ell} p_\ell \left( (N-\ell+1) p_\ell - 1 \right) .
\end{align*}
Rewriting
\begin{align*}
    \lambda_{N,\ell} = - \frac{N}{(N-\ell)(N-\ell-1)} \left( (N-\ell) p_{\ell+1} - 1 \right) + \frac{N}{N-\ell} p_\ell ,
\end{align*}
we see that
\begin{align*}
    \frac{\lambda_{N,\ell}}{\lambda_{\ell+1,\ell}} & = \frac{1}{\frac{N}{N-\ell-1} p_{\ell+1} \left( (N-\ell) p_{\ell+1} - 1 \right)} \left( - \frac{N}{(N-\ell)(N-\ell-1)} \left( (N-\ell) p_{\ell+1} - 1 \right) + \frac{N}{N-\ell} p_\ell \right) \\
    & = - \frac{1}{(N-\ell) p_{\ell+1}} + \frac{(N-\ell-1) p_\ell}{(N-\ell) p_{\ell+1} \left( (N-\ell) p_{\ell+1} - 1 \right)} .
\end{align*}
Using the above formula and the induction hypothesis~\eqref{eqn:aj-bj-closed-form-expressions}, we compute the term $b_\ell - b_{\ell+1} = \frac{\lambda_{N,\ell}}{\lambda_{\ell+1,\ell}} a_{\ell+1}$ which comes from the recursion~\eqref{eqn:appendix-optimal-family-proof-aj-bj-recursion}:
\begin{align*}
    \frac{\lambda_{N,\ell}}{\lambda_{\ell+1,\ell}} a_{\ell+1} & = \left( - \frac{1}{(N-\ell) p_{\ell+1}} + \frac{(N-\ell-1) p_\ell}{(N-\ell) p_{\ell+1} \left( (N-\ell) p_{\ell+1} - 1 \right)} \right) \frac{(N-\ell-2)! \cdot N (N-\ell) p_{\ell+1} \prod_{i=\ell+1}^{N-1} p_i}{\prod_{j=\ell+2}^{N-1} \left( (N-j+1) p_j - 1 \right)} \\
    & = - \frac{(N-\ell-2)! \cdot N \prod_{i=\ell+1}^{N-1} p_i}{\prod_{j=\ell+2}^{N-1} \left( (N-j+1) p_j - 1 \right)} + \frac{(N-\ell-1)! \cdot N \prod_{i=\ell}^{N-1} p_i}{\prod_{j=\ell+1}^{N-1} \left( (N-j+1) p_j - 1 \right)} \\
    & = -b_{\ell+1} + \frac{(N-\ell-1)! \cdot N \prod_{i=\ell}^{N-1} p_i}{\prod_{j=\ell+1}^{N-1} \left( (N-j+1) p_j - 1 \right)} .
\end{align*}
This proves the identity~\eqref{eqn:aj-bj-closed-form-expressions} for $b_\ell$.

Similarly, we compute
\begin{align*}
    \lambda_{N,\ell} + \lambda_{\ell,\ell-1} & = - \frac{N}{(N-\ell)(N-\ell-1)} \left( (N-\ell) p_{\ell+1} - 1 \right) + \frac{N}{N-\ell} p_\ell + \frac{N}{N-\ell} p_\ell ((N-\ell+1) p_\ell - 1) \\
    & = - \frac{N}{(N-\ell)(N-\ell-1)} \left( (N-\ell) p_{\ell+1} - 1 \right) + \frac{N(N-\ell+1)}{N-\ell} p_\ell^2 
\end{align*}
so
\begin{align*}
    \frac{\lambda_{N,\ell} + \lambda_{\ell,\ell-1}}{\lambda_{\ell+1,\ell}} & = \frac{1}{\frac{N}{N-\ell-1} p_{\ell+1} \left( (N-\ell) p_{\ell+1} - 1 \right)} \left( - \frac{N}{(N-\ell)(N-\ell-1)} \left( (N-\ell) p_{\ell+1} - 1 \right) + \frac{N(N-\ell+1)}{N-\ell} p_\ell^2   \right) \\
    & = - \frac{1}{(N-\ell) p_{\ell+1}} + \frac{(N-\ell+1)(N-\ell-1) p_\ell^2}{(N-\ell) p_{\ell+1} \left( (N-\ell) p_{\ell+1} - 1 \right)} .
\end{align*}
Plugging this into the right hand side of the recursion $a_\ell - b_{\ell+1} = \frac{\lambda_{N,\ell} + \lambda_{\ell,\ell-1}}{\lambda_{\ell+1,\ell}} a_{\ell+1}$ we see that
\begin{align*}
    & \frac{\lambda_{N,\ell} + \lambda_{\ell,\ell-1}}{\lambda_{\ell+1,\ell}} a_{\ell+1} \\
    & = \left( - \frac{1}{(N-\ell) p_{\ell+1}} + \frac{(N-\ell+1)(N-\ell-1) p_\ell^2}{(N-\ell) p_{\ell+1} \left( (N-\ell) p_{\ell+1} - 1 \right)} \right) \frac{(N-\ell-2)! \cdot N (N-\ell) p_{\ell+1} \prod_{i=\ell+1}^{N-1} p_i}{\prod_{j=\ell+2}^{N-1} \left( (N-j+1) p_j - 1 \right)} \\
    & = -b_{\ell+1} + \frac{(N-\ell-1)! \cdot N (N-\ell+1) p_\ell \prod_{i=\ell}^{N-1} p_i}{\prod_{j=\ell+1}^{N-1} \left( (N-j+1) p_j - 1 \right)}
\end{align*}
which proves the identity~\eqref{eqn:aj-bj-closed-form-expressions} for $a_\ell$.
This completes the induction.
\end{proof}

From~\eqref{eqn:interpolation-S_k,Nm1-gaussian-elimination-coefficients} we have
\begin{align*}
    s_{N-1,k} & = \frac{\lambda_{N,N-1}}{\lambda_{N,k}} s_{k,k} - \sum_{\ell=k+1}^{N-2} \frac{a_{\ell+1}}{\lambda_{\ell+1,\ell}} s'_{\ell,k} \\
    & \quad + \underbrace{2\lambda_{N,k} h_{N-1,N-1} + \lambda_{N,N-1} (q_k - 2h_{k,k}) + \sum_{\ell=k+1}^{N-2} \frac{a_{\ell+1}}{\lambda_{\ell+1,\ell}} t_{\ell,k}}_{:= r_{N-1,k}} . 
\end{align*}
Because $s_{\ell,k} = 0$ for $\ell=k+1,\dots,N-2$, if we prove $r_{N-1,k} = 0$ then we are done.
Plugging the definition~\eqref{eqn:appendix-optimal-family-tlk-definition} of $t_{\ell,k}$'s into $r_{N-1,k}$ we get
\begin{align}
    & r_{N-1,k} \nonumber \\
    & = 2 \lambda_{N,k} h_{N-1,N-1} + \lambda_{N,N-1} (q_k - 2 h_{k,k}) + \frac{a_{k+2}}{\lambda_{k+2,k+1}} \left( \lambda_{N,k} q_{k+1} + 2 \lambda_{k+1,k} (1 - h_{k,k}) + \lambda_{N,k+1} (q_k - 2 h_{k,k}) \right) \nonumber \\
    & \quad\quad + \sum_{\ell=k+2}^{N-2} \frac{a_{\ell+1}}{\lambda_{\ell+1,\ell}} \left( \lambda_{N,\ell} (q_k - 2h_{k,k}) + \lambda_{N,k} q_\ell \right) \nonumber \\
    & \begin{aligned}
        & = 2 \lambda_{N,k} h_{N-1,N-1} + \left( \lambda_{N,N-1} + \sum_{\ell=k+1}^{N-2} \frac{\lambda_{N,\ell}}{\lambda_{\ell+1,\ell}} a_{\ell+1} \right) (q_k - 2h_{k,k}) + \frac{2 \lambda_{k+1,k}}{\lambda_{k+2,k+1}} a_{k+2} (1-h_{k,k}) \\
        & \quad\quad + \lambda_{N,k} \sum_{\ell=k+1}^{N-2} \frac{a_{\ell+1}}{\lambda_{\ell+1,\ell}} q_\ell .
    \end{aligned}
    \label{eqn:interpolation-S_k,Nm1-bias-term-regrouped}  
\end{align}
From the proof of Proposition~\ref{proposition:aj-bj-has-closed-forms}, we have
\begin{align*}
    \frac{\lambda_{N,\ell}}{\lambda_{\ell+1,\ell}} a_{\ell+1} = b_\ell - b_{\ell+1} ,
\end{align*}
for $\ell=k+1,\dots,N-2$, so by telescoping, we can simplify:
\begin{align}
\label{eqn:appendix-optimal-family-proof-k=2-final-telescoping-1}
    \lambda_{N,N-1} + \sum_{\ell=k+1}^{N-2} \frac{\lambda_{N,\ell}}{\lambda_{\ell+1,\ell}} a_{\ell+1} & = \lambda_{N,N-1} + b_{k+1} - b_{N-1} = b_{k+1}  
\end{align}
Next, observe that
\begin{align*}
    q_\ell = - ((N-\ell) p_{\ell+1} - 1) + (N-\ell+1) p_\ell
\end{align*}
so that
\begin{align*}
    \frac{a_{\ell+1}}{\lambda_{\ell+1,\ell}} q_\ell & = \left[ \frac{1}{\frac{N}{N-\ell-1} p_{\ell+1} \left( (N-\ell) p_{\ell+1} - 1 \right)} \left( - ((N-\ell) p_{\ell+1} - 1) + (N-\ell+1) p_\ell \right) \right] a_{\ell+1} \\
    & = \left[ - \frac{N-\ell-1}{N p_{\ell+1}} + \frac{(N-\ell+1) (N-\ell-1) p_\ell}{N p_{\ell+1} ((N-\ell) p_{\ell+1} - 1)} \right] \frac{(N-\ell-2)! \cdot N (N-\ell) p_{\ell+1} \prod_{i=\ell+1}^{N-1} p_i}{\prod_{j=\ell+2}^{N-1} \left( (N-j+1) p_j - 1 \right)} \\
    & = - \frac{(N-\ell)! \prod_{i=\ell+1}^{N-1} p_i}{\prod_{j=\ell+2}^{N-1} \left( (N-j+1) p_j - 1 \right)} + \frac{(N-\ell+1)! \prod_{i=\ell}^{N-1} p_i}{\prod_{j=\ell+1}^{N-1} \left( (N-j+1) p_j - 1 \right)} \\
    & = - \frac{(N-\ell)(N-\ell-1)b_{\ell+1}}{N} + \frac{(N-\ell+1)(N-\ell) b_\ell}{N}
\end{align*}
which allows to simplify the summation via telescoping:
\begin{align}
\label{eqn:appendix-optimal-family-proof-k=2-final-telescoping-2}
    \sum_{\ell=k+1}^{N-2} \frac{a_{\ell+1}}{\lambda_{\ell+1,\ell}} q_\ell 
    = -2h_{N-1,N-1} + \frac{(N-k)(N-k-1)b_{k+1}}{N}
\end{align}
Plugging~\eqref{eqn:appendix-optimal-family-proof-k=2-final-telescoping-1} and \eqref{eqn:appendix-optimal-family-proof-k=2-final-telescoping-2} into~\eqref{eqn:interpolation-S_k,Nm1-bias-term-regrouped}, we obtain
\begin{align}
\begin{split}
    r_{N-1,k} & = 
    2 \lambda_{N,k} h_{N-1,N-1} + b_{k+1} (q_k - 2h_{k,k}) + \frac{2 \lambda_{k+1,k}}{\lambda_{k+2,k+1}} a_{k+2} (1-h_{k,k}) \\
    & \quad + \lambda_{N,k} \left( -2h_{N-1,N-1} + \frac{(N-k)(N-k-1)b_{k+1}}{N} \right)
\end{split}
\label{eqn:appendix-optimal-family-proof-rNm1k-second-last-expansion}
\end{align}
The terms $\pm 2\lambda_{N,k} h_{N-1,N-1}$ in~\eqref{eqn:appendix-optimal-family-proof-rNm1k-second-last-expansion} cancel out each other.
Now we plug in the identity
\begin{align*}
    \frac{\lambda_{k+1,k}}{\lambda_{k+2,k+1}} a_{k+2} & = \frac{\frac{N}{N-k-1} p_{k+1} ((N-k)p_{k+1} - 1)}{\frac{N}{N-k-2} p_{k+2} ((N-k-1) p_{k+2} - 1)}  \frac{(N-k-3)! \cdot N(N-k-1) p_{k+2} \prod_{i=k+2}^{N-1} p_i}{\prod_{j=k+3}^{N-1} ((N-j+1) p_j - 1)} \\
    & = \frac{(N-k-2)! \cdot N \prod_{i=k+1}^{N-1} p_i}{\prod_{j=k+2}^{N-1} ((N-j+1) p_j - 1)} ((N-k) p_{k+1} - 1) \\
    & = b_{k+1} ((N-k) p_{k+1} - 1) 
\end{align*}
into~\eqref{eqn:appendix-optimal-family-proof-rNm1k-second-last-expansion} to obtain
\begin{align*}
    r_{N-1,k} & = b_{k+1} (q_k - 2h_{k,k}) + 2 b_{k+1} ((N-k) p_{k+1} - 1) (1-h_{k,k}) + \lambda_{N,k} \frac{(N-k)(N-k-1)b_{k+1}}{N} \\
    & = b_{k+1} \left[ q_k - 2h_{k,k} + 2 ((N-k) p_{k+1} - 1) (1-h_{k,k}) + \frac{(N-k)(N-k-1)\lambda_{N,k}}{N} \right] .
\end{align*}
Finally, from the identity
\begin{align*}
    & q_k - 2h_{k,k} + 2 ((N-k) p_{k+1} - 1)(1 - h_{k,k}) \\
    & = q_k - 2h_{k,k} + 2(N-k)p_{k+1} - 2(N-k)p_{k+1}h_{k,k} - 2 + 2h_{k,k} \\
    & = 1 - (N-k) p_{k+1} + (N-k+1) p_k + 2(N-k) p_{k+1} - 2 (N-k) p_k - 2 \\
    & = - 1 + (N-k) p_{k+1} - (N-k-1) p_k \\
    & = - \frac{(N-k)(N-k-1)}{N} \lambda_{N,k}
\end{align*}
we conclude that $r_{N-1,k} = 0$.
This completes the proof of~\cref{proposition:appendix-optimal-family-solvability-up-to-k=2}.

\end{proof}

\cref{proposition:appendix-optimal-family-solvability-up-to-k=2}, together with~\cref{proposition:appendix-skk-zero-condition}, implies that we can inductively solve the systems $s_{\ell,k} = 0$ ($\ell=k,\dots,N$) from $k=N-2$ to $k=1$ to determine the solution $h_{i,k}^*$'s with $N-1 \ge i \ge k$ as continuous functions of $h_{1,1},\dots,h_{N-1,N-1}$.
This completes Step~4 within the proof outline of~\cref{subsection:appendix-optimal-family-proof-outline}, and together with the remaining arguments from~\cref{subsection:appendix-optimal-family-proof-outline}, finally proves~\cref{theorem:optimal-method-family}.

\newpage
\section{Remaining proofs and details for \cref{subsection:family-characterization-via-H-matrices}}
\label{section:appendix-optimal-family-remaining-proof}

\subsection{Proof of \cref{proposition:optimal-family-continuous-extension}}
\label{subsection:appendix-optimal-family-continuous-extension}

For simplicity, write $\left( H_{\text{OHM}} \right)_{k,j} = h_{k,j}^{(0)}$ and $\left( H_{\text{Dual-OHM}} \right)_{k,j} = h_{k,j}^{(1)}$ throughout this section.
In the proof of \cref{theorem:optimal-method-family} in \cref{section:appendix-optimal-family-theorem}, we have seen that
\begin{align*}
    p_{k}^{(0)} = \prod_{\ell=k}^{N-1} h_{\ell,\ell}^{(0)} = \frac{k}{N} , \quad p_{k}^{(1)} = \prod_{\ell=k}^{N-1} h_{\ell,\ell}^{(1)} = \frac{1}{N-k+1} 
\end{align*}
so that
\begin{align*}
    u := \left( p_2^{(0)}, \dots , p_{N-1}^{(0)} \right) \in \partial C , \quad v := \left( p_2^{(1)}, \dots , p_{N-1}^{(1)} \right) \in \partial C .
\end{align*}
Recall that for $w = (p_2, \dots, p_{N-1}) \in C$,
\begin{align*}
    \Phi (w) = 
    \begin{bmatrix}
        \frac{1}{N p_2} & 0 & \cdots & 0 \\
        h_{2,1}^* \left( \frac{1}{N p_2} , \frac{p_2}{p_3} , \dots , p_{N-1} \right) & \frac{p_2}{p_3} & \cdots & 0 \\
        \vdots & \vdots & \ddots & \vdots \\
        h_{N-1,1}^* \left( \frac{1}{N p_2} , \frac{p_2}{p_3} , \dots , p_{N-1} \right) & h_{N-1,2}^* \left( \frac{1}{N p_2} , \frac{p_2}{p_3} , \dots , p_{N-1} \right) & \cdots & p_{N-1}
    \end{bmatrix} .
\end{align*}
Clearly, we can directly extend the definition of $\Phi$ to $u,v$ for diagonal entries while preserving continuity.
For non-diagonal entries, however, we have to check whether $h_{i,k}^*$'s with $i>k$, which are defined as solutions of linear systems, converge to $h_{i,k}^{(m)}$ as $w \to \left( p_2^{(m)}, \dots , p_{N-1}^{(m)} \right)$, for $m=0,1$.
In~\cref{section:appendix-optimal-family-theorem}, we first characterize $h_{i,k}^*$ as solutions to the set of equations $s_{\ell,k} = 0$ ($\ell=k,\dots,N-2$) which is equivalent to a linear system determined by the invertible lower-triangular matrix $\vA$ defined therein, and then show that those solutions also satisfy $s_{N-1,k} = 0 = s_{N,k}$.
However, within the process of rewriting the system $s_{\ell,k} = 0$ ($\ell=k,\dots,N-2$) in terms of $\vA$ we use division by $\lambda_{N,k}$ (so we are implicitly assuming $\lambda_{N,k} > 0$), and the resulting matrix $\vA$ is not invertible if one of $\lambda_{j+1,j}$ ($j=k+1,\dots,N-2$) is zero.
Because $\lambda_{N,k}=0$ ($k=1,\dots,N-2$) for \ref{alg:halpern} and $\lambda_{j+1,j} = 0$ ($j=1,\dots,N-2$) for \ref{alg:dual_halpern}, the same argument is not directly applicable to these cases.

Instead, we retreat to the original set of equations and consider the system $s_{\ell,k} = 0$ ($k=k+1,\dots,N-1$), which has the matrix form
\begin{align*}
    \underbrace{\begin{bmatrix}
        \lambda_{k+2,k+1} + \lambda_{N,k+1} & \lambda_{N,k+1} & \cdots & \lambda_{N,k+1} & \lambda_{N,k+1} \\
        -\lambda_{k+2,k+1} & \lambda_{k+3,k+2} + \lambda_{N,k+2} & \cdots & \lambda_{N,k+2} & \lambda_{N,k+2} \\
        \vdots & \vdots & \ddots & \vdots & \vdots \\ 
        0 & 0 & \cdots & \lambda_{N-1,N-2} + \lambda_{N,N-2} & \lambda_{N,N-2} \\
        0 & 0 & \cdots & -\lambda_{N-1,N-2} & \lambda_{N,N-1}
    \end{bmatrix}}_{:=\vB}
    \underbrace{\begin{bmatrix}
        h_{k+1,k} \\
        h_{k+2,k} \\
        \vdots \\
        h_{N-2,k} \\
        h_{N-1,k}
    \end{bmatrix}}_{\vh} \quad\quad\quad\quad & \\
    = \underbrace{\frac{1}{2}\begin{bmatrix}
        -2\lambda_{k+1,k}(1-h_{k,k}) - \lambda_{N,k}(q_k - 2h_{k,k}) \\
        -\lambda_{N,k} q_{k+2} \\
        \vdots \\
        -\lambda_{N,k} q_{N-2} \\
        -2\lambda_{N,k} h_{N-1,N-1}
    \end{bmatrix}}_{:=\vb} . &
\end{align*}
As we take $w \to u$ ($u$ is the $(p_2,\dots,p_{N-1})$-coordinate for \ref{alg:halpern}), all $\lambda_{N,j}$ with $j=k+1,\dots,N-2$ vanishes, and $\vB$ converges to the lower bidiagonal matrix
\begin{align*}
    \vB^{(0)} = 
    \begin{bmatrix}
        \lambda_{k+2,k+1}^{(0)} & & & \\
        -\lambda_{k+2,k+1}^{(0)} & \lambda_{k+3,k+2}^{(0)} & & \\
        & \ddots & \ddots & \\
        & & -\lambda_{N-1,N-2}^{(0)} & \lambda_{N,N-1}^{(0)}
    \end{bmatrix}
\end{align*}
where $\lambda_{j+1,j}^{(0)} = \frac{j(j+1)}{N}$ denotes the limit of $\lambda_{j+1,j}$ as $w \to u$ (up to a constant, these are the weights multiplied to the monotonicity inequalities in the convergence proof of \ref{alg:halpern}).
It is clear that $\vB^{(0)}$ is invertible, and because $\lambda$'s depend continuously on $p_2,\dots,p_{N-1}$, the matrix $\vB$ is invertible in a neighborhood of $u$.
Within this neighborhood, the solutions $h_{k+1,k}^*, \dots, h_{N-1,k}^*$ are determined via Cramer's rule applied to the system $\vB\vh = \vb$.
As all entries of $\vb$ depend continuously on $p_2,\dots,p_{N-1}$ (because $\lambda$'s and $q_j$'s are), this proves that $h_{i,k}^* \to h_{i,k}^{(0)}$ as $w\to u$ for each $i=k+1,\dots,N-1$.
(Note that $h_{i,k}^{(0)}$ must be the unique solutions satisfying $\vB^{(0)} \vh = \vb^{(0)} = \lim_{w\to u} \vb$, as \ref{alg:halpern} also satisfies the identity~\eqref{eqn:proof-template-identity-appendix}.)

Similarly, for the case of \ref{alg:dual_halpern}, we observe that all $\lambda_{j+1,j}$ with $j=k+1,\dots,N-2$ vanish in the limit $w \to v$, so $\vB$ converges to the upper triangular matrix
\begin{align*}
    \vB^{(1)} =
    \begin{bmatrix}
        \lambda_{N,k+1}^{(1)} & \lambda_{N,k+1}^{(1)} & \cdots & \lambda_{N,k+1}^{(1)} \\
        & \lambda_{N,k+2}^{(1)} & \cdots & \lambda_{N,k+2}^{(1)} \\
        & & \ddots & \vdots & \\
        & & & \lambda_{N,N-1}^{(1)}
    \end{bmatrix}
\end{align*}
where $\lambda_{N,j}^{(1)} = \frac{N}{(N-j)(N-j+1)} = \lim_{w\to v} \lambda_{N,j}$ (as before, these are the weights multiplied to the monotonicity inequalities in the convergence proof of \ref{alg:dual_halpern} up to a constant).
Again, $\vB^{(1)}$ is invertible and thus $h_{i,k}^*$ are determined via Cramer's rule applied to the system $\vB\vh = \vb$ in a neighborhood of $v$, which shows that $h_{i,k}^* \to h_{i,k}^{(1)}$ as $w\to v$.

The above arguments altogether show that
\[
\begin{array}{lc}
   \Phi(w) \to H_{\text{OHM}}  &  \text{as } \,\, w \to u \\
   \Phi(w) \to H_{\text{Dual-OHM}}  &  \text{as } \,\, w \to v \\
\end{array}
\]
and thus concludes the proof of~\cref{proposition:optimal-family-continuous-extension}.

\subsection{Explicit characterization of optimal algorithm family for the case $N=3$}

For a nonexpansive operator $\opT$, we let $\opA = 2(\opI + \opT)^{-1} - \opI$, $x_{k+1} = \JA(y_k)$ and $\tilA x_{k+1} = y_k - x_{k+1}$ as usual.
Consider the algorithm defined by
\begin{align*}
    y_1 & = y_0 - h_{1,1} (y_0 - \opT y_0) \\
    & = y_0 - 2h_{1,1} \tilA x_1 \\
    y_2 & = y_1 - h_{2,1} (y_0 - \opT y_0) - h_{2,2} (y_1 - \opT y_1) \\
    & = y_1 - 2h_{2,1} \tilA x_1 - 2h_{2,2} \tilA x_2
\end{align*}
where $h_{1,1}, h_{2,2} \in \left[\frac{1}{2}, \frac{2}{3}\right]$, $h_{1,1}h_{2,2} = \frac{1}{3}$ and $h_{1,1} + h_{2,1} + h_{2,2} = 1$.
Let
\begin{align*}
    \lambda_{3,2} = 3h_{2,2} , \quad \lambda_{3,1} = 2 - 3h_{2,2} , \quad \lambda_{2,1} = 3 h_{2,2} (2h_{2,2} - 1)
\end{align*}
as defined in~\eqref{eqn:appendix-optimal-family-dual-variables} (we substitute $N=3$ and $p_2 = h_{2,2}$).
Clearly, we have $\lambda_{3,2}, \lambda_{3,1}, \lambda_{2,1} \ge 0$ because $\frac{1}{2} \le h_{2,2} \le \frac{2}{3}$.

Now, let $g_k = \tilA x_k$ for $k=1,2,3$.
Then
\begin{align*}
    x_2 - x_1 & = (y_1 - g_2) - (y_0 - g_1) = (y_1 - y_0) - (g_2 - g_1) = (1 - 2h_{1,1}) g_1 - g_2 \\
    x_3 - x_2 & = (y_2 - g_3) - (y_1 - g_2) = (y_2 - y_1) - (g_3 - g_2) = -2h_{2,1} g_1 + (1 - 2h_{2,2}) g_2 - g_3 \\
    x_3 - x_1 & = (x_3 - x_2) + (x_2 - x_1) = (1 - 2h_{1,1} - 2h_{2,1}) g_1 - 2h_{2,2} g_2 - g_3 \\
    x_3 - y_0 & = (x_3 - x_1) - (y_0 - x_1) = x_3 - x_1 - g_1 = -(2h_{1,1} + 2h_{2,1}) g_1 - 2h_{2,2} g_2 - g_3 
\end{align*}
so we can expand
\begin{align*}
    & \inprod{\tilA x_3}{x_3 - y_0} + 3\sqnorm{\tilA x_3} + \lambda_{2,1} \inprod{x_2 - x_1}{\tilA x_2 - \tilA x_1} + \lambda_{3,1} \inprod{x_3 - x_1}{\tilA x_3 - \tilA x_1} + \lambda_{3,2} \inprod{x_3 - x_2}{\tilA x_3 - \tilA x_2} \\
    & = \inprod{g_3}{x_3 - y_0} + 3\sqnorm{g_3} + \lambda_{2,1} \inprod{x_2 - x_1}{g_2 - g_1} + \lambda_{3,1} \inprod{x_3 - x_1}{g_3 - g_1} + \lambda_{3,2} \inprod{x_3 - x_2}{g_3 - g_2} \\
    & = \inprod{g_3}{-(2h_{1,1} + 2h_{2,1}) g_1 - 2h_{2,2} g_2 - g_3 } + 3\sqnorm{g_3} + \lambda_{2,1} \inprod{(1 - 2h_{1,1}) g_1 - g_2}{g_2 - g_1} \\
    & \quad\quad + \lambda_{3,1} \inprod{(1 - 2h_{1,1} - 2h_{2,1}) g_1 - 2h_{2,2} g_2 - g_3}{g_3 - g_1} + \lambda_{3,2} \inprod{-2h_{2,1} g_1 + (1 - 2h_{2,2}) g_2 - g_3}{g_3 - g_2} \\
    & = \left( \lambda_{2,1} (2h_{1,1} - 1) - \lambda_{3,1} (1-2h_{1,1}-2h_{2,1}) \right) \sqnorm{g_1} + \left( \lambda_{2,1}(2 - 2h_{1,1}) + 2\lambda_{3,1} h_{2,2} + 2\lambda_{3,2}h_{2,1} \right) \inprod{g_2}{g_1} \\
    & \quad\quad + \left( -(2h_{1,1} + 2h_{2,1}) + \lambda_{3,1} (2 - 2h_{1,1} - 2h_{2,1}) - 2\lambda_{3,2} h_{2,1} \right) \inprod{g_3}{g_1} + \left( -\lambda_{2,1} + \lambda_{3,2} (2h_{2,2} - 1) \right) \sqnorm{g_2} \\
    & \quad\quad + \left( -2h_{2,2} - 2\lambda_{3,1} h_{2,2} + \lambda_{3,2} (2 - 2h_{2,2}) \right) \inprod{g_3}{g_2} + \left( -1 + 3 - \lambda_{3,1} - \lambda_{3,2} \right) \sqnorm{g_3} \\
    & = s_{1,1} \sqnorm{g_1} + s_{2,1} \inprod{g_2}{g_1} + s_{3,1} \inprod{g_3}{g_1} + s_{2,2} \sqnorm{g_2} + s_{3,2} \inprod{g_3}{g_2} + s_{3,3} \sqnorm{g_3} .
\end{align*}
We verify:
\begin{align*}
    s_{3,3} & = 2 - \lambda_{3,1} - \lambda_{3,2} = 0 \\
    s_{3,2} & = -2h_{2,2} - 2\lambda_{3,1} h_{2,2} + \lambda_{3,2} (2 - 2h_{2,2}) \\
    & = -2h_{2,2} - 2(2 - 3h_{2,2})h_{2,2} + 3h_{2,2} (2 - 2h_{2,2}) = 0 \\
    s_{2,2} & = -\lambda_{2,1} + \lambda_{3,2} (2h_{2,2} - 1) \\
    & = -3h_{2,2} (2h_{2,2} - 1) + 3h_{2,2} (2h_{2,2} - 1) = 0 .
\end{align*}
Next,
\begin{align*}
    s_{3,1} & = -(2h_{1,1} + 2h_{2,1}) + \lambda_{3,1} (2 - 2h_{1,1} - 2h_{2,1}) - 2\lambda_{3,2} h_{2,1} \\
    & = 2 (h_{2,2} - 1) + (2 - 3h_{2,2}) 2h_{2,2} - 6h_{2,2} \left( 1 - \frac{1}{3h_{2,2}} - h_{2,2} \right) = 0
\end{align*}
where we plug in the definitions of $\lambda_{3,2}$ and eliminate $h_{2,1}, h_{1,1}$ using the conditions $h_{1,1}h_{2,2} = \frac{1}{3}$ and $h_{1,1} + h_{2,1} + h_{2,2} = 1$ to obtain the second equality.
Similarly,
\begin{align*}
    s_{2,1} & = \lambda_{2,1}(2 - 2h_{1,1}) + 2\lambda_{3,1} h_{2,2} + 2\lambda_{3,2}h_{2,1} \\
    & = 3h_{2,2} (2h_{2,2} - 1) \left( 2 - \frac{2}{3h_{2,2}} \right) + 2 (2 - 3h_{2,2}) h_{2,2} + 6 h_{2,2} \left(1 - \frac{1}{3h_{2,2}} - h_{2,2} \right) \\
    & = (2h_{2,2} - 1)(6h_{2,2} - 2) + 2h_{2,2} (2 - 3h_{2,2}) + 6h_{2,2} - 2 - 6h_{2,2}^2 \\
    & = 0
\end{align*}
and
\begin{align*}
    s_{1,1} & = \lambda_{2,1} (2h_{1,1} - 1) - \lambda_{3,1} (1-2h_{1,1}-2h_{2,1}) \\
    & = 3h_{2,2} (2h_{2,2} - 1) \left( \frac{2}{3h_{2,2}} - 1 \right) - (2 - 3h_{2,2}) (2h_{2,2} - 1) \\
    & = (2h_{2,2} - 1) (2 - 3h_{2,2}) - (2 - 3h_{2,2}) (2h_{2,2} - 1) \\
    & = 0 .
\end{align*}
This shows that
\begin{align*}
    0 & = \inprod{\tilA x_3}{x_3 - y_0} + 3\sqnorm{\tilA x_3} + \lambda_{2,1} \inprod{x_2 - x_1}{\tilA x_2 - \tilA x_1} \\
    & \quad \quad + \lambda_{3,1} \inprod{x_3 - x_1}{\tilA x_3 - \tilA x_1} + \lambda_{3,2} \inprod{x_3 - x_2}{\tilA x_3 - \tilA x_2} \\
    & \ge \inprod{\tilA x_3}{x_3 - y_0} + 3\sqnorm{\tilA x_3} ,
\end{align*}
which, together with \cref{lemma:convergence-proof-last-step}, proves $\sqnorm{\tilA x_3} \le \frac{\sqnorm{y_0 - y_\star}}{9}$.

Note that our choice of $h_{2,1}$ is consistent with the formula~\eqref{eqn:appendix-optimal-family-skk-zero-condition} for $k=1$, i.e., $2(h_{1,1} + h_{2,1}) = 2 - 2h_{2,2} = 1 - 2p_2 + 3p_1$.
Indeed, the calculation above is just a concrete demonstration of the more complicated series of computations presented in~\cref{section:appendix-optimal-family-theorem}.

\subsection{Examples of optimal algorithms not covered by \cref{theorem:optimal-method-family}}
\label{section:appendix-optimal-family-remaining-proof_missing}

Let $N > 3$, and fix $2 \le N' \le N-1$.
Consider the following algorithm that first runs $N'$ steps of \ref{alg:dual_halpern} (with $y_{N'-1}$ as terminal iterate) and then $N-N'$ steps of \ref{alg:halpern} (as if $y_{N'-1}$ was generated by usual \ref{alg:halpern} and then continuing on using the update rule of \ref{alg:halpern}):
\begin{align}
\begin{alignedat}{2}
    y_{k+1} & = y_k + \frac{N'-k-1}{N'-k} (\opT y_k - \opT y_{k-1}) \quad && \text{if } \,\, k=0,\dots,N'-2 \\
    y_{k+1} & = \frac{k+1}{k+2} \opT y_k + \frac{1}{k+2} y_0 \quad        && \text{if } \,\, k=N'-1,\dots,N-2  .
\end{alignedat}
\label{eqn:appendix-optimal-algorithm-outside-family}
\end{align}
We show that this algorithm also has the exact optimal rate
\begin{align*}
    \sqnorm{y_{N-1} - \opT y_{N-1}} \le \frac{4\sqnorm{y_0 - y_\star}}{N^2} .
\end{align*}
First, because $y_{N'-1}$ is generated via Dual-OHM for $N'-1$ iterations, we have
\begin{align*}
    & \inprod{\tilA x_{N'}}{x_{N'} - y_0} + N' \sqnorm{\tilA x_{N'}} \\
    & \le \inprod{\tilA x_{N'}}{x_{N'} - y_0} + N' \sqnorm{\tilA x_{N'}} + \sum_{j=1}^{N'-1} \frac{N'}{(N'-j)(N'-j+1)} \inprod{x_{N'} - x_j}{\tilA x_{N'} - \tilA x_j} \\
    & = 0
\end{align*}
where as usual, $\opA = 2(\opI + \opT)^{-1} - \opI$, $x_{k+1} = \JA(y_k)$ and $\tilA x_{k+1} = y_k - x_{k+1}$.
Now from the analysis of OHM, the identity
\begin{align*}
    \underbrace{\left( (k+1)^2 \sqnorm{\tilA x_{k+1}} + (k+1) \inprod{\tilA x_{k+1}}{x_{k+1} - y_0} \right)}_{U_{k+1}} - \underbrace{\left( (k+2)^2 \sqnorm{\tilA x_{k+2}} + (k+2) \inprod{\tilA x_{k+2}}{x_{k+2} - y_0} \right)}_{U_{k+2}} \\
    = (k+1)(k+2) \inprod{x_{k+2} - x_{k+1}}{\tilA{x_{k+2} - x_{k+1}}} 
\end{align*}
holds if the relationship
\begin{align*}
    y_{k+1} = \frac{k+1}{k+2} \opT y_k + \frac{1}{k+2} y_0 = \frac{k+1}{k+2} (y_k - 2\tilA x_{k+1}) + \frac{1}{k+2} y_0
\end{align*}
holds (regardless of whether $y_k$ is an iterate generated by OHM or not); see, e.g., \citet{RyuYin2022_largescale}.
Therefore, we have
\begin{align*}
    & N^2 \sqnorm{\tilA x_N} + N \inprod{\tilA x_{N}}{x_{N} - y_0} = U_N \\
    & \quad \le U_{N-1} \le \cdots \le U_{N'} = (N')^2 \sqnorm{\tilA x_{N'}} + N' \inprod{\tilA x_{N'}}{x_{N'} - y_0} \le 0
\end{align*}
and dividing both sides by $N$ and applying \cref{lemma:convergence-proof-last-step} we obtain $\sqnorm{\tilA x_N} \le \frac{\sqnorm{y_0 - y_\star}}{N^2}$.

However, one can expect that the algorithm~\eqref{eqn:appendix-optimal-algorithm-outside-family} will not belong to the optimal algorithm family in~\cref{theorem:optimal-method-family}, because its convergence proof above uses the set of inequalities
\begin{align*}
    \lambda_{i,j} \inprod{x_i - x_j}{\tilA x_i - \tilA x_j}
\end{align*}
with $(i,j) \in \{(N',1), \dots, (N',N'-1)\} \cup \{(N'+1, N'), \dots, (N, N-1)\}$, which is different from the set $\{(N,1),\dots,(N,N-1)\} \cup \{(2,1),\dots,(N,N-1)\}$ used in~\cref{theorem:optimal-method-family}.
Indeed, consider the algorithm~\eqref{eqn:appendix-optimal-algorithm-outside-family} in the case $N'=N-1$.
Then the upper $(N-2) \times (N-2)$ submatrix of its H-matrix equals the H-matrix of \ref{alg:dual_halpern} for with terminal iterate $y_{N-2}$. 
Next, it can be checked that the sum of each column in the H-matrix of~\ref{alg:dual_halpern} is always $\frac{1}{2}$ (regardless of the total iteration count), so that $y_{N-2} = y_0 - \tilA x_1 - \cdots - \tilA x_{N-2}$.
Then we see that
\begin{align*}
    y_{N-1} & = \frac{N-1}{N} \opT y_{N-2} + \frac{1}{N} y_0 \\
    & = \frac{N-1}{N} (y_{N-2} - 2\tilA x_{N-1}) + \frac{1}{N} y_0 \\
    & = y_0 - \frac{N-1}{N} \left( \tilA x_1 + \cdots + \tilA x_{N-2} \right) - \frac{2(N-1)}{N} \tilA x_{N-1} \\
    & = y_{N-2} + \frac{1}{N} \left( \tilA x_1 + \cdots + \tilA x_{N-2} \right) - \frac{2(N-1)}{N} \tilA x_{N-1} ,
\end{align*}
and thus, the H-matrix of~\eqref{eqn:appendix-optimal-algorithm-outside-family} with $N'=N-1$ is
\begin{align*}
    H = \begin{bmatrix}
        \frac{N-2}{N-1} & & & \\
        \vdots & \ddots & & \\
        -\frac{1}{(N-1)(N-2)} & \cdots & \frac{1}{2} \\
        -\frac{1}{2N} & \cdots & -\frac{1}{2N} & \frac{N-1}{N}
    \end{bmatrix} .
\end{align*}
For this H-matrix, we have
\begin{align*}
    p_{N-1} = \frac{N-1}{N} , \quad p_{N-2} = \frac{N-1}{2N} , \quad \dots , \quad p_2 = \frac{N-1}{(N-2)N}
\end{align*}
and it can be checked that $w = (p_2,\dots,p_{N-1}) \in C$.
However, if $H = \Phi(w)$, then by construction of $\Phi$ (see \eqref{eqn:appendix-optimal-family-qk-definition}, \eqref{eqn:family-existence-proof-induction-assumption}) we must have
\begin{align*}
    2 (h_{N-1,N-2} + h_{N-2,N-2}) = 1 - 2p_{N-1} + 3p_{N-2} .
\end{align*}
However, the left hand side is
\begin{align*}
    2 \left( \frac{1}{2} - \frac{1}{2N} \right) = \frac{N-1}{N} ,
\end{align*}
while the right hand side is
\begin{align*}
    1 - \frac{2(N-1)}{N} + \frac{3(N-1)}{2N} = \frac{N+1}{2N} ,
\end{align*}
which is a contradiction.
Therefore, we conclude that $H \ne \Phi(w)$.

\newpage
\section{Omitted details from~\cref{section:minimax}}

\subsection{Lyapunov analysis of \ref{alg:dual-feg}}
\label{subsection:appendix-Dual-FEG-Lyapunov}

In this section, we denote $\opA x = \sop{x}$ for simplicity. 
Recall the update rules of~\ref{alg:dual-feg}:
\begin{align}
\begin{split}
    x_{k+\half} & = x_k - \alpha z_k - \alpha \opA{x_k} \\
    x_{k+1}  & = x_{k+1/2} - \frac{N-k-1}{N-k}\alpha\left(\opA x_{k+1/2} - \opA x_k \right) \\ 
    z_{k+1} & = \frac{N-k-1}{N-k} z_k - \frac{1}{N-k} \opA{x_{k+\hf}},
\end{split}
\tag{Dual-FEG}
\end{align}
and the Lyapunov function
\begin{align*}
    V_k = -\alpha \sqnorm{z_k + \opA x_N} + \frac{2}{N-k} \inprod{z_k + \opA x_N}{x_k - x_N}
\end{align*}
for $k=1,\dots,N-1$.
We prove that
\begin{align*}
    V_{k} - V_{k+1}
    &= \frac{2}{ (N-k)(N-k-1)} \underbrace{ \inprod{\opA x_N - \opA x_{k+1/2}}{x_N - x_{k+1/2}} }_{ \mathrm{MI}_k } \\ 
    &  \qquad + \frac{1}{\alpha (N-k)^2}  \underbrace{ \Big( \|x_{k+1/2} - x_k\|^2 - \alpha^2 \|\opA x_{k+1/2} - \opA x_k\|^2 \Big) }_{ \mathrm{LI}_k }
\end{align*}
for $k=1,\dots,N-2$. 
This implies $V_k \ge V_{k+1}$ because $\mathrm{MI}_k \ge 0$ by monotonicity of $\opA$ and $\mathrm{LI}_k \ge 0$ from $L$-Lipschitz continuity of $\opA$ ($L$-smoothness of $\vL$) and $0 < \alpha \le \frac{1}{L}$.

We first decompose $V_k$ as following:
\begin{align*}
    V_k 
    = -\alpha \norm{ z_k }^2 - \alpha \norm{ \opA x_N }^2 
     + \underbrace{ \left(\frac{2}{N-k} \inprod{\opA x_N}{x_{k} - x_N} - 2\alpha \inprod{\opA x_N}{z_{k}} \right) }_{  :=V_{k}^{(1)} }  
     + \underbrace{ \frac{2}{N-k} \inprod{ z_k }{x_k - x_N} }_{  :=V_{k}^{(2)} }.
\end{align*}
Observe that $x_{k+1/2}$ can be written in the following two ways:
\begin{align}
    \begin{aligned}
        x_{k+\half} & = x_k - \alpha z_k - \alpha \opA{x_k} \\
        x_{k+1/2} &= x_{k+1} - \frac{N-k-1}{N-k} \alpha (\opA x_k - \opA x_{k+1/2}) . \label{eqn:dual-extra-anchor-xkhalf-identity} 
    \end{aligned}
\end{align} 
Appropriately plugging \eqref{eqn:dual-extra-anchor-xkhalf-identity} into $\mathrm{MI}_k$ we can rewrite it as
\begin{align*}
    \mathrm{MI}_k &= (N-k) \inprod{\opA x_N}{x_N - x_{k+1/2}} - (N-k-1) \inprod{\opA x_N}{x_N - x_{k+1/2}} - \inprod{\opA x_{k+1/2}}{x_N - x_{k+1/2}} \\
    & = (N-k) \inprod{\opA x_N}{x_N - x_{k+1} + \frac{N-k-1}{N-k} \alpha (\opA x_k - \opA x_{k+1/2})} \\
    & \quad \quad  - (N-k-1) \inprod{\opA x_N}{x_N - x_k + \alpha z_k + \alpha \opA x_k} - \inprod{\opA x_{k+1/2}}{x_N - x_{k+1/2}}  \\
    &= (N-k) \inprod{\opA x_N}{x_N - x_{k+1}} - (N-k-1) \inprod{\opA x_N}{x_N - x_k} \nonumber \\
    & \quad \quad - \alpha (N-k-1) \inprod{\opA x_N}{ \opA x_{k+1/2} +  z_k } - \inprod{ \opA x_{k+1/2}}{x_N - x_{k+1/2}} .
\end{align*}
Now multiply $\frac{2}{(N-k)(N-k-1)}$ to $\mathrm{MI}_k$:
\begin{align}
    & \frac{2}{(N-k)(N-k-1)} \mathrm{MI}_k \nonumber \\ 
    &= \frac{2}{N-k-1} \inprod{\opA x_N}{x_N - x_{k+1}} - \frac{2}{N-k} \inprod{\opA x_N}{x_N - x_k} \nonumber \\ 
    & \quad \quad - 2\alpha \inprod{\opA x_N}{ \frac{1}{N-k} \opA x_{k+1/2} + \frac{1}{N-k} z_k } - \frac{2}{N-k-1} \inprod{ \frac{1}{N-k} \opA x_{k+1/2}}{x_N - x_{k+1/2}} \label{eqn:dual-extra-anchor-proof-Axkhalf-before-replacement}
\end{align}
and plug the following identity (which is the third line of \ref{alg:dual-feg})
\begin{align}
    \label{eqn:dual-extra-anchor-proof-Axkhalf-in-terms-of-z} 
     \frac{1}{N-k} \opA x_{k+1/2} = \frac{N-k-1}{N-k} z_k - z_{k+1} 
\end{align}
into~\eqref{eqn:dual-extra-anchor-proof-Axkhalf-before-replacement} to obtain
\begin{align}
    & \frac{2}{(N-k)(N-k-1)} \mathrm{MI}_k \nonumber \\
    & = \frac{2}{N-k-1} \inprod{\opA x_N}{x_N - x_{k+1}} - \frac{2}{N-k} \inprod{\opA x_N}{x_N - x_k} 
        \nonumber \\
    & \quad \quad - 2\alpha \inprod{\opA x_N}{ z_k - z_{k+1}} - \frac{2}{N-k-1} \inprod{ \frac{N-k-1}{N-k} z_k - z_{k+1} }{x_N - x_{k+1/2}} \nonumber \\
    & = \frac{2}{N-k} \inprod{\opA x_N}{x_k - x_N} - 2\alpha \inprod{\opA x_N}{z_k}
    -  \left(\frac{2}{N-k-1} \inprod{\opA x_N}{x_{k+1} - x_N} - 2\alpha \inprod{\opA x_N}{z_{k+1}} \right)  \nonumber \\
    & \quad \quad
    + \underbrace{ 2 \inprod{  \frac{1}{N-k} z_k - \frac{1}{N-k-1} z_{k+1} }{x_{k+1/2} - x_N}}_{:=R_k} \nonumber \\
    & = V_k^{(1)} - V_{k+1}^{(1)} + R_k
    \label{eqn:dual-extra-anchor-proof-monotonicity-grouped} 
\end{align}
We rewrite $R_k$ as following, using \eqref{eqn:dual-extra-anchor-xkhalf-identity}:
\begin{align}
    R_k & = 2\inprod{ \frac{1}{N-k} z_k }{x_{k+1/2} - x_N} - 2\inprod{ \frac{1}{N-k-1} z_{k+1} }{x_{k+1/2} - x_N} \nonumber \\
    & = 2\inprod{ \frac{1}{N-k} z_k }{x_k - x_N - \alpha z_k - \alpha \opA x_k} 
        - 2\inprod{ \frac{1}{N-k-1} z_{k+1} }{ x_{k+1} - x_N - \frac{N-k-1}{N-k} \alpha (\opA x_k - \opA x_{k+1/2})} \nonumber \\
    & = \frac{2}{N-k} \inprod{ z_k }{x_k - x_N}
        - \frac{2}{N-k-1} \inprod{z_{k+1}}{x_{k+1} - x_N} \nonumber    \\
    & \quad \quad +  \frac{2\alpha}{N-k} \inprod{\opA x_k}{ z_{k+1} - z_k } 
    - \frac{2\alpha}{N-k} \|z_k\|^2 - 2\alpha \inprod{ z_{k+1} }{ \frac{1}{N-k} \opA x_{k+1/2}}  \nonumber \\
    & = V_k^{(2)} - V_{k+1}^{(2)} + \alpha \pr{ \frac{2}{N-k} \inprod{\opA x_k}{ z_{k+1} - z_k }  - \frac{2}{N-k} \|z_k\|^2
     - \frac{2(N-k-1)}{N-k} \inprod{z_{k+1}}{  z_k }  + 2\norm{ z_{k+1} }^2 } 
    \label{eqn:dual-extra-anchor-proof-Rk-rewritten}
\end{align}
where the last equality uses~\eqref{eqn:dual-extra-anchor-proof-Axkhalf-in-terms-of-z}.
Now multiplying $\frac{1}{\alpha(N-k)^2}$ to $\mathrm{LI}_k$ and applying the identities~\eqref{eqn:dual-extra-anchor-xkhalf-identity}, \eqref{eqn:dual-extra-anchor-proof-Axkhalf-in-terms-of-z} we obtain
\begin{align}
    & \frac{1}{\alpha(N-k)^2} \pr{\mathrm{LI}_k} \nonumber \\
    & = \frac{1}{\alpha(N-k)^2} \left( \|x_{k+1/2} - x_k\|^2 - \alpha^2 \|\opA x_{k+1/2} - \opA x_k\|^2 \right) \nonumber \\
    & = \frac{1}{(N-k)^2} \left( \alpha \norm{ z_k + \opA x_k}^2 - \alpha \| (N-k-1) z_k -  (N-k) z_{k+1} - \opA x_k\|^2 \right) \nonumber \\
    & =  \alpha \pr{ \norm{ \frac{1}{N-k} z_k + \frac{1}{N-k} \opA x_k}^2  
        - \left \| \frac{1}{N-k} z_k + \frac{1}{N-k} \opA x_k + ( z_{k+1} - z_k ) \right \|^2  } \nonumber \\ 
    &= - \alpha \inner{ z_{k+1} - \pr{ 1 - \frac{2}{N-k} } z_k + \frac{2}{N-k} \opA x_k }{ z_{k+1} - z_k }   \nonumber \\
    &= - \alpha  \pr{ \norm{ z_{k+1} }^2 + \pr{ 1  - \frac{2}{N-k} } \norm{ z_{k} }^2 
        - \frac{2(N-k-1)}{N-k} \inner{z_k}{z_{k+1}}
        + \frac{2}{N-k} \inprod{\opA x_k}{ z_{k+1} - z_k } }
    \label{eqn:dual-extra-anchor-lipschitzness}
\end{align}
holds.
Now, we add~\eqref{eqn:dual-extra-anchor-proof-monotonicity-grouped} with \eqref{eqn:dual-extra-anchor-lipschitzness}, plug in \eqref{eqn:dual-extra-anchor-proof-Rk-rewritten} and simplify to obtain:
\begin{align*}
    \frac{2}{(N-k)(N-k-1)} \pr{\mathrm{MI}_k} + \frac{1}{\alpha(N-k)^2} \pr{\mathrm{LI}_k}
    &= (V_k^{(1)} + V_k^{(2)} - \alpha \|z_k\|^2) - (V_{k+1}^{(1)} + V_{k+1}^{(2)} - \alpha \|z_{k+1}\|^2)   \\
    &= V_k - V_{k+1}.
\end{align*} 
It remains to show $V_{N-1} \ge 0$. 
When $k = N-1$, we have
\begin{align*}
    x_{N} = x_{N-1/2} = x_{N-1} - \alpha z_{N-1} - \alpha \opA x_{N-1} 
    \quad \Longleftrightarrow \quad 
    z_{N-1} + \opA x_{N-1} = - \frac{1}{\alpha} \pr{ x_N - x_{N-1} }.
\end{align*}
Therefore,
\begin{align*}
    V_{N-1} & = -\alpha \sqnorm{z_{N-1} + \opA x_N} + 2 \inprod{z_{N-1} + \opA x_N}{x_{N-1} - x_N} \\
    &= - \alpha \norm{ { z_{N-1} + \opA x_{N-1} } + \pr{ \opA x_{N} - \opA x_{N-1}  } }^2 + 2 \inner{  z_{N-1} + \opA x_{N-1} }{ x_{N-1}-x_N }  \\
    &= - \alpha \norm{ -\frac{1}{\alpha} \pr{ x_N - x_{N-1} } + \pr{ \opA x_{N} - \opA x_{N-1}  } }^2 + \frac{2}{\alpha} \norm{ x_N - x_{N-1} }^2  \\
    &= \frac{1}{\alpha} \pr{ \norm{ x_N - x_{N-1} }^2 - \alpha^2 \norm{ \opA x_{N} - \opA x_{N-1}  }^2 }
        + 2 \inner{ x_N - x_{N-1} }{ \opA x_{N} - \opA x_{N-1} }        
    \ge 0,
\end{align*}
which concludes the proof. 

\subsection{Proof of~\cref{proposition:dual-feg-H-dual}}
\label{subsection:proof-of-dual-feg-H-dual}

In \cref{section:appendix-algorithm-equivalence}, we derive the H-matrix forms of \ref{alg:feg} and \ref{alg:dual-feg}, respectively~\eqref{eqn:appendix-FEG-H-matrix-form} and \eqref{eqn:appendix-dual-FEG-H-matrix-form}.
It remains to check that the matrices are indeed in the anti-transpose relationship.
Recall the definition
\begin{align*}
    (H_{\text{Dual-FEG}})_{\ell,i} =     \begin{cases}
         h_{\ell/2,(i-1)/2} & \text{if } i \leq \ell \\
        0     & \text{if } i > \ell.
    \end{cases}
\end{align*}
For comparison, write
\begin{align*}
    (H_{\text{FEG}})_{\ell,i} =     \begin{cases}
         \hat{h}_{\ell/2,(i-1)/2} & \text{if } i \leq \ell \\
        0     & \text{if } i > \ell .
    \end{cases}
\end{align*}
We check that $h_{\ell/2,(i-1)/2} = (H_{\text{Dual-FEG}})_{\ell,i} = (H^\at_{\text{FEG}})_{\ell,i} = (H_{\text{FEG}})_{2N-i+1,2N-\ell+1} = \hat{h}_{N-(i-1)/2,N-\ell/2}$ by carefully dividing cases, which completes the proof.

\paragraph{Case 1.} For $\ell=i=2k+1$ ($k=0,\dots,N-1$),
\begin{align*}
    h_{\ell/2,(i-1)/2} = h_{k+1/2,k} = \alpha L = \hat{h}_{N-k,N-k-1/2} = \hat{h}_{N-(i-1)/2,N-\ell/2}.
\end{align*}

\paragraph{Case 2.} For $\ell=2k+1$ ($k=1,\dots,N-1)$, $i=2j+2$ ($j=0,\dots,k-1$),
\begin{align*}
    h_{\ell/2,(i-1)/2} = h_{k+1/2,j+1/2} = -\frac{N-k}{(N-j-1)(N-j)} \alpha L = \hat{h}_{N-j-1/2,N-k-1/2} = \hat{h}_{N-(i-1)/2,N-\ell/2}
\end{align*}
(note that $N-j > N-k$).

\paragraph{Case 3.} For $\ell=2k+1$ ($k=1,\dots,N-1)$, $i=2j+1$ ($j=0,\dots,k-1$),
\begin{align*}
    h_{\ell/2,(i-1)/2} = h_{k+1/2,j} = 0 = \hat{h}_{N-j,N-k-1/2} = \hat{h}_{N-(i-1)/2,N-\ell/2}
\end{align*}
(note that $N-j > N-k$).

\paragraph{Case 4.} For $\ell=i=2k+2$ ($k=0,\dots,N-1$),
\begin{align*}
    h_{\ell/2,(i-1)/2} = h_{k+1,k+1/2} = \frac{N-k-1}{N-k} \alpha L = \hat{h}_{N-k-1/2,N-k-1} = \hat{h}_{N-(i-1)/2,N-\ell/2} .
\end{align*}

\paragraph{Case 5.} For $\ell=2k+2$, $i=2k+1$ ($k=0,\dots,N-1$),
\begin{align*}
    h_{\ell/2,(i-1)/2} = h_{k+1,k} = -\frac{N-k-1}{N-k} \alpha L = \hat{h}_{N-k,N-k-1} = \hat{h}_{N-(i-1)/2,N-\ell/2} .
\end{align*}

\paragraph{Case 6.} For $\ell=2k+2$, $i=1,\dots,2k$ ($k=0,\dots,N-1$), because $\hat{h}_{k',j'} = 0$ for any $j'=0,1,\dots$ and $k' \ge j' + 3/2$,
\begin{align*}
    h_{\ell/2,(i-1)/2} = h_{k+1,(i-1)/2} = 0 = \hat{h}_{N-(i-1)/2,N-k-1} = \hat{h}_{N-(i-1)/2,N-\ell/2} .
\end{align*}

\begin{flushright}
\qedsymbol    
\end{flushright}

\newpage
\section{Omitted details from \cref{section:continuous-time}}
\label{section:appendix-continuous}

\subsection{Derivation of ODE as continuous-time limit} \label{appendix:continuous_time_limit_algorithms}

In this section, we derive the Dual-Anchor ODE as continuous-time limit of \ref{alg:dual_halpern} and \ref{alg:dual-feg}.
The derivation is informal in the sense that we do not rigorously derive the sup-norm convergence
\[
    \lim_{\alpha \to 0^+} \sup_{t\in [0,T]}\norm{X(t) - x_{\lfloor t/\alpha \rfloor}} = 0 ,
\]
but we clarify all essential correspondence between discrete and continuous-time quantities.

\subsubsection{\ref{alg:dual_halpern} to Dual-Anchor ODE}  \label{appendix:continuous_time_limit_dual_OHM}
Assume $\opA$ is continuous monotone operator and let $\alpha>0$. 
Consider \ref{alg:dual_halpern} with $\opT = 2\opJ_{\frac{\alpha}{2} \opA} - \opI$ 
(note that this differs from the ``usual identification'' $\opT = 2\JA - \opI$ by scale, because here we need a step-size $\alpha$ with respect to which we take the limit).
Let $x_{k+1} = \opJ_{\frac{\alpha}{2} \opA} y_k$, so that $x_{k+1} + \frac{\alpha}{2} \opA x_{k+1} = y_k$. 
Then
\begin{align*}
    \opT y_k = 2x_{k+1} - y_k = y_k - \alpha \opA x_{k+1},
\end{align*}
and we can rewrite the $z$-form of \ref{alg:dual_halpern}~\eqref{alg:dual_halpern_with_z} as
\begin{align*}
    \frac{\frac{z_{k+1}}{\alpha} - \frac{z_{k}}{\alpha}}{\alpha} & = -\frac{1}{\alpha(N-k)} \frac{z_k}{\alpha} - \frac{1}{\alpha(N-k)} \opA x_{k+1} \\ 
    \frac{y_{k+1} - y_k}{\alpha} &= -\frac{z_{k+1}}{\alpha} - \opA x_{k+1}.
\end{align*}
Identifying $\alpha k =t$, $\alpha N =T$, $y_k=Y(t)$, $\frac{z_{k}}{\alpha} = Z(t)$ and $x_k = X(t)$, we obtain
\begin{align*}
    \dot{Z}(t) & = -\frac{1}{T-t} Z(t) - \frac{1}{T-t} \opA (X(t+\alpha)) \\
    \dot{Y}(t) &= -Z(t+\alpha) - \opA (X(t+\alpha)).
\end{align*}
Assuming differentiability of $X,Y,Z$ and taking the limit $\alpha \to 0$ gives
\begin{align}
\label{eqn:appendix-continuous-dual-OHM-derivation-Y-ODE}
\begin{split}
    \dot{Z}(t) & = -\frac{1}{T-t} Z(t) - \frac{1}{T-t} \opA (X(t)) \\
    \dot{Y}(t) &= -Z(t) - \opA (X(t)).
\end{split}
\end{align} 
Once we show $\dot{Y}(t) = \dot{X}(t)$ in the limit $\alpha \to 0$, \eqref{eqn:appendix-continuous-dual-OHM-derivation-Y-ODE} becomes the Dual-Anchor ODE~\eqref{ode:dual-anchor}.
From $x_{k+1} + \frac{\alpha}{2} \opA x_{k+1} = y_k$, we have
\begin{align*}
    \dot{X}(t) + \mathcal{O}(\alpha)
    = \frac{x_{k+1} - x_k}{\alpha}
    = \frac{y_{k} - y_{k-1}}{\alpha} - \frac{1}{2} \pr{ \opA x_{k+1} - \opA x_{k} }
    = \dot{Y}(t) + \mathcal{O}(\alpha) - \frac{1}{2} \pr{ \opA (X(t+\alpha)) - \opA (X(t)) }
\end{align*}
which shows that indeed, $\dot{X}(t) = \dot{Y}(t)$ in the limit $\alpha \to 0$.

\subsubsection{\ref{alg:dual-feg} to Dual-Anchor ODE} \label{appendix:continuous_time_limit_dual_FEG}

From the definition of~\ref{alg:dual-feg}, we have
\begin{align} 
    \begin{split}
        x_{k+1} & = x_{k+\hf} - \frac{N-k-1}{N-k} \alpha \left( \opA x_{k+\hf} - \opA x_k \right) \\
        & = x_k - \alpha z_k - \alpha \opA x_k  - \frac{N-k-1}{N-k} \alpha \left( \opA x_{k+\hf} - \opA x_k \right)  
    \end{split}
 \label{eqn:appendix-continuous-dual-FEG-xkp1-xk-difference}
\end{align}
where $\opA = \nabla_\pm \vL$.
Therefore, we can write
\begin{align*}
    \frac{x_{k+1} - x_k}{\alpha} &= - z_k - \opA x_k - \frac{\alpha(N-k-1)}{\alpha(N-k)} \pr{ \opA x_{k+\hf} - \opA x_{k} }  \\
    \frac{z_{k+1} - z_k }{\alpha} & = -\frac{1}{\alpha(N-k)} z_k - \frac{1}{\alpha(N-k)} \opA x_{k+\hf} .
\end{align*}
Now identify $\alpha k =t$, $\alpha N = T$, $x_k = X(t)$ and $z_k = Z(t)$ to obtain
\begin{align*}
    \dot{X}(t) & = -Z(t) - \opA (X(t)) - \frac{T-t-\alpha}{T-t} \left( \opA (X(t) + \cO(\alpha)) - \opA (X(t)) \right) \\
    \dot{Z}(t) & = -\frac{1}{T-t} Z(t) - \frac{1}{T-t} \opA (X(t) + \cO(\alpha))
\end{align*}
where we use $x_{k+\hf} = x_k + \cO(\alpha)$.
Thus, under the limit $\alpha\to 0$, we obtain
\begin{align*}
    \dot{X}(t) & = -Z(t) - \opA (X(t)) \\
    \dot{Z}(t) & = -\frac{1}{T-t} Z(t) - \frac{1}{T-t} \opA (X(t))
\end{align*}
which is the Dual-Anchor ODE~\eqref{ode:dual-anchor}.

\subsection{Proof of~\cref{proposition:continuous-time-H-dual}} 
\label{appendix:continuous_time_H-matrix}

The Anchor ODE~\eqref{ode:anchor} \cite{RyuYuanYin2019_ode, SuhParkRyu2023_continuoustime} is
\begin{align}
\label{eqn:appendix-continuous-anchor-ode-restated}
    \dot{X}(t) = -\opA(X(t)) - \frac{1}{t} ( X(t) - X_0 )  
\end{align}
where $X(0) = X_0$. 
We first rewrite~\eqref{eqn:appendix-continuous-anchor-ode-restated} in the H-kernel form
\begin{align*}
    \dot{X}(t) = -\int_{0}^{t} H(t,s) \opA(X(s)) ds .
\end{align*} 
Multiply $t$ to both sides of~\eqref{eqn:appendix-continuous-anchor-ode-restated} and reorganize to get
\begin{align*}
    \frac{d}{dt} \pr{ t (X(t) - X_0) }
    = t\dot{X}(t) + (X(t) - X_0)
    = - t \opA(X(t)).
\end{align*}
Integrating from $0$ to $t$ and dividing both sides by $t$ we obtain
\begin{align*}
    X(t) - X_0 = -\frac{1}{t} \int_{0}^t s \opA(X(s)) ds.
\end{align*}
Differentiating both sides gives
\begin{align*}
    \dot{X}(t) 
    = \frac{1}{t^2} \int_{0}^t s \opA(X(s)) ds - \opA(X(t))
    = \int_{0}^t \pr{ \frac{s}{t^2} - \delta(s-t) }\opA(X(s)) ds =  -\int_{0}^{t} H(t,s) \opA(X(s)) ds 
\end{align*}
with
\[
    H(t,s) = -\frac{s}{t^2}  + \delta(s-t)
\]
and $\delta(\cdot)$ is the Dirac delta function. 
The continuous-time anti-diagonal transpose of this H-kernel is:
\begin{align*}
    H^\at(t,s) = H(T-s, T-t) = -\frac{T-t}{(T-s)^2}  + \delta(s-t) ,
\end{align*}
Therefore, the H-dual of the Anchor ODE is
\begin{align*}
    \dot{X}(t) = -\int_0^t H^\at (t,s) \opA (X(s)) ds
    = \int_{0}^t \pr{ \frac{T-t}{(T-s)^2}  - \delta(s-t) }\opA(X(s)) ds
    = \int_{0}^t \frac{T-t}{(T-s)^2} \opA(X(s)) ds - \opA(X(t)).
\end{align*}
Now define
\begin{align*}
    Z(t) = - \dot{X}(t) - \opA(X(t)) = - (T-t) \int_{0}^t \frac{1}{(T-s)^2} \opA(X(s)) ds .
\end{align*}
Differentiating both sides of the above equation gives
\begin{align*}
    \dot{Z}(t)
    = \int_{0}^t \frac{1}{(T-s)^2} \opA(X(s)) ds - \frac{1}{T-t} \opA(X(t))
    = -\frac{1}{T-t} Z(t) - \frac{1}{T-t} \opA(X(t)). 
\end{align*}
The last equation, together $\dot{X}(t) = -Z(t) - \opA(X(t))$ which follows directly from definition of $Z$, becomes~\eqref{ode:dual-anchor}, which is the Dual-Anchor ODE.
This shows that Dual-Anchor ODE and Anchor ODE are H-duals of each other. 

\subsection{Existence, uniqueness and Lyapunov analysis of Dual Anchor ODE}

\subsubsection{Well-posedness of the dynamics}

\begin{theorem} \label{appendix_thm:well-posedness_for_Lipschitz_continuous_A}
Suppose $\opA \colon \reals^d \to \reals^d$ is Lipschitz continuous and let $T > 0$. 
Consider the differential equation  
\begin{align}
\begin{split}
    \pmat{ \dot{X}(t) \\ \dot{Z}(t) }
     & = - \pmat{ Z(t) + \opA(X(t)) \\ \frac{1}{T-t} Z(t) + \frac{1}{T-t} \opA(X(t))  }  
\end{split} 
\label{eqn:dual-anchor-ode-appendix-vector-form}
\end{align}
for $t \in (0, T)$, with initial conditions $X(0)=X_0$ and $Z(0)=0$. 
Then there is a unique solution $\pmat{ X \\ Z } \in \mathcal{C}^1([0,T), \reals^{2d})$ that satisfies \eqref{eqn:dual-anchor-ode-appendix-vector-form} for all $t \in (0, T)$. 
\end{theorem}

\begin{proof} 
    The right hand side of~\eqref{eqn:dual-anchor-ode-appendix-vector-form}, as a function of $t$ and $\pmat{X\\Z}$, can be rewritten as
    \begin{align*}
        F\pr{t,\pmat{ X \\ Z }}
         & = - \pmat{ Z + \opA(X) \\ \frac{1}{T-t} Z + \frac{1}{T-t} \opA(X)  }  
         = - \pmat{ I & 0 \\ 0 & \frac{1}{T-t} I  }  \pmat{ Z + \opA(X) \\ Z + \opA(X) } .
    \end{align*}
    Let $L>0$ be the Lipschitz continuity parameter for $\opA$. 
    Fix $\overline{t} \in (0,T)$; then for $t\in[0,\overline{t}]$, $F$ is Lipschitz continuous with parameter
    \[
        \sqrt{ 2\left( 1 + \frac{1}{(T-\overline{t})^2} \right) } \max\set{1,L}
    \]
    in $\pmat{ X \\ Z }$. 
    By classical ODE theory, Lipschitz continuity of $F$ on $[0,\overline{t}]$ implies that the solution $\pmat{ X \\ Z } \in \mathcal{C}^1([0,\overline{t}], \reals^{2d})$ uniquely exists.
    Since $\overline{t} \in (0,T)$ is arbitrary, the solution uniquely exists on $[0,T)$ and the proof is complete.   
\end{proof}

\begin{corollary} \label{appendix:equivalence_of_two_ODE}
    Suppose $\opA \colon \reals^d \to \reals^d$ is Lipschitz continuous and let $T > 0$. 
    The solution $\pmat{ X \\ Z } \colon [0,T) \to \reals^{2d}$ to~\eqref{eqn:dual-anchor-ode-appendix-vector-form} satisfies the differential equation
    \begin{align} 
    \label{eqn:appendix-continuous-dual-anchor-ODE-second-order-restated}
        \ddot{X}(t) + \frac{1}{T-t} \dot{X}(t) + \frac{d}{dt} \opA (X(t)) = 0 
    \end{align}
    almost everywhere in $t\in[0,T)$. 
    Conversely, if $X \in \mathcal{C}^1([0,T), \reals^{d})$ satisfies~\eqref{eqn:appendix-continuous-dual-anchor-ODE-second-order-restated} almost everywhere in $t\in[0,T)$, then with $Z \colon [0,T) \to \reals^{d}$ defined as $Z(t) = -\dot{X}(t) -\opA(X(t))$, $\pmat{ X \\ Z }$ is the solution of~\eqref{eqn:dual-anchor-ode-appendix-vector-form}. 
\end{corollary}

\begin{proof}
    Let $L$ be the Lipschitz continuity parameter of $\opA$ and take $\overline{t} \in (0,T)$. 
    As $t \mapsto \dot{X}(t)$ is continuous on the compact interval $[0,\overline{t}]$, we have $M = \max_{t\in[0,\overline{t}]} \norm{ \dot{X}(t) } < \infty$. 
    Therefore $t \mapsto X(t)$ is $M$-Lipschitz, which implies that $t \mapsto \opA(X(t))$ is $LM$-Lipschitz; hence $\opA(X(t))$ is differentiable almost everywhere in $t$. 
    Then $\dot{X}(t) = -Z(t) - \opA(X(t))$ is absolutely continuous in $t$ (being the sum of $\cC^1$ function $Z$ and Lipschitz continuous $\opA(X(t))$) and hence differentiable almost everywhere.
    Now the equivalence between~\eqref{eqn:dual-anchor-ode-appendix-vector-form} and~\eqref{eqn:appendix-continuous-dual-anchor-ODE-second-order-restated} can be proved via same arguments as in~\cref{subsection:appendix-equivalence-dual-anchor-ode}, with all equalities holding almost everywhere in $t$. 
\end{proof}

\subsubsection{Regularity at the terminal time $t=T$} \label{appendix:proof_of_ode_terminal_infomation}

Next, we show the solution on $[0,T)$ can be continuously extended to the terminal time $t=T$ due to its favorable properties. 
First, we need:

\begin{lemma} \label{appendix_lemma:bound_of_dotX}
    Suppose $\opA \colon \reals^d \to \reals^d$ is Lipschitz continuous and monotone.
    Let $\pmat{X\\Z}$ be the solution to~\eqref{eqn:dual-anchor-ode-appendix-vector-form} on $[0,T)$.
    Then 
    \begin{align*}
        \Psi(t) =  \frac{1}{(T-t)^2} \norm{ \dot{X}(t) }^2
    \end{align*}
    is nonincreasing in $t \in [0,T)$. 
    In particular, this implies  
    \begin{align*}
        \norm{\dot{Z}(t)}^2 = \frac{1}{(T-t)^2} \norm{ \dot{X}(t) }^2 \le \frac{1}{T^2} \norm{ \opA(X_0) }^2.
    \end{align*}
\end{lemma}

\begin{proof}
Take the inner product with $\frac{1}{(T-t)^2}\dot{X}(t)$ to both sides of~\eqref{eqn:appendix-continuous-dual-anchor-ODE-second-order-restated}, we have for almost every $t\in [0,T)$,
\begin{align*}
    0 
    &= \frac{1}{(T-t)^2} \inner{\ddot{X}(t)}{\dot{X}(t)} + \frac{1}{(T-t)^3} \norm{ \dot{X}(t) }^2  + \frac{1}{(T-t)^2} \inner{ \frac{d}{dt} \opA(X(t)) }{ \dot{X}(t) } \\
    &= \frac{d}{dt} \pr{ \frac{1}{2(T-t)^2} \norm{ \dot{X}(t) }^2 } + \frac{1}{(T-t)^2} \inner{ \frac{d}{dt} \opA(X(t)) }{ \dot{X}(t) }.
\end{align*}
Note that because $\dot{X}(t)$ is absolutely continuous, $\Psi(t)$ is also absolutely continuous.
Integrating from $0$ to $t$ and reorganizing, we obtain the following ``conservation law'':
\begin{align*}
    E 
    \equiv \frac{1}{2T^2} \norm{ \dot{X}(0) }^2
    = \underbrace{\frac{1}{2(T-t)^2} \norm{ \dot{X}(t) }^2}_{\frac{\Psi(t)}{2}} + \int_{0}^{t} \frac{1}{(T-s)^2} \inner{ \frac{d}{ds} \opA(X(s)) }{ \dot{X}(s) } ds
\end{align*}
($E$ is a constant).
Note that by monotonicity of $\opA$, 
\begin{align*}
    \inner{ \frac{d}{dt} \opA(X(t)) }{ \dot{X}(t) }
    = \lim_{h\to 0} \frac{ \inner{ \opA(X(t+h)) - \opA(X(t)) }{ X(t+h) - X(t) } }{h^2} \ge 0.
\end{align*}
Therefore for almost every $t\in(0,T)$,
\begin{align*}
    \dot{\Psi}(t) = -\frac{2}{(T-t)^2} \inner{ \frac{d}{dt} A(X(t)) }{ \dot{X}(t) } \le 0 ,
\end{align*}
from which we conclude that $\Psi(t)$ is nonincreasing.
Finally, \eqref{eqn:dual-anchor-ode-appendix-vector-form} gives $\dot{Z}(t) = \frac{1}{T-t}\dot{X}(t)$ and $\dot{X}(0) = -Z(0) - \opA(X_0) = - \opA(X_0)$, so we conclude that 
\begin{align*}
    \norm{ \dot{Z}(t) }^2
    = \frac{1}{(T-t)^2} \norm{ \dot{X}(t) }^2
    = \Psi(t)
    \le \Psi(0) 
    = \frac{1}{T^2} \norm{ \opA(X_0) }^2.
\end{align*}

\end{proof}

\begin{corollary}
    \label{lemma:ode_terminal_infomation}
    Suppose $\opA \colon \reals^d \to \reals^d$ is Lipschitz continuous and monotone.
    Then, given the solution $\pmat{X\\Z}$ to~\eqref{eqn:dual-anchor-ode-appendix-vector-form} on $[0,T)$, one can continuously extend $X(t)$ to $t = T$, and the extension $X \colon [0,T] \to \reals^d$ satisfies
    \begin{align*}
        \lim_{t \to T^-} \dot{X}(t) = \lim_{t \to T^-} \frac{X(t) - X(T)}{t-T} = 0.
    \end{align*}
\end{corollary} 

\begin{proof}  
    Let $t \in [0,T)$ and let $h>0$ satisfy $t+h < T$. Then by~\cref{appendix_lemma:bound_of_dotX},
    \begin{align*}
        \norm{ X(t+h) - X(t) }
        = \norm{ \int_{t}^{t+h} \dot{X}(s)  ds }
        \le \int_{t}^{t+h} \norm{ \dot{X}(s) }  ds 
        \le \int_{t}^{t+h} \frac{T-s}{T} \norm{ \opA(X_0) }  ds 
        \le h \norm{ \opA(X_0) } .
    \end{align*}
    This shows that $X\colon [0,T) \to \reals^d$ is Lipschitz continuous, which implies any sequence $\{X(t_j)\}$ with $0 < t_j \nearrow T$ is a Cauchy sequence, and such sequential limit is unique.
    Setting $X(T)$ to this unique limit, we see that $X\colon [0,T] \to \reals^d$ is a (unique) continuous function extending $X$. 
    Additionally, \cref{appendix_lemma:bound_of_dotX} gives
    \begin{align*}
        \lim_{t\to T^-} \norm{ \dot{X}(t) } 
        \le \lim_{t\to T^-} \frac{T-t}{T} \norm{ \opA(X_0) }
        = 0 .
    \end{align*}
    Therefore,
    \begin{align*}
        \lim_{t\to T^-} \norm{ \frac{X(t) - X(T)}{t-T} }
        = \lim_{t\to T^-} \frac{1}{T-t}  \norm{ \int_{t}^{T} \dot{X}(s) ds  }
        \le \lim_{t\to T^-} \frac{1}{T-t} \int_{t}^{T} \norm{ \dot{X}(s) } ds
        = \lim_{t\to T^-} \norm{ \dot{X}(t) }
        = 0
    \end{align*}
    where the second-to-last equality uses L'Hôpital's rule.
\end{proof}

\subsubsection{Proof of \cref{thm:continuous_convergnece_rate}} \label{appendix:proof_of_continuous_convergnece_rate}
Recall that we define $V\colon [0,T) \to \reals$ by
\begin{align*}  
    V(t) &= -\norm{ Z(t) + \opA(X(T)) }^2 
     + 2\inner{ Z(t) + \opA(X(T)) }{ \frac{1}{T-t} \pr{ X(t) - X(T) } } .
\end{align*}
Differentiate $V$ and apply the substitutions $Z = -( \dot{X} + \opA(X) )$ and $\dot{Z} = \frac{1}{T-t} \dot{X}$ to obtain
\begin{align*}
    \dot{V}(t)
    &= -2 \inner{\dot{Z}(t)}{ Z(t) + \opA(X(T))} 
        + 2\inner{ \dot{Z}(t) }{ \frac{1}{T-t} \pr{ X(t) - X(T) } }  \\
    & \qquad + 2\inner{ Z(t) + \opA(X(T)) }{ \frac{d}{dt} \pr{ \frac{1}{T-t} \pr{ X(t) - X(T) } } } \\  
    &= 2\inner{\frac{1}{T-t}\dot{X}(t)}{ \dot{X}(t) + \opA(X(t)) - \opA(X(T)) + \frac{1}{T-t} \pr{X(t) - X(T)} } \\ 
    & \qquad - 2\inner{ \dot{X}(t) + \opA(X(t)) - \opA(X(T)) }{ \frac{1}{(T-t)^2} \pr{ X(t) - X(T) } + \frac{1}{T-t} \dot{X}(t) } \\
    &= -\frac{2}{(T-t)^2} \inner{ X(t) - X(T) }{ \opA(X(t)) - \opA(X(T)) } \\ 
    &\le 0.
\end{align*}
On the other hand, by~\cref{lemma:ode_terminal_infomation}, we have $\lim_{t\to T^-}\dot{X}(t) = 0$, which implies
\begin{align*}
    \lim_{t\to T^-} Z(t) = \lim_{t\to T^-} \pr{ -\dot{X}(t) - \opA(X(t)) } = -\opA(X(T)) 
\end{align*}
and thus,
\begin{align*} 
    \lim_{t\to T^-} V(t)
    = \lim_{t\to T^-} \left( -\norm{ Z(t) + \opA(X(T)) }^2 \right)
     + \lim_{t\to T^-} 2 \inner{ Z(t) + \opA(X(T)) }{ \frac{ X(t) - X(T)}{T-t}  }
    = 0
\end{align*}
where the last equality holds because both $Z(t) + \opA(X(T))$ and $\frac{X(t) - X(T)}{t-T}$ converges to $0$ as $t\to T^-$.
Therefore,
\begin{align*}
    0 = \lim_{t\to T^-} V(t) \le V(0) = - \norm{\opA(X(T))}^2 - \frac{2}{T}\inner{\opA(X(T))}{X(T)-X_0}
\end{align*}
where the last equality uses $Z(0) = 0$.
Now multiplying $-\frac{T}{2}$ to both sides and applying \cref{lemma:convergence-proof-last-step}, we conclude 
\begin{align*}
    \norm{ \opA(X(T)) }^2 \le \frac{4 \norm{ X_0 - X_\star }^2 }{T^2} .
\end{align*}

\subsubsection{Correspondence of continuous-time analysis with discrete-time analysis} 

We overview how the continuous-time analysis presented above corresponds to respective analyses of \ref{alg:dual_halpern} and \ref{alg:dual-feg}.  
Precisely, we verify that under the identifcation of terms from \cref{appendix:continuous_time_limit_algorithms}, the following holds true: 
\begin{itemize}
    \item [(i)] The discrete Lyapunov function $V_k$ corresponds to $V(t)$ 
    \item [(ii)] The consecutive difference $V_{k+1} - V_{k}$ corresponds to $\dot{V}(t)$ 
\end{itemize}

\begin{itemize}
\item \ref{alg:dual_halpern} \vspace{.2cm} \\ 
In~\cref{appendix:continuous_time_limit_dual_OHM} we consider~\ref{alg:dual_halpern} with $\opT = 2\opJ_{\frac{\alpha}{2} \opA} - \opI$ for continuous monotone  $\opA$ to derive the continuous-time limit, and we have the identities $x_{k+1} = \opJ_{\frac{\alpha}{2} \opA} y_k$ and $x_{k+1} + \frac{\alpha}{2} \opA x_{k+1} = y_k$.  
Therefore, the corresponding Lyapunov analysis of~\ref{alg:dual_halpern} should use $g_k = \frac{\alpha}{2}\opA x_k$ (instead of $g_k = \tilA x_k$).
Recall that we identify $\alpha k = t$, $\alpha N = T$, $x_k = X(t)$, $y_k = Y(t)$ and $\frac{z_{k}}{\alpha} = Z(t)$.
\begin{itemize}
    \item [(i)] $\frac{V_k}{\alpha^2} \, \longleftrightarrow \, V(t)$. \vspace{.2cm} \\  
        Replacing $g_N = \tilA x_N$ with $\frac{\alpha}{2} \opA x_N$ in the Lyapunov function of \ref{alg:dual_halpern} gives
        \begin{align*}
            V_k = -\frac{N-k-1}{N-k} \sqnorm{z_k + \alpha \opA x_{N}} + \frac{2}{N-k} \inprod{z_k + \alpha \opA x_{N}}{y_k - y_{N-1}}.
        \end{align*}
        Therefore
        \begin{align*}
            \frac{V_k}{\alpha^2} 
            &= - \pr{ 1 - \frac{\alpha}{\alpha(N-k)} } \sqnorm{ \frac{z_k}{\alpha} + \opA x_{N}} 
                + \frac{2}{\alpha(N-k)} \inprod{ \frac{z_k}{\alpha} + \opA x_{N}}{y_k - y_{N-1}} \\
            &= -\pr{ 1 + \mathcal{O}(\alpha) } \sqnorm{ Z(t) + \opA(X(T)) } 
                + \frac{2}{T-t} \inprod{ Z(t) + \opA(X(T)) }{ Y(t) - Y(T-\alpha)}.
        \end{align*}
        From the identity $x_{k+1} + \frac{\alpha}{2} \opA x_{k+1} = y_k$ we have $Y(t) = X(t+\alpha) + \mathcal{O}(\alpha)$, so in the limit $\alpha\to 0$, $Y(t)=X(t)$.
        Then the above equation establishes the desired correspondence, as 
        \[
            (\textrm{RHS}) = -\sqnorm{Z(t) + \opA(X(T))} + \frac{2}{T-t} \inprod{Z(t)+\opA(X(T))}{X(t)-X(T)} .
        \]
        
    \item [(ii)] $\frac{ \frac{1}{\alpha^2} V_{k+1} - \frac{1}{\alpha^2} V_k}{\alpha} \, \longleftrightarrow \, \dot{V}(t)$. \vspace{.2cm} \\
        The Lyapunov analysis of~\ref{alg:dual_halpern} establishes
        \begin{align*}
            V_{k+1} - V_k = -\frac{4}{(N-k)(N-k-1)} \inprod{x_{k+1} - x_N}{g_{k+1} - g_N} .
        \end{align*}
        Replacing $g_k$ with $\frac{\alpha}{2} \opA x_{k}$, dividing both sides by $\alpha^3$ and applying the identifications, we obtain
        \begin{align*}
            \frac{ \frac{1}{\alpha^2} V_{k+1} - \frac{1}{\alpha^2} V_k}{\alpha}
            &= -\frac{4}{\alpha^3 (N-k)(N-k-1)} \inprod{x_{k+1} - x_N}{ \frac{\alpha}{2} \opA x_{k+1} - \frac{\alpha}{2} \opA x_{N} } \\
            &= -\frac{2}{(\alpha N-\alpha k)(\alpha N-\alpha k-\alpha)} \inprod{x_{k+1} - x_N}{ \opA(X(t+\alpha) - \opA (X(T)) } \\
            &= -\frac{2}{(T-t)^2 + \mathcal{O}(\alpha)} \inprod{X(t+\alpha) - X(T)}{ \opA(X(t+\alpha) - \opA (X(T)) }.
        \end{align*}
        Taking $\alpha\to 0$, we have 
        \begin{align*}
            (\textrm{RHS}) = - \frac{2}{(T-t)^2} \inprod{X(t) - X(T)}{\opA(X(t)) - \opA(X(T))} = \dot{V}(t) 
        \end{align*}
        which gives the desired correspondence.
\end{itemize}

\item \ref{alg:dual-feg} \vspace{.2cm} \\ 
    As above, we use the identification $\alpha k =t$, $\alpha N = T$, $x_k = X(t)$ and $z_k = Z(t)$.
    \begin{itemize} 
    \item [(i)] $\frac{V_k}{\alpha} \, \longleftrightarrow \, V(t)$. \vspace{.2cm} \\ 
        The Lyapunov function of \ref{alg:dual-feg} is
        \begin{align*}
            V_k = -\alpha \sqnorm{z_k + \opA x_N} + \frac{2}{N-k} \inprod{z_k + \opA x_N}{x_k - x_N}.
        \end{align*}
        Dividing both sides by $\alpha$ and applying the identifcations, we have
        \begin{align*}
            \frac{V_k}{\alpha} 
            &= - \sqnorm{z_k + \opA x_N} + \frac{2}{\alpha(N-k)} \inprod{z_k + \opA x_N}{x_k - x_N} \\
            &= - \sqnorm{ Z(t) + \opA (X(T))} + \frac{2}{T-t} \inprod{ Z(t) + \opA(X(T))}{X(t) - X(T)},
        \end{align*}
        we get the desired correspondence. 
        
    \item [(ii)] $\frac{ \frac{1}{\alpha} V_{k+1} - \frac{1}{\alpha} V_k}{\alpha} \, \longleftrightarrow \, \dot{V}(t)$. \vspace{.2cm} \\
        The Lyapunov analysis of~\ref{alg:dual-feg} establishes
        \begin{align*}
            V_{k+1} - V_{k}
            &= -\frac{2}{ (N-k)(N-k-1)} \inprod{\opA x_N - \opA x_{k+1/2}}{x_N - x_{k+1/2}}  \\ &\quad
            - \frac{1}{\alpha (N-k)^2}  \Big( \|x_{k+1/2} - x_k\|^2 - \alpha^2 \|\opA x_{k+1/2} - \opA x_k\|^2 \Big) 
        \end{align*}
        Because $x_{k+1/2} = x_k + \cO(\alpha)$, dividing the both sides by $\alpha^2$, we obtain
        \begin{align*}
            \frac{ \frac{1}{\alpha} V_{k+1} - \frac{1}{\alpha} V_k}{\alpha}
            &= -\frac{2}{ \alpha^2(N-k)(N-k-1)} \inprod{\opA x_N - \opA x_{k+1/2}}{x_N - x_{k+1/2}}  \\ &\quad
            - \frac{1}{\alpha^2 (N-k)^2}  \Big( \frac{1}{\alpha}\|x_{k+1/2} - x_k\|^2 - \alpha \|\opA x_{k+1/2} - \opA x_k\|^2 \Big)  \\
            &= -\frac{2}{(T-t)^2 + \mathcal{O}(\alpha)} \inprod{ \opA(X(T)) - \opA(X(t) + \mathcal{O}(\alpha))}{ X(T) - X(t) + \mathcal{O}(\alpha)}  \\ 
            & \quad \underbrace{- \frac{1}{(T-t)^2}  \Big( \frac{1}{\alpha} \mathcal{O}(\alpha^2) - \alpha \| \opA(X(t) + \mathcal{O}(\alpha)) - \opA(X(t)) \|^2 \Big)}_{\cO(\alpha)} .
        \end{align*} 
        Taking $\alpha \to 0$, we have
        \begin{align*}
            (\textrm{RHS}) = - \frac{2}{(T-t)^2} \inprod{X(t) - X(T)}{\opA(X(t)) - \opA(X(T))} = \dot{V}(t)
        \end{align*}
        which gives the desired correspondence.
\end{itemize}
\end{itemize}

\subsection{Generalization to differential inclusion with maximally monotone $\opA$}
\label{appendix:continuous-time-differential-inclusion}

So far, we have analyzed the Dual-Anchor ODE with respect to Lipschitz continuous $\opA$.
In this section, we deal with its extension to differential inclusion, which is a generalized form of continuous-time model, covering the case of general (possibly set-valued) maximally monotone operators:
\begin{align} \label{ode:dual-anchor_inclusion}
\begin{split}
    \pmat{ \dot{X}(t) \\ \dot{Z}(t) }
     & \in - \pmat{ Z(t) + \opA(X(t)) \\ \frac{1}{T-t} Z(t) + \frac{1}{T-t} \opA(X(t)) } .  
\end{split} 
\end{align}
We say that $(X(t), Z(t))$ is a solution to this differential inclusion if it satisfies~\eqref{ode:dual-anchor_inclusion} almost everywhere in $t$. (So unlike in ODEs, we do not require $X, Z$ to be differentiable everywhere, but only require absolute continuity, which implies differentiability almost everywhere.)
Although~\eqref{ode:dual-anchor_inclusion} is technically not an ODE, with a slight abuse of notation, we will often refer to it as (generalized) Dual-Anchor ODE throughout the section.

\subsubsection{Existence of a solution to the generalized Dual-Anchor ODE} 
\begin{theorem} \label{appendix_thm:well-posedness_for_general_maximal_monotone}
    Let $\opA \colon \reals^d \rightrightarrows \reals^d$ be maximally monotone and let $T > 0$. 
    Given the initial conditions $X(0)=X_0 \in \dom \opA$ and $Z(0)=0$, there exists a solution to the generalized Dual-Anchor ODE, i.e., an absolutely continuous curve $\pmat{ X \\ Z } \colon [0,T) \to \reals^{2d}$ that satisfies \eqref{ode:dual-anchor_inclusion} for $t \in (0, T)$ almost everywhere. 
\end{theorem}

We proceed with similar ideas as in \citep[Chpater 3]{Jean-PierreArrigo1984_differential} and \citep[Appendix B.2]{SuhParkRyu2023_continuoustime}, i.e., we construct a sequence $\{(X_\yap (t), Z_\yap (t))\}_{\delta > 0}$ of solutions to ODEs approximating~\eqref{ode:dual-anchor_inclusion}, with $\opA$ replaced by Yosida approximations $\opA_\yap$ (which are much better-behaved, being single-valued and Lipschitz continuous).
Then we show that the limit of $(X_\yap (t), Z_\yap (t))$ in $\yap \to 0$ is a solution to the original inclusion~\eqref{ode:dual-anchor_inclusion}.

Below we present some well-known facts needed to rigorously establish the approximation argument.

\begin{lemma} \label{appendix_lemma:properties_of_Yosida_approximation}
Let $\opA\colon \reals^d \rightrightarrows \reals^d$ be maximally monotone.
Then for $\delta > 0$, the Yosida approximation operators
    $$ \opA_{\yap} = \frac{1}{\yap} \pr{ \opI - \opJ_{\yap \opA} } = \frac{1}{\yap} \pr{ \opI - \pr{ \opI + \yap \opA }^{-1} } $$
satisfy the followings:
\begin{itemize}
    \item [(i)] $\forall x \in \OurSpace$, $\opA_\yap(x) \in \opA(\opJ_{\yap \opA} x)$.
    \item [(ii)] $\opA_\yap \colon \reals^d \to \reals^d$ is singled-valued, $\frac{1}{\yap}$-Lipschitz continuous and maximally monotone.
    \item [(iii)] $\norm{\opA_\yap(x)} \le \norm{m(\opA(x))}$.
    \item [(iv)] $\lim_{\yap\to0+}\opA_\yap (x) = m(\opA(x)) := \argmin_{u \in \opA(x)} \|u\|$. (The minimum norm element of the set $\opA(x)$ is well-defined because $\opA(x)$ is a closed and convex set due to maximal monotonicity of $\opA$.)
\end{itemize}
\end{lemma}

\begin{proof}
    See \citep[Chpater 3.1, Theorem 2]{Jean-PierreArrigo1984_differential}. 
\end{proof}
 
\begin{lemma} \label{appendix_lemma:existence of subsequence}
Let $\{x_n(\cdot)\}_n$ be a sequence of absolutely continuous functions $x_n\colon I \to \reals^d$, where $I \subset \reals$ is an interval of finite length.
Suppose that
\begin{itemize}
    \item [(i)] $\forall t \in I$, the set $\{x_n(t)\}_{n}$ is bounded.
    \item [(ii)] There exists $M > 0$ such that $\norm{\dot{x}_n(t)} \le M$ almost everywhere in $t \in I$.
\end{itemize}
Then there exists a subsequence $x_{n_k}(\cdot)$ such that  
\begin{itemize}
    \item [(1)] $x_{n_k}(\cdot)$ uniformly converges to $x(\cdot)$ over compact subsets of $I$.
    \item [(2)] $x(\cdot)$ is absolutely continuous, so $\dot{x}$ exists almost everywhere.
    \item [(3)] $\dot{x}_{n_k}(\cdot)$ converges weakly to $\dot{x}(\cdot)$ in $L^2(I,\reals^d)$.
\end{itemize}
\end{lemma}

\begin{proof}
This is a simplified version of \citep[Chapter 0.3, Theorem~4]{Jean-PierreArrigo1984_differential}. For completeness, we outline the proof below.

The family $\{x_n(\cdot)\}_n$ is pointwisely bounded (condition~(i)) and equicontinuous as
\[
    \norm{x_n(t) - x_n(s)} = \norm{\int_s^t \dot{x}_n (\tau) \, d\tau} \le M |t-s|, \quad \forall s, t \in I .
\]
By Arzel\`a-Ascoli theorem, there exists a subsequence $x_{n_k}(\cdot)$ such that $x_{n_k} \to x$ uniformly over compact subsets of $I$, where $x \colon I \to \reals^d$ is continuous (so far we have not shown absolute continuity).
Now observe that $\{\dot{x}_{n_k}(\cdot)\}_k$ is a bounded set within $L^\infty (I, \reals^d) = L^1 (I, \reals^d)^*$.
Therefore, by Banach-Alaoglu theorem, $\dot{x}_{n_k}(\cdot)$ again has a subsequence that converges in the weak-$*$ sense.
By replacing $n_k$ with this subsequence if necessary, assume without loss of generality that $\dot{x}_{n_k} \xrightarrow{\text{w}*} y \in L^\infty (I, \reals^d)$.
This means that for any $c \in L^1 (I, \reals^d)$,
\begin{align}
\label{eqn:xnk-dot-weak-star-convergence}
    \lim_{k \to \infty} \int_I \inprod{\dot{x}_{n_k}(\tau)}{c(\tau)} \, d\tau = \int_I \inprod{y(\tau)}{c(\tau)} \, d\tau .
\end{align}
For any $s, t \in I$ and $\va \in \reals^d$, taking $c(\tau) = 1_{[s,t]}(\tau) \va$ in particular, we have
\[
    \lim_{k \to \infty} \int_s^t \inprod{\dot{x}_{n_k}(\tau)}{\va} \, d\tau = \int_s^t \inprod{y(\tau)}{\va} \, d\tau = \inprod{\int_s^t y(\tau)\,d\tau}{\va} .
\]
On the other hand, the left hand side equals $\lim_{k\to\infty} \inprod{x_{n_k}(t) - x_{n_k}(s)}{\va} = \inprod{x(t) - x(s)}{\va}$.
Since $\va \in \reals^d$ is arbitrary, this shows
\[
    x(t) - x(s) = \int_s^t y(\tau) \, d\tau , \quad y \in L^\infty (I, \reals^d)
\]
so $x$ is absolutely continuous (in fact, Lipschitz continuous) and $\dot{x} = y$ almost everywhere.
Finally, note that $L^2 (I, \reals^d) \subset L^1 (I, \reals^d)$ (as $I$ is of finite length), so we can take $c \in L^2 (I, \reals^d)$ in~\eqref{eqn:xnk-dot-weak-star-convergence}, showing that $\dot{x}_{n_k} \to y$ weakly in $L^2 (I, \reals^d)$.
\end{proof}

Using the above results, we can now prove~\cref{appendix_thm:well-posedness_for_general_maximal_monotone}.
For $\delta > 0$, the ODE
\begin{align} \label{ode:Yosida_approximated_dynamics}
    \begin{split}
        \pmat{ \dot{X}_\delta(t) \\ \dot{Z}_\delta(t) }
         & = - \pmat{ Z_\delta(t) + \opA_\delta(X_\delta(t)) \\ \frac{1}{T-t} Z_\delta(t) + \frac{1}{T-t} \opA_\delta(X_\delta(t))  }  
    \end{split} 
\end{align}
with initial conditions $X_\delta(0) = X_0 \in \dom \opA$ and $Z_\delta(0) = 0$ has a unique $\cC^1$ solution $(X_\delta (t), Z_\delta (t))$ for $t \in [0,T]$ by~\cref{appendix_thm:well-posedness_for_Lipschitz_continuous_A} and \cref{lemma:ode_terminal_infomation}.
Additionally by~\cref{appendix_lemma:bound_of_dotX}, $X_\delta$ and $Z_\delta$ satisfy
\begin{align}
\label{eqn:appendix-inclusion-approximate-solution-dot-bounds}
    \norm{ \dot{X}_\delta(t)} \le \frac{T-t}{T} \norm{ m(\opA(X_0)) } \le \norm{ m(\opA(X_0)) }, 
    \qquad 
    \norm{ \dot{Z}_\delta(t)} \le \frac{1}{T} \norm{ m(\opA(X_0)) } .
\end{align}
Therefore, we can apply~\cref{appendix_lemma:existence of subsequence} to obtain a positive sequence $\delta_n \to 0$ such that
\begin{itemize}
    \item $\pmat{X_{\delta_n}\\Z_{\delta_n}} \to \pmat{X\\Z}$ uniformly on $[0, T]$, where $\pmat{X\\Z}$ is absolutely continuous,
    \item $\pmat{\dot{X}_{\delta_n} \\ \dot{Z}_{\delta_n}} \to \pmat{\dot{X} \\ \dot{Z}}$ weakly in $L^2([0,T], \reals^{2d})$.
\end{itemize}
Now, we wish to show that $(X(t),Z(t))$ is a solution to~\eqref{ode:dual-anchor_inclusion}.
Define $G_{\yap}, G \colon [0,T] \to \OurSpace$ as
\begin{align*}
    G_{\yap}(t) = Z_{\yap}(t) + \dot{X}_{\yap}(t) , \quad G(t) = Z(t) + \dot{X}(t) .
\end{align*}  
By construction, 
\begin{align*} 
    G_{\yap_{n}}(\cdot) \to G(\cdot) \quad \textit{weakly } \, \text{in } \, L^2([0,T], \reals^d) .
\end{align*}
Note that because $(X_\yap, Z_\yap)$ solves~\eqref{ode:Yosida_approximated_dynamics} and \cref{appendix_lemma:properties_of_Yosida_approximation}~(i) holds, we have
\begin{align*}
    G_{\yap_{n}}(t) = -\opA_\delta (X_\delta (t)) \in -\opA(\opJ_{\yap_{n}\opA}(X_{\yap_{n}}(t))).
\end{align*}
On the other hand, because $Z_{\delta_n}(0) = 0$, by~\eqref{eqn:appendix-inclusion-approximate-solution-dot-bounds} we have
\begin{align*}
    \norm{ Z_{\yap_n}(t) } 
    = \norm{ \int_{0}^{t} \dot{Z}_{\yap_n}(s)  ds }
    \le \int_{0}^{t} \norm{ \dot{Z}_{\yap_n}(s) }  ds
    \le \int_{0}^{t} \frac{\norm{ \opA_{\yap_n}(X_0)}}{T}  ds 
    \le \norm{ \opA_{\yap_n}(X_0)}
    \le \norm{ m(\opA(X_0))} .
\end{align*}
Together with the norm bound on $\dot{X}$ in~\eqref{eqn:appendix-inclusion-approximate-solution-dot-bounds}, this implies that for all $t \in [0, T]$,
\begin{align*}
    \norm{X_{\yap_n}(t) - \opJ_{\yap_n\opA}(X_{\yap_n}(t))} 
        = \yap_n \norm{\opA_{\yap_n}(X_{\yap_n}(t))} 
        = \yap_n \norm{ Z_{\yap_n}(t) + \dot{X}_{\yap_n}(t) } 
        \le \yap_n \left( \norm{ Z_{\yap_n}(t) } + \| \dot{X}_{\yap_n}(t) \| \right) \le 2 \yap_n \norm{ m(\opA(X_0))}.
\end{align*}
As $X_{\yap_n}(\cdot)$ converges uniformly to $X(\cdot)$, the above result shows that $\opJ_{\yap_n \opA} (X_{\yap_n}(\cdot))$ uniformly converges to $X(\cdot)$ as well. In particular, 
\begin{align*} 
    \opJ_{\yap_n \opA} (X_{\yap_n}(\cdot)) \to X(\cdot) \quad \textit{strongly } \, \text{in } \, L^2([0,T], \reals^d) .
\end{align*}

Now, define $\mathcal{A} \colon L^2([0,T],\OurSpace) \to L^2([0,T],\OurSpace)$ by 
\[
    \cA(y) = \left\{ u \in L^2([0,T],\OurSpace) \,|\, u(t) \in \opA(y(t)) \,\, \text{for a.e.} \,\, t\in [0,T] \right\} .
\]
$\cA$ is monotone because $\opA$ is monotone: if $u \in \cA(y), v \in \cA(z)$ then
\[
    \inprod{u - v}{y - z}_{L^2 ([0,T], \reals^d)} = \int_0^T \inprod{u(t) - v(t)}{y(t) - z(t)} \, dt \ge 0 \quad \text{since} \,\, u(t) \in \opA (y(t)), v(t) \in \opA (z(t)) \,\, \text{a.e. in} \,\, t \in [0,T] . 
\]
If $\cI \colon L^2([0,T],\OurSpace) \to L^2([0,T],\OurSpace)$ is the identity operator, then $\cI + \cA$ is surjective: for any $u \in L^2([0,T],\OurSpace)$, we have $y(t) = \JA(u(t)) \in L^2([0,T], \OurSpace)$ because $\JA$ is nonexpansive, and then $u \in \cA(y)$ by construction.
By Minty's surjectivity theorem \cite{Minty1962_monotone}, this implies that $\cA$ is maximally monotone.
Now because $\opJ_{\delta_n \opA} (X_{\delta_n})$ converges to $X$ strongly and $-G_{\yap_n} \in \cA (\opJ_{\delta_n \opA} (X_{\delta_n}))$ converges to $-G$ weakly, and the graph of a maximally monotone operator is closed under the strong-weak topology \citep[Proposition~20.38]{bauschke2011convex}, we conclude that $-G \in \cA (X)$, i.e.,
\[
    G(t) = Z(t) + \dot{X}(t) \in -\opA (X(t)) \iff \dot{X}(t) \in -\left( Z(t) + \opA (X(t)) \right) \quad \text{a.e. in} \,\, t \in [0,T] .
\]
Finally, because $(T-t) \dot{Z}_\delta (t) = \dot{X}_\delta (t)$ for all $\delta > 0$ we have
\[
    (T-t) \dot{Z}_{\delta_n} = \dot{X}_{\delta_n} \xrightarrow{\text{w}} \dot{X} \quad \text{in} \,\, L^2([0,T], \reals^d) .
\]
On the other hand, we had $\dot{Z}_{\delta_n} \xrightarrow{\text{w}} \dot{Z}$ in $L^2 ([0,T], \reals^d)$, which implies
\[
    (T-t) \dot{Z}_{\delta_n} \xrightarrow{\text{w}} (T-t) \dot{Z} 
\]
because for any $c(t) \in L^2([0,T], \reals^d)$, we have $(T-t) c(t) \in L^2 ([0,T], \reals^d)$ as well.
By uniqueness of the (weak) limit, we have $\dot{X} = (T-t) \dot{Z}$ a.e., so
\[
    \dot{Z}(t) = \frac{1}{T-t} \dot{X}(t) \in -\frac{1}{T-t} \left( Z(t) + \opA (X(t)) \right) \quad \text{a.e. in} \,\, t \in [0,T].
\]
This shows that $\pmat{X\\Z}$ is indeed a solution of~\eqref{ode:dual-anchor_inclusion}.

\subsubsection{Behavior at the terminal time $t=T$} \label{appendix:proof_of_ode_terminal_infomation_inclusion}

So far, we have successfully constructed an absolutely continuous solution $\pmat{X\\Z} \colon [0,T] \to \reals^{2d}$.
In this section we show that $X(t), Z(t)$ have two very favorable properties at the terminal time: $X$ has left derivative 0, and $Z(T) \in -\opA (X(T))$.
To show this, we need: 

\begin{lemma} \label{appendix_lemma:bound_of_dotX_inclusion}
    Let $\opA \colon \reals^d \rightrightarrows \reals^d$ be maximally monotone and let $\pmat{ X \\ Z }$ be the solution of~\eqref{ode:dual-anchor_inclusion} constructed in~\cref{appendix_thm:well-posedness_for_general_maximal_monotone}.  
    Then for almost every $t \in [0,T)$, 
    \begin{align*}
        \frac{1}{(T-t)^2} \norm{ \dot{X}(t) }^2 \le \frac{1}{T^2} \norm{ m(\opA(X_0)) }^2.  
    \end{align*}
\end{lemma}

\begin{proof}
Let $\delta_n > 0$ be the sequence taken in the proof of~\cref{appendix_thm:well-posedness_for_general_maximal_monotone}, for which the solutions $(X_{\delta_n}, Z_{\delta_n})$ to~\eqref{ode:Yosida_approximated_dynamics} with $\delta=\delta_n$ converge uniformly to $(X, Z)$ and $(\dot{X}_{\delta_n}, \dot{Z}_{\delta_n})$ converge weakly to $(\dot{X}, \dot{Z})$. 
We have shown $\norm{ \dot{X}_{\delta_{n}}(t) } \le \frac{T-t}{T} \norm{ m(\opA(X_0)) }$ in~\cref{eqn:appendix-inclusion-approximate-solution-dot-bounds}. 
Thus, the proof is done once we show 
\[
    \norm{\dot{X}(t)} \le \limsup_{n\to\infty} \norm{ \dot{X}_{\delta_{n}}(t) } 
\]
for almost every $t \in [0,T]$. 
(This would be straightforward if $\dot{X}$ were the pointwise limit of $\dot{X}_{\delta_n}$, but it is not the case; it is the weak limit. We need a careful argument, as presented below.)

Let $D$ be any measurable subset of $[0,T]$. 
Since $\dot{X}_{\yap_{n_k}} \to \dot{X}$ weakly in $L^2([0,T],\OurSpace)$ and $\chi_{D}\dot{X} \in L^2([0,T],\OurSpace)$, we have
\begin{align}
\begin{aligned} \label{appendix_ineq:limsup_dotX_1}
    \int_{D} \norm{\dot{X}(t)}^2 dt 
    & = \int_{0}^{T} \inner{\dot{X}(t)}{ \chi_{D}(t)\dot{X}(t) } dt 
    = \lim_{n\to\infty} \int_{0}^{T} \inner{\dot{X}_{\yap_{n}}(t)}{ \chi_{D}(t)\dot{X}(t) } dt \\
    & = \lim_{n\to\infty} \int_{D} \inner{\dot{X}_{\yap_{n}}(t)}{ \dot{X}(t)} dt     
    \le \limsup_{n\to\infty} \int_{D} \norm{\dot{X}_{\yap_{n}}(t)} \norm{\dot{X}(t)} dt .
\end{aligned}
\end{align}
Now from $\norm{\dot{X}_{\yap_{n}}(\cdot)} \le \norm{ m(\opA(X_0)) }$ and $ \norm{\dot{X}(\cdot)} \in L^2 ([0,T],\OurSpace) \subset  L^1([0,T],\OurSpace)$, we obtain
\begin{align*}
    \norm{\dot{X}_{\yap_{n}}(\cdot)} \norm{\dot{X}(\cdot)} \le \norm{ m(\opA(X_0)) } \norm{\dot{X}(\cdot)} \in L^1([0,T], \OurSpace).
\end{align*}
Thus by reverse Fatou Lemma,
\begin{align*}
    \limsup_{n\to\infty} \int_{D} \norm{\dot{X}_{\yap_{n}}(t)} \norm{\dot{X}(t)} dt \le \int_{D} \limsup_{n\to\infty} \norm{\dot{X}_{\yap_{n}}(t)} \norm{\dot{X}(t)} dt.
\end{align*}
Combining the above inequality with \eqref{appendix_ineq:limsup_dotX_1} we obtain
\begin{align*}
    \int_D \pr{ \limsup_{n\to\infty} \norm{\dot{X}_{\yap_{n}}(t)} - \norm{\dot{X}(t)} } \norm{\dot{X}(t)} dt \ge 0.
\end{align*}
As $D$ was an arbitrary measurable subset of $[0,T]$, we conclude that for almost every $t \in [0,T]$,
\begin{align*}
    \pr{ \limsup_{n\to\infty} \norm{\dot{X}_{\yap_{n}}(t)} - \norm{\dot{X}(t)} } \norm{\dot{X}(t)} \ge 0
    \quad \Longrightarrow \quad 
    \norm{\dot{X}(t)} \le \limsup_{n\to\infty} \norm{\dot{X}_{\yap_{n}}(t)} .
\end{align*} 
\end{proof}

\begin{corollary} \label{appendix_lemma:terminal_point_information_inclusion}
    Let $\opA \colon \reals^d \rightrightarrows \reals^d$ be maximally monotone and let $\pmat{ X \\ Z }$ be the solution of~\eqref{ode:dual-anchor_inclusion} constructed in~\cref{appendix_thm:well-posedness_for_general_maximal_monotone}. 
    Then the following holds true:
    \begin{align*} 
        \lim_{t \to T^-} \frac{X(t) - X(T)}{t-T} = 0, \qquad -Z(T) \in \opA(X(T)).
    \end{align*} 
\end{corollary}

\begin{proof} 
    Using~\cref{appendix_lemma:bound_of_dotX_inclusion}, we obtain
    \begin{align*}
        \lim_{t\to T^-} \norm{ \frac{X(t) - X(T)}{t-T} }
        \le \lim_{t\to T^-} \int_{t}^{T} \frac{\norm{\dot{X}(s)}}{T-t} ds 
        \le \lim_{t\to T^-} \int_{t}^{T} \frac{\norm{ \dot{X}(s) }}{T-s}  ds
        \le \lim_{t\to T^-} \int_{t}^{T} \frac{\norm{  m(\opA(X_0))  }}{T} ds
        = 0 ,
    \end{align*}
    which proves the first property.
    
    Next, because $(X(t), Z(t))$ satisfies~\eqref{ode:dual-anchor_inclusion} a.e.\ and $\|\dot{X}(t)\| \le \frac{T-t}{T} \norm{m(\opA(X_0))}$ a.e.\ (by~\cref{appendix_lemma:bound_of_dotX_inclusion}), we can take a sequence $t_k \in (0,T)$ such that $\lim_{k\to\infty} t_k = T$ and  
    \begin{align*}
        -\pr{ \dot{X}(t_k) + Z(t_k) } \in \opA(X(t_k)), \qquad
        \norm{ \dot{X}(t_k) } \le \frac{T-t_k}{T} \norm{ \opA(X_0) }.
    \end{align*}
    Then we have $\lim_{k\to\infty} \norm{ \dot{X}(t_k) } \le \lim_{k\to\infty} \frac{T-t_k}{T} \norm{ \opA(X_0) } = 0$.
    On the other hand, $X(t_k) \to X(T)$ and $Z(t_k) \to Z(T)$ because $X, Z$ are continuous.
    Finally, because the graph of $\opA$ is closed in $\reals^d \times \reals^d$ \citep[Proposition~20.38]{bauschke2011convex}, we conclude 
    \begin{align*}
        -Z(T) = -\lim_{k\to\infty} \pr{ \dot{X}(t_k) + Z(t_k) } \in \opA\pr{ \lim_{k\to\infty}X(t_k) } = \opA(X(T)).
    \end{align*}

\end{proof}

\subsubsection{Convergence analysis} \label{appendix:proof_of_continuous_convergnece_rate_inclusion}

Based on the previous analyses, we can prove that the constructed solution has a Lyapunov function similar to that of~\cref{thm:continuous_convergnece_rate} (the case of Lipschitz continuous operators).

\begin{theorem} \label{thm:continuous_convergnece_rate_inclusion}
Let $\opA\colon \reals^d \rightrightarrows \reals^d$ be maximally monotone and $\zer \opA \ne \emptyset$. 
Let $\pmat{ X \\ Z }$ be the solution of~\eqref{ode:dual-anchor_inclusion} constructed in~\cref{appendix_thm:well-posedness_for_general_maximal_monotone}.
Then the function $V\colon [0,T) \to \reals$ defined by
\begin{align*}
    V(t) &= -\norm{ Z(t) - Z(T) }^2 + \frac{2}{T-t} \inner{ Z(t) - Z(T) }{  X(t) - X(T) }  
\end{align*}
is nonincreasing, and $\lim_{t\to T^-} V(t) = 0$.
Furthermore, for $X_\star \in \zer \opA$,
\begin{align*}
    \norm{ m(\opA(X(T))) }^2
    \le \norm{ Z(T) }^2 
    \le \frac{4 \norm{ X_0 - X_\star }^2 }{T^2} 
\end{align*}
where $m(\opA(X(T)))$ is the minimum norm element of $\opA(X(T))$.  
    
\end{theorem}

\begin{proof} 
Let $\delta_n > 0$ be the sequence taken in the proof of~\cref{appendix_thm:well-posedness_for_general_maximal_monotone}. Let
\begin{align*}
    V_{\delta_{n}}(t) = -\norm{ Z_{\delta_{n}}(t) - Z_{\delta_{n}}(T) }^2 + \frac{2}{T-t} \inner{ Z_{\delta_{n}}(t) - Z_{\delta_{n}}(T)  }{  X_{\delta_{n}}(t) - X_{\delta_{n}}(T) }
\end{align*}
for $t\in [0,T)$. 
As $\pmat{ X_{\delta_{n}} \\  Z_{\delta_{n}} }$ converges uniformly on $[0,T]$, for all $t \in [0,T)$ we have
\begin{align*}
    V(t) = \lim_{n\to\infty} V_{\delta_{n}}(t). 
\end{align*}
Observe that $V_{\delta_n}(\cdot)$ is nonincreasing on $[0,T)$ for each $n$ (by the result of~\cref{thm:continuous_convergnece_rate} and $Z_{\delta_n}(T) = -\opA(X_{\delta_n}(T))$, which follows from~\cref{appendix_lemma:terminal_point_information_inclusion} and uniqueness of the solution in the Lipschitz case).
Therefore, for any $s,t \in [0,T)$ such that $s<t$, we have
\begin{align*}
    V(s) = \lim_{n\to\infty} V_{\delta_{n}}(s) \ge \lim_{n\to\infty} V_{\delta_{n}}(t) = V(t) ,
\end{align*}
which shows that $V$ is nonincreasing. 
Furthermore, from $\lim_{t\to T^-}Z(t) = Z(T)$ (continuity of $Z$) and~\cref{appendix_lemma:terminal_point_information_inclusion}, we have
\begin{align*}
    \lim_{t\to T^-} V(t) 
    = -\norm{ \lim_{t\to T^-} Z(t) - Z(T) }^2 - 2 \inner{ \lim_{t\to T^-} Z(t) - Z(T) }{ \lim_{t\to T^-}\frac{ X(t) - X(T) }{ t-T } }
    = 0 .
\end{align*}
Therefore
\begin{align*}
    0 = \lim_{t\to T^-} V(t) \le V(0) = -\norm{ Z(T) }^2 - \frac{2}{T} \inner{ Z(T) }{ X_0 - X(T) } ,
\end{align*}
and~\cref{lemma:convergence-proof-last-step} gives
\begin{align*}
    \norm{ Z(T) }^2 \le \frac{4\norm{ X_0 - X_\star }^2}{T^2} . 
\end{align*}
Finally, because $-Z(T) \in \opA(X(T))$ by~\cref{appendix_lemma:terminal_point_information_inclusion}, the left hand side is lower bounded by $\norm{ m(\opA(X(T))) }^2$, which gives the desired convergence rate. 

\end{proof}
\newpage
\section{Omitted details from~\cref{section:experiments}}
\label{section:appendix-experiment-omitted-details}

\subsection{Extragradient algorithm specification}

In Figures~\ref{subfig:2D_bilinear_trajectories}, \ref{subfig:ouyang_construction_performance}, we display comparison with Extragradient (EG) \cite{Korpelevich1976_extragradient}.
For completeness, we specify its definition here:
\begin{align*}
    x_{k+1/2} & = x_k - \alpha \sop{x_k} \\
    x_{k+1}   & = x_k - \alpha \sop{x_{k+1/2}} .
\end{align*}

\subsection{Details of worst-case construction from~\citet{OuyangXu2021_lower}}
\label{subsection:appendix-experiment-Ouyang-construction}

For~\cref{subfig:ouyang_construction_performance}, we use the following construction due to~\citet{OuyangXu2021_lower}, which has been used to establish complexity lower bounds:
\[
A = \frac{1}{4} \begin{bmatrix}
  &  &  & -1 & 1 \\
  &  & \iddots & \iddots & \\
  & -1 & 1 \\
  -1 & 1\\
  1
\end{bmatrix} \in \reals^{n\times n}, \quad
b = \frac{1}{4} \begin{bmatrix}
1 \\ 1 \\ \vdots \\ 1 \\ 1 
\end{bmatrix} \in \reals^n, \quad
g = \frac{1}{4} \begin{bmatrix}
0 \\ 0 \\ \vdots \\ 0 \\ 1
\end{bmatrix} \in \reals^n,
\]
and $G = 2A^\intercal A$.
In \citet{OuyangXu2021_lower}, it is shown that $\|A\| \le \frac{1}{2}$ (so $\|G\| \le \frac{1}{2}$), so that the bilinear function
\[
    \vL(u,v) = \frac{1}{2} u^\intercal G u - g^\intercal u - \inprod{Au - b}{v}
\]
is $1$-smooth.

\subsection{H-dual algorithms produce identical terminal iterates for linear operators}
\label{subsection:appendix-experiment-identical-terminal-iterates}

In this section, we only consider explicit algorithms written in the form
\begin{align}
\label{eqn:appendix-explicit-algorithm-h-matrix-form}
    x_{k+1} = x_k - \sum_{j=0}^{k} h_{k+1,j} \opA x_j 
\end{align}
where $\opA: \reals^d \to \reals^d$ is a single-valued operator.
If $\opA = \nabla_\pm \vL$, \ref{alg:feg} and \ref{alg:dual-feg} are instances of this class if we identify $1/2$-indexed iterates with integer iterates in increasing order.
If $\opA = \nabla f$, constant step-size gradient descent or Nesterov's accelerated gradient method \cite{Nesterov1983_method} are instances of this class.
However, we do not need to restrict the discussion of this section to a specific operator class, because the following result does not rely on any particular property of the operator except for linearity.

\begin{proposition}
Suppose $\opA: \reals^d \to \reals^d$ is linear, i.e., there exists a matrix $A \in \reals^{d\times d}$ such that $\opA x = Ax$ for all $x \in \reals^d$.
Let $N \ge 1$ and consider an algorithm defined by~\eqref{eqn:appendix-explicit-algorithm-h-matrix-form} for $k=0,1,\dots,N-1$.
Then its H-dual algorithm, defined by 
\begin{align*}
    \hat{x}_{k+1} = \hat{x}_k - \sum_{j=0}^{k} h_{N-j,N-k-1} \opA x_j ,
\end{align*}
satisfies
\begin{align*}
    x_N = \hat{x}_N 
\end{align*}
provided that $x_0 = \hat{x}_0$.
\end{proposition}

\begin{proof}

Because $\opA$ is linear, we can explicitly describe $x_k$ from the algorithm~\eqref{eqn:appendix-explicit-algorithm-h-matrix-form} via matrix polynomial in $A$ and $x_0$.
For $k=1,2,3$, we have
\begin{align*}
    x_1 & = x_0 - h_{1,0} Ax_0 \\
    x_2 & = x_1 - h_{2,0} Ax_0 - h_{2,1} Ax_1 \\
    & = (x_0 - h_{1,0} Ax_0) - h_{2,0} Ax_0 - h_{2,1} A (x_0 - h_{1,0} Ax_0) \\
    & = x_0 - (h_{1,0} + h_{2,0} + h_{2,1}) Ax_0 + h_{1,0}h_{2,1} A^2 x_0
\end{align*}
and
\begin{align*}
    x_3 & = x_2 - h_{3,0} Ax_0 - h_{3,1} Ax_1 - h_{3,2} Ax_2 \\
    & = (x_0 - (h_{1,0} + h_{2,0} + h_{2,1}) Ax_0 + h_{1,0}h_{2,1} A^2 x_0) - h_{3,0} Ax_0 - h_{3,1} A (x_0 - h_{1,0} Ax_0) \\
    & \quad\quad - h_{3,2} A (x_0 - (h_{1,0} + h_{2,0} + h_{2,1}) Ax_0 + h_{1,0}h_{2,1} A^2 x_0) \\
    & = x_0 - (h_{1,0} + h_{2,0} + h_{3,0} + h_{2,1} + h_{3,1} + h_{3,2}) Ax_0 \\
    & \quad\quad + (h_{1,0}h_{2,1} + h_{1,0}h_{3,1}  + h_{1,0}h_{3,2} + h_{2,0}h_{3,2} + h_{2,1}h_{3,2}) A^2 x_0 - h_{1,0}h_{2,1}h_{3,2} A^3 x_0 .
\end{align*}
The following lemma formally generalizes the pattern for these expressions:
\begin{lemma}
    \label{lemma:polynomial_expression_of_algorithm}
    For linear operator $\opA x = Ax$, the iterates of~\eqref{eqn:appendix-explicit-algorithm-h-matrix-form} satisfy
    \begin{align}
    \label{eqn:appendix-terminal-iterate-xk-matrix-polynomial-form}
        x_k = \sum_{m=0}^{k} (-1)^m P(k,m) A^m x_0.
    \end{align}
    for $k=1,\dots,N$, where $P(k,0) = 1$ and for $m=1,\dots,k$,
    \begin{align}
    \label{eqn:appendix-terminal-iterate-P-definition}
        P(k,m) = \sum_{\substack{ i(1)\le j(2), \dots, i(m-1)\le j(m) \\ i(m) \le k}} \prod_{\ell=1}^m h_{i(\ell),j(\ell)} .
    \end{align}
\end{lemma} 

\begin{proof}[Proof of \cref{lemma:polynomial_expression_of_algorithm}]
We first clarify the definition of $P(k,m)$ by providing some examples.
First, we have
\begin{align*}
    P(1,1) = h_{1,0}, \qquad P(2,1) = h_{1,0} + h_{2,0} +  h_{2,1}, \qquad P(3,1) = h_{1,0} + h_{2,0} + h_{3,0} + h_{2,1} + h_{3,1} + h_{3,2} .
\end{align*}
This is because in~\eqref{eqn:appendix-terminal-iterate-P-definition}, when $m=1$, the constraints $i(1)\le j(2), \dots, i(m-1)\le j(m)$ become vacuous so we add all $h_{i,j}$'s with $i \le k$ (note that always $j \le i-1$).
Next, we have
\begin{align*}
    P(2,2) = h_{1,0} h_{2,1} , \qquad P(3,2) = h_{1,0}h_{2,1} + h_{1,0}h_{3,1}  + h_{1,0}h_{3,2} + h_{2,0}h_{3,2} + h_{2,1}h_{3,2} .
\end{align*}
Observe that by definition of $P$, we have to choose the $(i,j)$ pairs within the product to satisfy the constraint $i(1) \le j(2)$, so if we choose $(i(1),j(1)) = (1,0)$ then we must have $j(2)\ge 1$, so $P(k,1)$ contains the products of $h_{1,0}$ with $h_{i,j}$'s satisfying $k\ge i > j \ge 1$. 
Similarly, if we choose $i(1) = 2$ then one must have $j(2) \ge 2$, and the only $(i,j)$ pair satisfying $3 \ge i > j \ge 2$ is $(i,j) = (3,2)$, so $h_{2,0}h_{3,2}$ and $h_{2,1}h_{3,2}$ are the only terms within $P(3,2)$ that involve $h_{2,0}$ and $h_{2,1}$, respectively.

Because the constraint $j(\ell) < i(\ell)$ is implicit, the definition~\eqref{eqn:appendix-terminal-iterate-P-definition} is requiring the $(i(\ell),j(\ell))$ pairs consisting the product to satisfy
\begin{align*}
    1 \le i(1) \le j(2) < \cdots < j(m-1) < i(m-1) \le j(m) < i(m) \le k .
\end{align*}
Thus, when $m=k$, the only possible choice of indices is $i(\ell) = \ell$ and $j(\ell) = \ell-1$, so
\begin{align*}
    P(k,k) = h_{1,0}h_{2,1} \cdots h_{k,k-1}
\end{align*}
and in particular, $P(3,3) = h_{1,0}h_{2,1}h_{3,2}$.

The above observations already show that~\eqref{eqn:appendix-terminal-iterate-xk-matrix-polynomial-form} holds for $k=1,2,3$.
To prove the general statement, we use induction.
Suppose~\eqref{eqn:appendix-terminal-iterate-xk-matrix-polynomial-form} holds up to $k$, and consider
\begin{align*}
    x_{k+1} = x_k - \sum_{n=0}^k h_{k+1,n} Ax_n .
\end{align*}
By induction hypothesis, we have
\begin{align*}
    x_{k+1} & = \sum_{m=0}^k (-1)^m P(k,m) A^m x_0 - \sum_{n=0}^k h_{k+1,n} A \left( \sum_{m=0}^n (-1)^m P(n,m) A^m x_0 \right) \\
    & = \sum_{m=0}^k (-1)^m P(k,m) A^m x_0 + \sum_{n=0}^k \sum_{m=0}^n (-1)^{m+1} h_{k+1,n} P(n,m) A^{m+1} x_0 \\
    & = \sum_{m=0}^k (-1)^m P(k,m) A^m x_0 + \sum_{n=0}^k \sum_{m=1}^{n+1} (-1)^m h_{k+1,n} P(n,m-1) A^m x_0 \\
    & = \sum_{m=0}^k (-1)^m P(k,m) A^m x_0 + \sum_{m=1}^{k+1} \sum_{n=m-1}^{k} (-1)^m h_{k+1,n} P(n,m-1) A^m x_0 \\
    & = x_0 + h_{k+1,k} P(k,k) A^{k+1} x_0 + \sum_{m=1}^{k} (-1)^m \left( P(k,m) + \sum_{n=m-1}^k h_{k+1,n} P(n,m-1) \right) A^m x_0 .
\end{align*}
Observe that $h_{k+1,k} P(k,k) = h_{k+1,k} \left( h_{1,0}h_{2,1} \cdots h_{k,k-1} \right) = h_{1,0}h_{2,1} \cdots h_{k+1,k} = P(k+1,k+1)$.
It remains to show that
\begin{align*}
    P(k,m) + \sum_{n=m-1}^k h_{k+1,n} P(n,m-1) = P(k+1,m)
\end{align*}
for $m=1,\dots,k$.
Rewrite the right hand side:
\begin{align}
    P(k+1,m) & = \sum_{\substack{ i(1)\le j(2), \dots, i(m-1)\le j(m) \\ i(m) \le k+1}} \prod_{\ell=1}^m h_{i(\ell),j(\ell)} \nonumber \\
    & = \sum_{\substack{ i(1)\le j(2), \dots, i(m-1)\le j(m) \\ i(m) \le k}} \prod_{\ell=1}^m h_{i(\ell),j(\ell)} + \sum_{\substack{ i(1)\le j(2), \dots, i(m-1)\le j(m) \\ i(m) = k+1}} \prod_{\ell=1}^m h_{i(\ell),j(\ell)} \nonumber \\
    & = P(k,m) + \sum_{n=1}^k \sum_{\substack{ i(1)\le j(2), \dots, i(m-1)\le j(m) = n}} h_{k+1,n} \prod_{\ell=1}^{m-1} h_{i(\ell),j(\ell)} \nonumber \\
    & = P(k,m) + \sum_{n=1}^k h_{k+1,n} \sum_{\substack{ i(1)\le j(2), \dots, i(m-1)\le j(m) = n}} \prod_{\ell=1}^{m-1} h_{i(\ell),j(\ell)} \nonumber \\
    & = P(k,m) + \sum_{n=m-1}^k h_{k+1,n} \sum_{\substack{ i(1)\le j(2), \dots, i(m-2)\le j(m-1) \\ i(m) \le n}} \prod_{\ell=1}^{m-1} h_{i(\ell),j(\ell)} \label{eqn:appendix-terminal-iterate-P-induction-final-expansion}
\end{align}
where the last equality holds because for $n < m-1$, the inner summation with respect to $i(1)\le j(2), \dots, i(m-2) \le j(m-1)$ is vacuous as it is impossible to choose $i(\ell)$'s satisfying $1\le i(1) < i(2) < \cdots i(m-1) \le n$.
Now replace the inner summation in~\eqref{eqn:appendix-terminal-iterate-P-induction-final-expansion} using the definition of $P(n,m-1)$:
\begin{align*}
    P(k+1,m) = \sum_{n=m-1}^k h_{k+1,n} P(n,m-1) .
\end{align*}
This completes the induction.
\end{proof}

\cref{lemma:polynomial_expression_of_algorithm} explicitly characterizes $x_k$ in terms of matrix polynomial.
Now the proof is done once we show that the explicit expression for $x_N$ is invariant under anti-diagonal transposing, i.e., the replacement $h_{k,j} \mapsto h_{N-j,N-k}$.
More precisely, we have to show that $P(N,m) = \hat{P}(N,m)$, where
\begin{align}
\label{eqn:appendix-terminal-iterate-P-N-hat}
    \hat{P}(N,m) = \sum_{\substack{ i(1)\le j(2), \dots, i(m-1)\le j(m) \\ i(m) \le N}} \prod_{\ell=1}^m h_{N-j(\ell),N-i(\ell)}
\end{align}
so that
\begin{align*}
    \hat{x}_N = \sum_{m=0}^N (-1)^m \hat{P}(N,m) A^m \hat{x}_0 .
\end{align*}
Once this is done, provided that $\hat{x}_0 = x_0$, we can conclude
\begin{align*}
    \hat{x}_N = \sum_{m=0}^N (-1)^m \hat{P}(N,m) A^m \hat{x}_0 = \sum_{m=0}^N (-1)^m P(N,m) A^m x_0 = x_N .
\end{align*}

In the product within~\eqref{eqn:appendix-terminal-iterate-P-N-hat}, we make the substitution $i'(\ell) = N-j(\ell)$ and $j'(\ell) = N-i(\ell)$.
Then for $\ell=1,\dots,m-1$,
\begin{align*}
    i(\ell) \le j(\ell+1) \iff N-j(\ell+1) \le N-i(\ell) \iff i'(\ell+1) \le j'(\ell)
\end{align*}
and the condition $i(m) \le N$ is vacuous (can be dropped) because we are considering $h_{i,j}$ only for $i=1,\dots,N$.
Thus,
\begin{align*}
    \hat{P}(N,m) = \sum_{\substack{ i'(m)\le j'(m-1), \dots, i'(2)\le j'(1)}} \prod_{\ell=1}^m h_{i'(\ell),j'(\ell)} .
\end{align*}
Finally, reversing the order of $(i'(1),j'(1)), \dots, (i'(m),j'(m))$ via substitution $i(\ell) = i'(m+1-\ell), j(\ell) = j'(m+1-\ell)$ we obtain
\begin{align*}
    \hat{P}(N,m) = \sum_{\substack{ i(1)\le j(2), \dots, i(m-1)\le j(m)}} \prod_{\ell=1}^m h_{i(\ell),j(\ell)} = P(N,m)
\end{align*}
which concludes the proof.

\end{proof}

\subsection{Faster convergence of the Dual-Anchor ODE with strongly monotone operators}
\label{subsection:appendix-experiment-early-stopping}

In this section, we show that the Dual-Anchor ODE converges much more rapidly than the primal Anchor ODE for the subclass of strongly monotone operators, which suggests a new potential value of the dual algorithms. 
While the Dual-Anchor ODE can only guarantee $\|\opA (X(T))\| = \cO(\epsilon)$ at the terminal time $T = \Omega(\frac{1}{\epsilon})$ when $\opA$ is merely monotone, when $\opA$ is strongly monotone and Lipschitz, the Dual-Anchor ODE achieves $\|X(t) - X(T)\| = \cO(\epsilon)$ and $\|\opA (X(t))\| = \cO(\epsilon)$ at $t = \cO(\log \frac{1}{\epsilon}) \ll T$ and makes negligible progress thereafter (similar to the pattern shown in \cref{subfig:ouyang_strongly_monotone_performance}). 
This allows one to apply \textit{early stopping}; instead of waiting until the terminal time, we stop and return $X(t)$.
On the other hand, this is not the case for primal Anchor ODE which, even when $\opA$ is strongly monotone, behaves conservatively and converges no faster than the $\|\opA(X(t))\|^2 = \cO(1/t^2)$ rate.
Below we provide the details.

\paragraph{Fast decay of $\norm{\opA(X(t))}$ and $\norm{X(t) - X(T)}$ under strong monotonicity.}
We established in \cref{appendix_lemma:bound_of_dotX} that $\Psi(t) = \frac{1}{(T-t)^2} \norm{\dot{X}(t)}^2$ is nonincreasing by deriving
\begin{align*}
    \dot{\Psi}(t) = -\frac{2}{(T-t)^2} \inner{\frac{d}{dt} \opA(X(t))}{\dot{X}(t)}.
\end{align*}
If $\opA$ is $\mu$-strongly monotone for some $\mu > 0$, then the above identity implies
\begin{align*}
    \dot{\Psi}(t) \le -\frac{2\mu}{(T-t)^2} \norm{\dot{X}(t)}^2 = -2\mu \Psi(t)
\end{align*}
so by Gr\"{o}nwall's inequality,
\begin{align*}
    \Psi(t)
    \le \Psi(0) e^{-2\mu t} = \frac{\norm{\opA(X_0)}^2}{T^2} e^{-2\mu t} 
\end{align*}
for any $t \in (0, T)$.
Reorganizing, we obtain
\begin{align}
\label{eqn:dot_X_decay_in_strongly_monotone_case}
    \norm{ \dot{X}(t) } = \sqrt{(T-t)^2 \Psi(t)} \le \frac{T-t}{T} e^{-\mu t} \norm{\opA(X_0)} \le e^{-\mu t} \norm{\opA(X_0)}.
\end{align}

\paragraph{Complexity analysis.}
Suppose the desired accuracy level is $\|\opA(\cdot)\| \le \epsilon$, so that we choose $T = \Theta(\frac{1}{\epsilon})$ according to the worst-case guarantee $\norm{\opA (X(T))} = \cO\left(\frac{1}{T}\right)$.
Given that $\opA$ is $\mu$-strongly monotone and $L$-Lipschitz, one achieves
\begin{align*}
    \norm{\dot{X}(t)} = \cO \left(\frac{1}{LT^2}\right)  \quad \text{at} \,\, t = \Theta \pr{ \frac{1}{\mu} \log LT } = \Theta \pr{ \frac{1}{\mu} \log \frac{L}{\epsilon} }
\end{align*}
by~\eqref{eqn:dot_X_decay_in_strongly_monotone_case} and then
\begin{align*}
    \norm{X(t) - X(T)} & \le \int_t^T \norm{\dot{X}(s)} \, ds \le \int_t^T \frac{T-s}{T-t} \norm{\dot{X}(t)} \, ds \le T \cO \left( \frac{1}{LT^2} \right) = \cO \left( \frac{1}{LT} \right) = \cO \left( \frac{\epsilon}{L} \right) ,
\end{align*}
where the second inequality uses $\Psi(s)\le \Psi(t)$ for $s \ge t$.
Finally, this implies
\begin{align*}
    \norm{\opA (X(t))} \le \norm{\opA (X(T))} + \cO \left(L \norm{X(t) - X(T)}\right) = \cO \left( \epsilon \right) .
\end{align*}

\paragraph{Anchor ODE is no faster than $\cO\pr{ \frac{1}{t} }$ under strong monotonicity.} 
Suppose $\opA$ is linear ($\opA x = Ax$ for some matrix $A$) and assume $A$ is invertible. In this case the solution to the Anchor ODE can be explicitly characterized \citep{SuhParkRyu2023_continuoustime}:
\begin{align*}
    X_{\text{anchor}}(t) = \frac{1}{t} A^{-1} \pr{ I - e^{-tA} }X_0.
\end{align*}
Now if $A$ is $\mu$-strongly monotone, then $A - \mu I$ is monotone and $0 \in \zer (A - \mu I)$, so we have
\begin{align*}
    & \frac{d}{dt} \norm{ e^{-t(A-\mu I)} X_0 }^2 = -2 \inner{ (A-\mu I) e^{-t(A-\mu I)} X_0  }{ e^{-t(A-\mu I)} X_0 } \le 0 \\
    & \implies \norm{X_0} \ge \norm{ e^{-t(A-\mu I)} X_0 } = \norm{ e^{\mu t I} e^{-tA} X_0 } = e^{\mu t} \norm{ e^{-tA} X_0 } ,
\end{align*}
which implies
\begin{align*}
    \norm{ A X_{\text{anchor}}(t) } = \frac{1}{t} \norm{ \pr{ I - e^{-tA} } X_0 } \ge \frac{1}{t} (1 - e^{-\mu t}) \norm{X_0} .
\end{align*}
Therefore, the convergence of the Anchor ODE is generally not faster than $\cO\left(\frac{1}{t}\right)$ even with strongly monotone (and linear) operators.

\end{document}